\documentclass[1.5pt, a4paper]{article}

 \usepackage{stmaryrd} 
\usepackage[utf8]{inputenc}
\usepackage[T1]{fontenc}
\usepackage{lmodern}
\usepackage{microtype}
\usepackage[pdfborder={0 0 0}]{hyperref}

\usepackage{color}
\usepackage{paralist}
\usepackage{graphicx}
\usepackage[all,2cell]{xy}
\UseAllTwocells
\xyoption{rotate}

\usepackage{verbatim}
\usepackage{tikz,pgfplots}
\usepackage{xcolor}

\usepackage[tbtags]{amsmath} 
\usepackage{amssymb}
\usepackage{amsthm}
\usepackage{mathtools}

\allowdisplaybreaks[1]

\usepackage{marginnote}

\usepackage[text={16.5cm, 24.2cm}, centering]{geometry}
\newcommand{\arrowOut}{
\tikz \draw[-stealth, line width=1.5pt] (-1pt,0) -- (1pt,0);
}

\newcommand{\Z}{\mathbb{Z}}
\newcommand{\C}{\mathbb{C}}
\newcommand{\R}{\mathbb{R}}

\newcommand{\mac}{\mathcal C}

  \newcommand{\maq}{\mathcal Q}

\newcommand{\Hom}{\mathrm{Hom}}
\newcommand{\End}{\mathrm{End}}
\newcommand{\Aut}{\mathrm{Aut}}
\newcommand{\Ob}{\mathrm{Ob}}

\newcommand{\Set}{\mathrm{Set}}
\newcommand{\Grpd}{\mathrm{Grpd}}
\newcommand{\GRPd}{\mathrm{GRPd}}

\newcommand {\repGrpd} {\mathrm{repGrpd}}
\newcommand{\Top}{\mathrm{Top}}
\newcommand {\Lan}[0] {\mathrm{Lan}}
\newcommand {\Vect}[0] {\mathrm{Vect}}
\newcommand{\Graph}[0]{\mathrm{Graph}}

\newcommand{\pack} [0]{\mathrm{Pack}}

\newcommand{\inv}[0]{{-1}}
\newcommand{\oo}[0]{\otimes}
\newcommand{\id}[0]{\mathrm{id}}
\newcommand{\tr}[0]{\mathrm{tr}}

\newcommand{\Fib}{\mathrm{Fib}}

\newcommand{\A}{\mathcal{A}}
\newcommand{\B}{\mathcal{B}}
\newcommand{\G}{\mathcal{G}}
\renewcommand{\H}{\mathcal{H}}

\newcommand{\CAT}{\mathrm{CAT}}
\newcommand{\HOM}{\mathrm{HOM}}
\newcommand{\V}{\mathbf{V}}
\newcommand{\op}{\mathrm{op}}

\newcommand{\T}{\mathcal{T}}

\newcommand{\K}{\mathcal{K}}
\renewcommand{\L}{\mathcal{L}}

\newcommand{\Act}{\mathrm{Act}}

\newcommand{\Mod}{\mathrm{GrMod}}
\newcommand{\PC}{\mathrm{PC}}

\newcommand{\fib}{\mathrm{fib}}

\newtheorem{theorem}{Theorem}[section]

\newtheorem*{theorem*}{Theorem}
\newtheorem{lemma}[theorem]{Lemma}
\newtheorem{proposition}[theorem]{Proposition}
\newtheorem{corollary}[theorem]{Corollary}
\newtheorem{definition}[theorem]{Definition}
\newtheorem{remark}[theorem]{Remark}

\theoremstyle{definition}
\newtheorem{example}[theorem]{Example}

\newcommand*\stdthebibliography{}
\let\stdthebibliography\thebibliography
\renewcommand{\thebibliography}[1]{
\stdthebibliography{#1}\setlength{\itemsep}{-2pt}} 
\usepackage{todonotes}

\def\mytitle{A geometrical description of untwisted 3d Dijkgraaf-Witten TQFT with defects}
\def\myauthors{J.~Faria Martins, C.~Meusburger}
\hypersetup{pdftitle=\mytitle, pdfauthor=\myauthors}

\begin{document}

\begin{center}
  {\LARGE\mytitle}

  \vspace{1em}

  {\large
    Jo\~ao Faria Martins\footnote{{\tt J.FariaMartins@leeds.ac.uk}, orcid \url{https://orcid.org/0000-0001-8113-3646}}\\
    School of Mathematics\\
     University of Leeds\\
      Leeds, LS2 9JT, 
      United Kingdom\\[+2ex]
    Catherine Meusburger\footnote{{\tt catherine.meusburger@math.uni-erlangen.de}, orcid \url{https://orcid.org/0000-0003-3380-3398}}\\
     Department Mathematik \\
  Friedrich-Alexander-Universit\"at Erlangen-N\"urnberg \\
  Cauerstra\ss e 11, 91058 Erlangen, Germany\\[+4ex]
 }

{April 7, 2026}

\end{center}
\begin{abstract}
We give a simple, geometric and explicit construction of 3d untwisted Dijkgraaf-Witten theory with defects of all codimensions. It is given as a symmetric monoidal functor from a defect cobordism category into the category of finite-dimensional complex vector spaces. 
The objects of this category are  oriented stratified surfaces and its morphisms are equivalence classes of stratified   cobordisms, both labelled with higher categorical data. 
This TQFT is constructed in terms of geometric quantities such as fundamental groupoids and bundles  and requires neither state sums  on triangulations  nor diagrammatic calculi for higher categories.
It is obtained from a functor that assigns to each defect surface a representation of a gauge groupoid and to each defect cobordism a fibrant span of groupoids and an intertwiner between the groupoid representations at its boundary.  It is constructed by homotopy theoretic methods and allows for an explicit computation of examples. In particular, we show how the 2d 
part of this defect TQFT gives a simple description of defects of all codimensions in Kitaev's quantum double model.
\end{abstract}

\section{Introduction}

\textbf{Context}\\
Defects in (topological) quantum field theories and the associated categorical symmetries are of strong interest from a  mathematical and  from a mathematical physics perspective.  They arise naturally as domain walls and field insertions in conformal field theories, see for instance Fr\"ohlich, Fuchs, Runkel and Schweigert \cite{FFRS},  Kapustin and Saulina \cite{KS}  and Fuchs, Schweigert, Wood and Yang \cite{FSWY}.
Defects also feature 
 in models from condensed matter physics and topological quantum computing such as Levin-Wen models \cite{LW} and Kitaev's quantum double models \cite{K}, which were identified as the 2d part  of a  topological quantum field theory (TQFT) of Turaev-Viro \cite{TV,BW} type by Balsam and Kirillov \cite{BK}.  The  inclusion of defects into  these models was  achieved by  Kitaev and Kong \cite{KK}, who also identified their algebraic data.

Defects also play an essential role in modern approaches to topological symmetries in quantum field theory pioneered by  Freed, Moore and Teleman \cite{FMT}. 
Similar defect structures are also considered in the framework of (stratified) factorisation homology \cite{AFT} and were applied there by Keller and M\"uller \cite{KM} to describe factorisation homology of surfaces with bundles and by Jordan, Le, Schrader and Shapiro \cite{JLSS} to construct generalisations of quantum decorated character varieties.  

 In addition to these motivations from  topology and mathematical physics  defects in TQFTs are  of interest from a  representation theoretical perspective. They give new insights into  representation theoretical and categorical concepts such as Frobenius-Schur indicators and Brauer-Picard group(oid)s of tensor categories, see for instance Farnsteiner and Schweigert \cite{FS}, Priel \cite{Pr} and Fuchs, Priel, Schweigert and Valentino \cite{FPSV}.

In this article we focus on defects in 3d TQFTs. 
A  $d$-dimensional TQFT  is a symmetric monoidal functor $\mathcal Z:\mathrm{Cob}_d\to \Vect_\C$ from a cobordism category $\mathrm{Cob}_d$, whose objects are closed $(d-1)$-manifolds and whose morphisms  equivalence classes of cobordisms between them.  In contrast,  a  defect TQFT involves manifolds and cobordisms with distinguished marked submanifolds that are labelled with higher categorical data. 
They are encoded in a defect cobordism category $\mathrm{Cob}^{\mathrm{def}}_d$ whose objects and morphisms are \emph{stratified $(d-1)$-manifolds} and \emph{stratified cobordisms} with higher categorical defect data. This makes them interesting also  from a  topological viewpoint. Whereas TQFTs give rise to manifold invariants, defect TQFTs describe  
 manifolds with embedded submanifolds of various dimensions, such as embedded surfaces or knots and links.

Defects in 3d TQFTs were investigated extensively.  Defect lines in Turaev-Viro TQFTs \cite{TV,BW} were first considered by Turaev and Virelizier \cite{TVr, TVr2} and by Balsam and Kirillov \cite{BK0}.  Fuchs, Schweigert and Valentino showed in \cite{FSV0}  how topological defects in 3d TQFTs of Reshetikhin-Turaev \cite{RT} type are related to the Witt group of modular fusion categories.  In \cite{FSV} this  was then specialised to Dijkgraaf-Witten TQFT \cite{DW}.  
Line and surface defects in Reshetikhin-Turaev TQFTs  were treated by Carqueville, Runkel and Schaumann in \cite{CRS} and domain walls in Reshetikhin-Turaev TQFTs by
Koppen, Mulevi\v{c}ius, Runkel and Schweigert in \cite{KMRS}.  

An approach to 3d defect TQFTs via diagrammatic calculi was initiated by Carqueville, Meusburger and Schaumann in \cite{CMS} based on the Gray category diagrams developed by Barrett, Meusburger and Schaumann in \cite{BMS}. These diagrams also feature  in the orbifold approach to defect TQFTs by Carqueville, Runkel and Schaumann \cite{CRS2, CRS3} and in  \cite{CMR+}.
State sum models for Turaev-Viro TQFTs \cite{TV,BW} with defects of all codimensions {were constructed} by Meusburger in \cite{M} and in more generality by Carqueville and M\"uller \cite{CM}   in the orbifold approach.

 \textbf{Motivation}\\
Despite the number of results on defect TQFTs in the literature, their applications to gauge  theory and  to geometrical and topological questions are hindered by the fact that the algebraic data is quite involved. They are usually constructed with advanced categorical methods and rely heavily on diagrammatic calculi for higher categories. The resulting formulation is often rather abstract and implicit  and  involves auxiliary choices such as triangulations or cell decompositions.
This makes it difficult to interpret defect TQFTs in  terms of
  gauge theoretical or geometric  concepts such as bundles, even in those cases where the underlying TQFT has a clear geometric interpretation.

 \textbf{Dijkgraaf-Witten TQFT with defects}\\
In this article we address this problem for the simplest example of a 3d defect TQFT of Turaev-Viro-Barrett-Westbury type, namely  untwisted Dijkgraaf-Witten TQFT \cite{DW}. The theory without 
defects admits a direct geometrical construction in terms of  $G$-bundles for finite groups $G$.  In the language of Quinn's finite total homotopy TQFT \cite[Section 4]{Q}, it describes the homotopy content of function spaces and groupoid cardinality~\cite{FMP1,FMP,BHW}.
As it is directly related to Kitaev's quantum double model \cite{K} for the group algebra $\C[G]$ and the associated Levin-Wen models, the associated defect TQFT can be applied to describe  defects in these models.

The categorical data labelling  defects in 3d Dijkgraaf-Witten TQFT is largely understood and can be derived from the general defect data for Turaev-Viro-Barrett-Westbury TQFTs by specialising to spherical fusion categories $\mathrm{Vec}_G^\omega$ of $G$-graded vector spaces with a 3-cocycle $\omega: G^{\times 3} \to\C^\times$. This leads to the following labelling of the strata of  defect surfaces and defect cobordisms  with algebraic defect data
\begin{compactitem}
\item $\mathrm{codim}$ 0-strata: finite groups (and 3-cocycles),
\item $\mathrm{codim}$ 1-strata:  finite sets with group actions  (and 2-cocycles),
\item $\mathrm{codim}$ 2-strata: (projective) representations of action groupoids for group actions,
\item $\mathrm{codim}$ 3-strata:  intertwiners between them.
\end{compactitem}
In the following, we do not consider the most general data. We require all cocycles to be trivial, which implies that the strata of codimension 2 and 3 are simply labelled with representations of action groupoids and intertwiners between them. 
We expect that our methods can be extended to the case with non-trivial cocycles, but this will be technically more involved, and we leave it for future work.

 \textbf{Main results}\\
Our main result is the explicit  construction of a Dijkgraaf-Witten  TQFT with defects -  a symmetric monoidal functor $\mathcal Z:\mathrm{Cob}_3^{\mathrm{def}}\to \Vect_\C$ from a
 cobordism category  of stratified  surfaces and  stratified cobordisms with defect data  into the category of vector spaces (Theorem \ref{cor:limtot}).
 
 The construction does not use  diagrammatic calculi for higher categories and does not require choices of triangulations or cell decompositions,  although we work in the PL framework.  It also
 does not rely on  previous results on  Turaev-Viro TQFTs with defects  such as the ones in \cite{M,CM} and thus  gives an independent construction of  defect Dijkgraaf-Witten TQFT.  As in Morton \cite{Mo}, which motivated some of our constructions,  our defect TQFT 
  is   formulated entirely in terms of elementary geometrical and topological
  quantities such as bundles,  fundamental groupoids,  action groupoids for group actions on sets  and representations of the latter. 
  
 While most of the results are obtained by homotopy theoretic methods, they can be expressed in a simple way and have a clear geometric interpretation. In particular, the  formalism  allows one to efficiently and straightforwardly compute the vector spaces associated to  stratified surfaces and the linear maps assigned to cobordisms between them. This is rather difficult in other formalisms, and we illustrate this with a number of examples. Although choices of triangulations or cell decompositions are not required, we show that 
 making such choices results in a simple graph gauge theoretical formulation of the defect TQFT.
 This  allows in particular  a direct application to quantum double models with defects.
  We hope that our formalism can help to bridge a gap  between more abstract approaches and concrete geometric applications.

 \textbf{Construction of the defect TQFT}\\ 
The  TQFT is constructed in two steps.
 The first  step  involves only the defect data associated to strata of codimensions 0 and 1, referred to as \emph{classical defect data}. 
 It constructs with this data for each stratified surface or cobordism $X$ a gauge groupoid $\A^X\sslash \G^X$, an action groupoid with gauge configurations as objects and gauge transformations as morphisms.
We  show (Theorem \ref{th:classcobfunc}) that this assignment defines a symmetric monoidal functor $C:\mathrm{Cob}_3^{\mathrm{def}}\to\mathrm{span}^{\mathrm{fib}}(\Grpd)$ into the category of fibrant spans of essentially finite groupoids.  The fibrancy of these spans has a direct gauge theoretical interpretation, namely that gauge transformations at the boundary extend to the bulk.
 
Concretely, these gauge groupoids are obtained as follows. We encode the combinatorics of the stratification of $X$ in a graded graph $Q^X$, whose vertices represent the strata and whose edges describe local incidence of lower-dimensional strata to higher-dimensional ones. We also consider the category $\maq^X$ freely generated by this graph. 
The classical defect data then defines a functor $D^X:\maq^X\to \Grpd$ that assigns to each stratum a groupoid constructed from the classical defect data. 
By associating  to each stratum $s$ of $X$ a {topological space} $\hat s$
constructed from $s$,
 we obtain another functor $T^X:\maq^X\to \Grpd$  that describes the topology of the stratification and assigns to each stratum $s$ the fundamental groupoid of $\hat s$.  The gauge groupoid $\A^X\sslash \G^X$ is then defined as the end of the functor $\HOM_\Grpd (T^X,D^X):\maq^{Xop}\times\maq^X\to \Grpd$.

 The second step in the construction also takes into account the defect data assigned to strata of codimensions 2 and 3, in the following called \emph{quantum defect data}. With this data, we assign to each stratified surface $\Sigma$ a representation $F_\Sigma:\A^\Sigma\sslash \G^\Sigma\to \Vect_\C$ of its action groupoid. 
 
For each
 stratified cobordism $\Sigma_0\xrightarrow{i_0} M\xleftarrow{i_1} \Sigma_1$ with associated  span $\A^{\Sigma_0}\sslash G^{\Sigma_0}\xleftarrow{P_0} \A^M\sslash \G^M \xrightarrow{P_1} \A^{\Sigma_1}\sslash \G^{\Sigma_1}$
 the defect data associated to its 1-skeleton  defines a natural transformation $\mu^M:F_{\Sigma_0}P_0\Rightarrow F_{\Sigma_1}P_1$.
We show (Proposition \ref{prop:deffunc}) that these assignments define a symmetric monoidal functor $G:\mathrm{Cob}_3^{\mathrm{def}}\to \mathrm{span}^{\mathrm{fib}}(\repGrpd)$, where $\mathrm{span}^{\mathrm{fib}}(\repGrpd)$ is the
symmetric monoidal category whose objects are representations of essentially finite groupoids and whose morphisms are given by fibrant spans together with natural transformations. 

By combining this functor with a functor $L:\mathrm{span}^{\mathrm{fib}}(\repGrpd)\to \Vect_\C$ that is analogous to the functors constructed by Haugseng \cite{Ha}, Trova \cite{T} and Schweigert and Woike \cite{SW1} and given by the limit on the objects, we obtain the defect TQFT $\mathcal Z=LG:\mathrm{Cob}_3^{\mathrm{def}}\to \Vect_\C$ in Theorem \ref{cor:limtot}.

\smallskip
 \textbf{Applications and examples}\\
The first application of our results is in the context of Kitaev's quantum double models \cite{K}. 
As shown in \cite{BK},  Dijkgraaf-Witten TQFT for a finite group $G$  is directly related to Kitaev's quantum double \cite{K} and the associated Levin-Wen models \cite{LW} for the group algebra $\C[G]$. More specifically, this TQFT assigns to an oriented surface $\Sigma$ 
the \emph{ground state} or \emph{protected space} of the quantum double model. It is also shown in \cite{BK} that codim 2 defects describe excitations in these models. The full defect data was identified in \cite{KK} in the more general context of Levin-Wen models for  unitary spherical fusion categories.

Our  defect TQFT  assigns to each stratified surface with defect data a vector space that can be viewed as  generalised ground state of these models in the presence of defects and coincides with it, if the defect data is chosen to be trivial. It therefore provides a canonical definition of a generalised ground state. It also gives an efficient way to compute it that does not require choices of triangulations or cell-decompositions. 
 In Section \ref{subse:qudoublemods}  we compute this vector space explicitly for a number of examples and  show how various defect constellations (excitations, domain walls, domain wall excitations) arise from our formalism via the specific choices of defect data and stratifications.

In Section \ref{subsec:defcobs}   we treat  a number of 3d examples.  We  derive a general formula for the number our defect TQFT assigns to a closed stratified 3-manifold with defect data and also
 consider cobordisms with boundaries.
 In all cases, the results are given in terms of simple geometrical quantities such as group homomorphisms from fundamental groups into finite groups, fixed points of group actions or representations of certain stabiliser groups. Choosing  transparent defect data reduces the defect TQFT  to the one without defects. 

 \textbf{Structure of the article}\\
Section \ref{sec:grpds} is purely algebraic. It contains the required background on groupoids in Section \ref{sec:groupoidbasic}, describes the categories $\mathrm{span}^{\mathrm{fib}}(\Grpd)$ and $\mathrm{span}^{\mathrm{fib}}(\repGrpd)$ in Section \ref{sec:fibspangrpd} and constructs  the symmetric monoidal functor $L:\mathrm{span}^{\mathrm{fib}}(\repGrpd)\to\Vect_\C$ used in the definition of the defect TQFT in Section \ref{subsec:quantumquinn}. 

In Section \ref{sec:stratgraph} we introduce the stratifications required in the definition of the defect cobordism category and their description by graded graphs. These stratifications are oriented stratifications of PL manifolds, subject to certain conditions on the
neighbourhoods of strata. We introduce a notion of local strata that describes the incidence of lower-dimensional strata to higher-dimensional ones. We then show how stratifications can be encoded in graded graphs.
Section \ref{sec:cobcat} then introduces the defect data for untwisted Dijkgraaf-Witten theory and defines the associated defect cobordism category. 

Section \ref{sec:gauge} introduces the gauge groupoids of stratified surfaces and cobordisms. In Section \ref{sec:fundgrpd} we construct  the functor $T^X:\maq^X\to\Grpd$ that describes the fundamental groupoids of the PL manifolds associated to the strata of $X$.  Section \ref{subsec:pack} introduces an auxiliary category whose objects describe the functors $D^X,T^X:\maq^X\to \Grpd$ associated to a stratification, and whose morphisms are induced by certain  graph maps and natural transformations. In Section \ref{subsec:gaugegrpd} we then define the gauge groupoid of a stratified surface or cobordism. 
We give an explicit description of gauge configurations  in terms of functors from  fundamental groupoids assigned to  strata into groupoids constructed from their defect data. Gauge transformations are  natural transformations between such functors that respect the maps relating  different strata.  
We show that gauge configurations and gauge transformations are determined uniquely by the functors and natural transformations for strata of codimension $\leq 1$ and $0$, respectively. 

In Section \ref{subsec:redgaugegrpd} we give a simplified description of the gauge groupoid in terms of fundamental groupoids with basepoints. 
In Section \ref{subsec:geomdescript} we show that this can be simplified further for certain stratifications, such as triangulations or their duals. The theory then becomes a generalised graph gauge theory for a graph $\Gamma$ that is   the 1-skeleton of the Poincar\'e  dual to the stratification. The gauge groupoid  is then  given as the action groupoid
$
\mathcal A\sslash \mathcal G=\left(\Pi_{e\in E_\Gamma} M_e\right)\sslash  \left(\Pi_{v\in V_\Gamma} G_v\right)
$,
where $M_e$ is the set assigned to the codim 1-stratum dual to an edge $e\in E_\Gamma$ and $G_v$ the group assigned to the codim 0-stratum dual to a vertex  $v\in V_\Gamma$.

Section \ref{sec:defect} constructs the defect TQFT  $\mathcal Z:\mathrm{Cob}_3^{\mathrm{def}}\to\Vect_\C$. In Section \ref{subsec:classdeftqft} we show that the gauge groupoids of stratified surfaces and cobordisms define a symmetric monoidal  functor $C:\mathrm{Cob}_3^{\mathrm{def}}\to \mathrm{span}^{\mathrm{fib}}(\Grpd)$. After proving some auxiliary results in Section \ref{subsec:edgeproj}, we then show in Section \ref{subsec:quantumtqft} that the full defect data of the theory promotes it  to a symmetric monoidal  functor  $G:\mathrm{Cob}_3^{\mathrm{def}}\to \mathrm{span}^{\mathrm{fib}}(\repGrpd)$ and that composition with the functor   $L:\mathrm{span}^{\mathrm{fib}}(\repGrpd)\to \Vect_\C$ from Section \ref{subsec:quantumquinn} yields a defect TQFT. We show that it  reduces to the usual Dijkgraaf-Witten TQFT in its formulation \`a la Quinn \cite{Q} in the absence of defects. 

Section \ref{sec:examples} contains  examples that illustrate the formalism. After some background on the defect data in Section \ref{sec:defexamples}, we give a detailed description of the vector spaces assigned to stratified surfaces in Section \ref{subse:qudoublemods} and relate them to defects in quantum double models. Section \ref{subsec:defcobs} contains examples of defect cobordisms.  The appendix contains some purely categorical results that are used in the proof of Theorem \ref{th:classcobfunc}.

The reader mainly interested in applications and examples may focus on Sections \ref{subsec:redgaugegrpd}, \ref{subsec:geomdescript} and \ref{sec:examples}.

\section{Spans of groupoids and groupoid representations}
\label{sec:grpds}

\subsection{Groupoids and groupoid representations}
\label{sec:groupoidbasic}

Recall that the category $\Grpd$  of groupoids is complete and cocomplete, that it is a cartesian  monoidal category and that it 
 is cartesian closed \cite[Section 1.1]{Ke}. The exponential of  a groupoid $\mathcal H$ by a groupoid $\mathcal G$ is the groupoid $\mathcal H^{\mathcal G}=\GRPd(\G,\H)$ of functors from $\mathcal G$ to $\mathcal H$ and natural transformations between them. Given groupoids $\mathcal{G},\mathcal{H}$ and $\mathcal{K}$, we have a canonical isomorphism of groupoids, natural in $\mathcal{G},\mathcal{H}$ and $\mathcal{K}$,
 \begin{equation}\label{eq:tensor-hom-grpd}
 \GRPd(\mathcal G\times\mathcal H,\mathcal K)\cong \GRPd\big(\mathcal G, {\K^\H}\big).
 \end{equation}

We have a functor $\Grpd^{op}\times \Grpd \to \Grpd$ that sends a pair $(\H,\K)$ to $\K^\H$.   For a functor $F: \G \to \H$ of groupoids  we denote by $F^*:=\GRPd(F,\id_\K) : \K^{\H}\to \K^\G$ the functor that precomposes with $F$. More generally,
we have a composition functor $\GRPd(\G,\H) \times \GRPd(\H,\K) \to \GRPd(\G,\K)$
that corresponds to a functor
$M: \GRPd(\G,\H) \to  \GRPd(\G,\K)^{ \GRPd(\H,\K)}$ 
under the isomorphism in  \eqref{eq:tensor-hom-grpd}.  This functor $M$ sends a functor $F:\G\to H$ to $M(F)=F^*$ and a natural transformation  $\eta: F_1 \Rightarrow F_2$ to the natural transformation  $\eta^*:=M(\eta): F_1^* \Rightarrow F_2^*$ with components  $\eta^*_f=f\eta  : f F_1 \Rightarrow f F_2$ for objects $f: \H \to \K$ in $\GRPd(\H,\K)$.

The category $\Grpd$ has a model category structure.  A functor $F:\mathcal E\to\mathcal B$ is  a  \textbf{fibration} if for every morphism $h: B'\to F(E)$ in $\mathcal B$  with $E \in \Ob\mathcal E$,  there is a morphism $g: E'\to E$ with $F(g)=h$. It is  a  \textbf{cofibration} if it is injective on the objects. It is a  \textbf{weak equivalence} if it is an equivalence of groupoids, see for instance \cite{Hollander}. 

We denote by $\bullet$ the terminal  and by $\emptyset$ the initial object in the category $\Grpd$. The  \textbf{interval groupoid}  is the groupoid $I$ with objects 0,1 and unique morphism  $d\colon 0\to 1$. 
The   \textbf{arrow groupoid} $\mathcal G^I$ for a groupoid $\mathcal G$ has as objects morphisms  in $\mathcal G$. Morphisms  in $\mathcal G^I$ from $f: G_0\to G_1$ to 
$f': G'_0\to G'_1$ are pairs $(f_0,f_1)$ of morphisms $f_0: G_0\to G'_0$ and $f_1: G_1\to G'_1$ with 
$f'\circ f_0=f_1\circ f$.
There are  functors $p_0, p_1: \mathcal G^I\to\mathcal G$ given by $p_i(f)=G_i$ and $p_i(f_0,f_1)=f_i$. 
The groupoid  morphism $\langle p_0,p_1 \rangle\colon \mathcal{G}^I \to \mathcal{G}\times \mathcal{G}$ induced by the universal property of the product is a fibration of groupoids.

For each groupoid $\mathcal G$ there are functors $\iota^0=(\id_\mathcal G, 0),\iota^1=(\id_\mathcal G, 1): \mathcal G\to \mathcal G\times I$ and a natural transformation
$\iota^d=(\id_{\mathcal G}, d): \iota^0\Rightarrow\iota^1$.
Functors $F: \mathcal G\times I\to \mathcal H$ correspond to  natural transformations $F^d: F^0\Rightarrow F^1$ between functors $F^0=F\iota^0,F^1=F\iota^1: \mathcal G\to\mathcal H$.

For a groupoid $\mathcal B$ the slice category $\Grpd_{\mathcal B}=\Grpd/\mathcal B$ of \textbf{groupoids over $\mathcal B$} has as objects  morphisms $p: \mathcal{G} \to \mathcal B$ in $\Grpd$.
Morphisms from $p: \mathcal{G} \to \mathcal B$ to $p': \mathcal{G}' \to \mathcal B$ are functors $f: \mathcal{G}\to \mathcal{G}'$ with $p'f=p$, called  \textbf{functors over $\mathcal B$} and denoted $f:p\to_{\mathcal B} p'$. A  \textbf{natural transformation $\eta: f \Rightarrow_{\mathcal{B}} f'$ over $\mathcal B$}  from $f:p\to_{\mathcal B} p'$ to $f':p\to_{\mathcal B} p'$  is a natural transformation $\eta: f\Rightarrow f'$ with $p'(\eta_G)=\id_{p(G)}$ for all $G\in \Ob\mathcal G$.

For any groupoid $\mathcal B$ the category $\Grpd_{\mathcal B}$  has an induced model category structure, where a functor over $\mathcal B$ is a fibration, cofibration or weak equivalence if and only if the underlying 
functor between
groupoids is, see \cite[Theorem 7.6.5]{Hirschhorn}. In particular, an object $p:\mathcal{G} \to \mathcal B$ of $\Grpd_{\mathcal B}$ is fibrant if and only if the functor $p:  \mathcal{G} \to \mathcal B$ is a fibration, and any groupoid over $\mathcal B$ is cofibrant in $\Grpd_{\mathcal B}$.  

\begin{lemma}\label{lem:doldgrp} Let $p:\mathcal G\to\mathcal B$ and $p':\mathcal G'\to\mathcal B$ be fibrations in $\Grpd$. Then any functor $f:p\to_{\mathcal B} p'$  of groupoids over $\mathcal B$ that is an equivalence of groupoids is an equivalence in $\Grpd_{\mathcal B}$.  
\end{lemma}
\begin{proof} 
By the Whitehead lemma \cite[Theorem 7.5.10]{Hirschhorn}  the functor $f: p\to_{\mathcal B} p'$ over $\mathcal B$  
is a homotopy equivalence of groupoids over $\mathcal B$. Now observe that for any  object $p: \mathcal{G} \to \mathcal B$ of $\Grpd_{\mathcal B}$, a cylinder object is the groupoid $\mathcal{G}\times I$, together with the functor 
$p\pi_1: \mathcal{G}\times I  \to \mathcal B$, where $\pi_1 : \mathcal{G}\times I \to\mathcal{G}$ is {the projection}. This allows one to pass from a homotopy equivalence in $\Grpd_{\mathcal B}$ to an equivalence of groupoids over $\mathcal B$.  
\end{proof}

We denote by $\Vect_\C$ the category of finite-dimensional vector spaces over $\C$. A functor $\rho:\mathcal G\to\Vect_\C$ from a groupoid $\mathcal G$ is called a  \textbf{representation} of $\mathcal G$ and a natural transformation $\eta: \rho\Rightarrow \rho'$ between such functors an  \textbf{intertwiner}. 
For a representation $\rho: \mathcal G\to \Vect_\C$ we denote by $\rho^*:\mathcal G\to \Vect_\C$ the  \textbf{dual representation} given by $\rho^*(G)=\rho(G)^*$ for $G\in\Ob\mathcal G$  and  $\rho^*(f)=\rho(f^\inv)^*: \rho^*(G)\to \rho^*(G')$ on the morphisms $f: G\to G'$. The  \textbf{trivial representation} of a groupoid $\mathcal G$ is given by the constant functor $\C:\mathcal G\to\Vect_\C$ that assigns to each object the vector space $\C$ and to each morphism the identity map $\id_\C$.

For any groupoid $\mathcal G$ the category $\Vect_\C^{\mathcal G}$ of representations of $\mathcal G$ and intertwiners between them inherits a symmetric monoidal structure from $\Vect_\C$, with the pointwise tensor product. The dual object of a representation $\rho: \mathcal G\to \Vect_\C$ is the dual representation $\rho^*:\mathcal G\to \Vect_\C$ and the evaluation and coevaluation are induced 
by the evaluation and coevaluation of $\Vect_\C$. The category $\Vect_\C^{\mathcal G}$ also inherits a canonical pivotal structure from $\Vect_\C$ and thus becomes a spherical category. 

The  \textbf{action groupoid} for an action $\rhd: G\times M\to M$ of a group $G$ on a set $M$ is denoted $M\sslash G$ and has
 \begin{compactitem}
 \item as objects  elements of the set $M$,
 \item as morphisms from $m\in M$ to $m'\in M$  group elements $g\in G$ with $g\rhd m=m'$.
 \end{compactitem}

 For a
 $G\times H^{op}$-set $M$  we denote by  $\rhd: (G\times H)\times M\to M$, $(g,h)\rhd m=g\rhd m\lhd h^\inv$
 the associated left action of  $G\times H$ on $M$ with action groupoid $M\sslash G\times H$.

Functors between action groupoids can be constructed from group homomorphisms and maps between sets that satisfy an equivariance condition with respect to the group actions and group homomorphisms. Such functors can be described more formally in terms of the following category.
 
\begin{definition}\label{def:mod}  The category $\Mod$ has 
\begin{compactitem}
\item as objects pairs $(G,M)$ of a group $G$ and a $G$-set $M$,
\item as morphisms $(\phi,f):(G,M)\to (G',M')$ pairs of a group homomorphism $\phi: G \to G'$ and a map $f: M\to M'$ such that 
$f(g\rhd m)=\phi(g)\rhd f(m)$ for all $g\in G$ and $m\in M$.
\end{compactitem}
\end{definition} 

\begin{lemma} \label{lem:actfunctset}
There is a functor $\Act: \Mod \to \Grpd$ that sends 
\begin{compactitem}
\item an object $(G,M)$ to the action groupoid $M\sslash G$,
\item a morphism $(\phi,f): (G,M)\to (G', M')$ to the functor $F: M\sslash G\to M'\sslash G'$ given by $f: M\to M'$ on the objects and by $\phi: G\to G'$ on the morphisms. 
\end{compactitem}
\end{lemma}

\subsection{Fibrant spans of groupoids and spans of groupoid representations}
\label{sec:fibspangrpd}

In this section we assemble the background and some relevant facts on spans of groupoids and groupoid representations. Recall first the definition of a span in a category $\mac$ and of the associated category of spans in a category $\mathcal C$ with pullbacks.

\begin{definition}\label{Def:span} Let $\mac$ be a category.
A \textbf{span}  in  $\mac$ from $X_1\in\Ob\mac$ to $X_2\in\Ob\mac$ is a diagram $X_1 \xleftarrow{P_1} A \xrightarrow{P_2} X_2$ in $\mac$. Two spans $X_1 \xleftarrow{P_1} A \xrightarrow{P_2} X_2$ and $X_1 \xleftarrow{P'_1} A' \xrightarrow{P'_2} X_2$ from $X_1$ to $X_2$ are called equivalent, and we write
 $$(X_1 \xleftarrow{P_1} A \xrightarrow{P_2} X_2) \sim (X_1 \xleftarrow{P'_1} A' \xrightarrow{P'_2} X_2),$$
if there is an isomorphism $f: A\to A'$ such that the following diagram commutes
$$
\begin{gathered}
\xymatrix@R=4pt{
 & A  \ar[dd]_\cong^{f}\\
X_1  \ar@{<-}[ru]^{P_1 } \ar@{<-}[rd]_{P'_1}  & & X_2 \ar@{<-}[lu]_{P_2} \ar@{<-}[ld]^{P'_2}\\
& A'}
 \end{gathered}.
 $$
 \end{definition}

\begin{lemma}
If $\mac$ has pullbacks then there is a category $\mathrm{span}(\mac)$ with the same objects as $\mac$ and with equivalence classes of spans from $X_0$ to $X_1$ as morphisms from $X_0$ to $X_1$. The composition is via pullback
$$ [ X_2\xleftarrow{Q_1}   B \xrightarrow{Q_2}  X_3]\circ [ X_1\xleftarrow{P_1}  A\xrightarrow{P_2}  X_2]=[ X_1\xleftarrow{P_1T_1} 
 A\times_{ X_2}  B \xrightarrow{Q_2 T_2}  X_3 ],$$
where $T_1:  A\times_{ X_2} {B} \to A$ and 
$T_2:  A\times_{ X_2} {B} \to B$
denote pullback projections. 
 The identity on an object is given by the equivalence class of the identity span: $$\id_A=[A \xleftarrow{\id_A} A \xrightarrow{\id_A} A].$$
 \end{lemma}
 Dually, we can define cospans in a category $\mac$ and the category $\mathrm{cospan}(\mac)$, if $\mac$ is a category with pushouts.

In the following, it will be essential 
 the spans under consideration
are \textbf{fibrant}. The category $\Lambda=(0\leftarrow 1 \rightarrow 0')$ is an inverse  category in the sense of \cite[\S 5]{Hovey}. Consequently, the category $\Grpd^\Lambda$, whose objects are spans of groupoids,  
 can be provided with the  \emph{injective}
model category structure, using the nomenclature of \textit{loc cit}.  A fibrant span is  a fibrant object with respect to this model category structure. More concretely, this is stated as follows.

\begin{definition}
A \textbf{span} of groupoids $\mathcal G_1\xleftarrow{F_1} \mathcal G\xrightarrow{F_2}\mathcal G_2$  is called \textbf{fibrant}, if  the 
induced functor  $\langle F_1,F_2\rangle: \mathcal G\to \mathcal G_1\times\mathcal G_2$ is a fibration. 
\end{definition}

We will only consider spans of essentially finite groupoids.  A groupoid $\mathcal G$ is called \textbf{essentially finite}, if it is equivalent to a \textbf{finite} groupoid, a groupoid with finitely many objects and morphisms.
In an essentially finite groupoid there are only finitely many isomorphism classes of objects and all Hom-sets are finite.

Fibrant spans of essentially finite groupoids also 
form a category, which we will see later is symmetric monoidal.
However, its identity morphisms are not given by the identity spans that define the identity morphisms in  
$\mathrm{span}(\Grpd)$. Those  spans are in general not fibrant. 
Instead, the identity morphism on a groupoid $\G$ is defined by  the arrow groupoid
 $\G^I$  and its projection functors $p_0,p_1:\G^I\to \G$ from Section \ref{sec:groupoidbasic}. 

\pagebreak

\begin{lemma} \label{def:spangrpd}We have a  symmetric monoidal category $\mathrm{span}^{\rm fib}(\Grpd)$,  which is given as follows: 
\begin{compactitem}
\item   Objects are essentially finite groupoids $\mathcal G$.
\item   Morphisms from $\mathcal G_1$ to $\mathcal G_2$   are equivalence classes of fibrant spans $\mathcal G_1\xleftarrow{P_1}  \mathcal G\xrightarrow{P_2} \mathcal G_2$ of essentially finite groupoids. 

Two spans 
are equivalent, 
$(\mathcal G_1\xleftarrow{P_1} \mathcal  G\xrightarrow{P_2} \mathcal G_2) \sim_\fib (\mathcal G_1\xleftarrow{P'_1} \mathcal  G'\xrightarrow{P'_2} \mathcal G_2)$,
 if there are functors $F: \mathcal G\to\mathcal G'$ and $F':\mathcal G'\to \mathcal G$ with $P'_i F=P_i$ and $P_iF'=P'_i$ that form an equivalence over $\mathcal G_1\times \mathcal G_2$:  there are natural isomorphisms $\eta: F'F\Rightarrow \id_{\mathcal G}$ and $\eta': FF'\Rightarrow \id_{\mathcal G'}$ with $P_i \eta=\id_{P_i}$ and $P'_i\eta'=\id_{P'_i}$.

\item Composition of morphisms is  by  pullback, where $T_1: \mathcal H\times_{\mathcal G_2} \mathcal{K} \to \mathcal{H}$ and $T_2: \mathcal H\times_{\mathcal G_2} \mathcal{K} \to \mathcal{K}$ denote pullback projections: 
$$ [\mathcal G_2\xleftarrow{Q_1}  \mathcal K\xrightarrow{Q_2} \mathcal G_3]\circ [\mathcal G_1\xleftarrow{P_1} \mathcal H\xrightarrow{P_2} \mathcal G_2]=[\mathcal G_1\xleftarrow{P_1T_1} 
\mathcal H\times_{\mathcal G_2} \mathcal K \xrightarrow{Q_2 T_2} \mathcal G_3 ].$$

\item  The identity morphism on $\mathcal G$ is $[\mathcal G\xleftarrow{p_0} \mathcal G^I\xrightarrow{p_1} \mathcal G]$.

\item The symmetric monoidal structure is  given by  the product of groupoids.  The associativity and unit constraints and the symmetric braiding of $\mathrm{span}^{\rm fib}(\Grpd)$ are obtained from those of $\Grpd$, as an isomorphism $ f: \mathcal{G} \to \mathcal{G}'$ of groupoids defines a fibrant span  $$[\mathcal G\xleftarrow{p_0} \mathcal G^I\xrightarrow{f\,p_1} \mathcal{G}'].$$
\end{compactitem}
\end{lemma}
\begin{proof} This is proven in \cite[Section 3.2]{Tor}, in the dual context of cofibrant spans of topological spaces and in \cite[Section 3.2] {FMP} in the context of homotopy finite spans of homotopy finite spaces. The argument generalises 
to the context of groupoids via Lemma \ref{lem:doldgrp}. Note that the fibrancy condition on spans is essential to ensure that the composition of fibrant spans of essentially finite groupoids is a  fibrant span of essentially finite groupoids. It is also required to ensure  that $[\mathcal G\xleftarrow{p_0} \mathcal G^I\xrightarrow{p_1} \mathcal G]$ indeed acts as an identity  and  that the associativity, unit constraints and  braiding are natural isomorphisms.
\end{proof}

We  now consider a category that is closely related to $\mathrm{span}^{\mathrm{fib}}(\Grpd)$, where the objects are representations of groupoids and morphisms are given by fibrant spans of groupoids, together with certain natural transformations. Identity morphism on a groupoid $\G$ are again defined  by  the arrow groupoid
 $\G^I$  and its projection functors $p_0,p_1:\G^I\to \G$. 
For this, note  that any  representation $\rho: \G \to \Vect_\C$ defines functors $\rho p_0,\rho p_1: \G^I\to \Vect_\C$ 
and  a natural transformation  
$\sigma^\rho:\rho p_0\Rightarrow \rho p_1$ with components $\sigma^\rho_f=\rho(f)$ for all objects $f$ of $\G^I$.

\begin{lemma}\label{def:spanrepgrpd}
There is a symmetric monoidal category $\mathrm{span}^{\fib}(\repGrpd)$,   given as follows:
\begin{compactitem}
\item Objects are pairs $(\mathcal G,\rho)$ of an essentially finite  groupoid $\mathcal G$ and a functor $\rho:\mathcal G\to\mathrm{Vect}_\C$.
\item  Morphisms from $(\mathcal G_1,\rho_1)$ to $(\mathcal G_2,\rho_2)$ are  equivalence classes of pairs $(\mathcal G_1\xleftarrow{P_1} \mathcal G\xrightarrow{P_2} \mathcal G_2,\sigma)$ of a fibrant span $\mathcal G_1\xleftarrow{P_1} \mathcal G\xrightarrow{P_2} \mathcal G_2$ of essentially finite groupoids and a natural transformation $\sigma: \rho_1 P_1\Rightarrow \rho_2P_2$. Two such pairs are equivalent $$(\mathcal G_1\xleftarrow{P_1} \mathcal G\xrightarrow{P_2} \mathcal G_2,\sigma)\sim_{\fib} (\mathcal G_1\xleftarrow{P'_1} \mathcal G'\xrightarrow{P'_2} \mathcal G_2,\sigma'),$$ if there are functors $F:\mathcal G\to\mathcal G'$ and $F':\mathcal G'\to \mathcal G$  with $P'_i F=P_i$, $P_iF'_i=P'_i$,  that form an equivalence over $\mathcal G_1\times\mathcal G_2$ and satisfy $\sigma' F=\sigma$ and $\sigma F'=\sigma'$. 

\item   Composition of morphisms  given by  pullback and composition of natural transformations
$$ [(\mathcal G_2\xleftarrow{Q_1}\mathcal K\xrightarrow{Q_2} \mathcal G_3,\tau )]\circ [(\mathcal G_1\xleftarrow{P_1}\mathcal H\xrightarrow{P_2} \mathcal G_2, \sigma)]=\big[\big(\mathcal G_1\xleftarrow{P_1T_1} \mathcal H\times_{\mathcal G_2} \mathcal K\xrightarrow{Q_2T_2} \mathcal G_3, (\tau T_2 )\circ (\sigma T_1)\big )\big].$$ 

\item The identity morphism on $(\mathcal G,\rho)$ is  $[(\mathcal G\xleftarrow{p_0}\mathcal G^I\xrightarrow{p_1} \mathcal G, \sigma^\rho)]$ with 
$\sigma^\rho_f=\rho(f)$ for all morphisms $f$ in $\mathcal G$.

\item The symmetric monoidal structure is given by the product of groupoids and  the tensor product in $\Vect_\C$. 
\end{compactitem}
\end{lemma}
\begin{proof} That the composition  descends to the quotient and is associative is proven with no difficulty. So we will focus on the unit axiom, which is not straightforward.

1.~Let  $\mathcal G_1\xleftarrow{q_1} \mathcal K\xrightarrow{q_2} \mathcal G_2$ be a fibrant span of essentially finite groupoids, $\rho_1: \G_1 \to \Vect_\C$ and $\rho_2: \G_2 \to \Vect_\C$ representations and  $\sigma: \rho_1q_1\Rightarrow\rho_2q_2$
a natural transformation. We show that
 $$(\mathcal G_1\xleftarrow{q_1} \mathcal K\xrightarrow{q_2} \mathcal G_2,\sigma) \sim_{\mathrm{fib}} {(\mathcal G_2\xleftarrow{p_0}\mathcal G_2^I\xrightarrow{p_1} \mathcal G_2, \sigma^{\rho_2})} \circ (\mathcal G_1\xleftarrow{q_1} \mathcal K\xrightarrow{q_2} \mathcal G_2,\sigma).$$
For this we consider the following diagram, which commutes up to the three natural transformations shown and where the bottom quadrilateral is a pullback diagram
\begin{align}\label{eq:pbdiag}
\xymatrix{ &&&\Vect_\C &&&\\  &\G_1\ar[dl]_{\rho_1} \ar[rru]^{\rho_1}&&\K  \xtwocell[u]{}<>{ <1> \sigma } \ar[d]^F\ar[ll]_-{q_1} \ar[rr]^{q_2} &&\G_2 \ar[dr]^{\rho_2}  \ar[ull]_{\rho_2} \\\Vect_\C &&\K\ar[lu]_-{q_1} \ar[dr]^{q_2} \xtwocell[ll]{}<>{  ^<0> \sigma }
   & \K \times_{\G_2}\G_2^I \ar[l]_-{T_1} \ar[r]^-{T_2} & \G_2^I \ar[dl]_-{p_0} \ar[ur]^-{p_1} \xtwocell[rr]{}<>{ ^<-1> {\sigma^{\rho_2}}\quad } && \Vect_\C\\\
& && \G_2. \ar[ulll]^{\rho_2}  \ar[urrr]_{\rho_2}   && 
 } 
 \end{align}
The functor $F$ is induced via the universal property of the pullback by the identity functor $\id_\K:\K\to \K$ and the functor $ I q_2: \K \to \G_2^I$, where $I: \G_2 \to \G_2^I$ is the inclusion that sends objects of $\G_2$ to constant functors and morphisms to constant natural transformations. The commuting diagram implies
 $ (\sigma^{\rho_2} T_2F) \circ (\sigma T_1F)=\sigma$.
 
2.~By Lemma \ref{lem:doldgrp}, 
there exists  a functor $F': \K\times_{\G_2} \G_2^I \to \K$ that also makes the diagram commute, up to the shown natural transformations, and such that there are natural isomorphisms over $\G_1 \times \G_2$
$$F'F\xRightarrow[\G_1 \times \G_2]{\eta}{ \id_\K} \qquad \textrm{ 
and } \qquad 
FF'\xRightarrow[\G_1 \times \G_2]{\eta'}{ \id_{\K\times_{\G_2} \G_2^I}} .$$
To complete the proof, we need an explicit expression for such a functor $F'$, obtained as follows. 
For each object  
 $(k,f)$ of $\K\times_{\G_2} \G_2^I$, where $k\in \Ob\K$ and $f: q_2(k) \to y$ is a morphism in $\G_2$, we fix a lift $\hat{f}_k: k \to k_f$ of $(1_{q_1(k)},f)$ with respect to  the fibration $\langle q_1,q_2 \rangle: \K \to \G_1 \times \G_2 $.  We choose the lifts such that $(\hat 1_x)_k=1_k$, whenever ${q_2}(k)=x$.
We then define
  $$F'\left (\vcenter{ \xymatrix{ k\ar[d]_g \\k'} } ,  \vcenter{    \xymatrix{ q_2(k) \ar[d]_{q_2(g)} \ar[r]^{f} & y \ar[d]^{h}  \\ 
    q_2(k') \ar[r]_{f'}  &y' }}  \right) = \left( \vcenter{\xymatrix{ k_f \ar[d]^{\hat{f}_{k'}\circ  g \circ \hat{f}_k^{-1}} \\ {k_f'} }}\right ) .  $$
 It follows directly that this  defines a functor $F': \K\times_{\G_2}\G_2^I\to \K$. By construction, the functor $F:\K\to \K\times_{\G_2} \G_2^I$ in diagram \eqref{eq:pbdiag} is given by 
 $$F\left (\vcenter{ \xymatrix{ k\ar[d]_g \\k'} } \right)= \left (\vcenter{ \xymatrix{ k\ar[d]_g \\k'} } ,  \vcenter{    \xymatrix{ q_2(k) \ar[d]_{q_2(g)} \ar[r]^{1_{q_2(k)}} & q_2(k) \ar[d]^{q_2(g)}  \\ 
    q_2(k') \ar[r]_{1_{q_2(k')}}  &q_2(k') }}  \right).  $$
With our condition on the lifts, we then have $F'F=\id_\K$ and  
$$
FF'\left (\vcenter{ \xymatrix{ k\ar[d]_g \\k'} } ,  \vcenter{    \xymatrix{ q_2(k) \ar[d]_{q_2(g)} \ar[r]^{f} & y \ar[d]^{h}  \\ 
    q_2(k') \ar[r]_{f'}  &y' }}  \right)=\left (\vcenter{\xymatrix{ k_f \ar[d]^{\hat{f}_{k'}\circ  g \circ \hat{f}_k^{-1}} \\ {k_f'} }} ,  \vcenter{    \xymatrix{ y \ar[d]_{f\circ q_2(g)\circ f^\inv} \ar[r]^{1_{y}} & y \ar[d]^{f\circ q_2(g)\circ f^\inv}  \\ 
    y' \ar[r]_{1_y}  & y' }}  \right).
$$
The inverses of the lifts $\hat f_k: k\to k_f$ then define a natural transformation $\eta': FF'
\xRightarrow[\G_1\times\G_2]{}
\id_{\K\times_{\G_2} \G_2^I}$. 
Therefore, $F$ and $F'$ define an equivalence over $\G_1\times\G_2$.

3.~The natural transformation for  the composite  span is $\sigma'=({\sigma^{\rho_2}} T_2)\circ (\sigma T_1): \rho_1q_1T_1\Rightarrow \rho_2p_1 T_2$. It remains to show that $\sigma F'=\sigma'$ and $\sigma' F=\sigma$. From 1.~we have $\sigma'F=(\sigma^{\rho_2} T_2F)\circ (\sigma T_1F)=\sigma$. Using the naturality of $\sigma$ and that $\rho_1q_1(\hat f_k)=1_{\rho_1q_1(k)}$, we obtain  for all objects $(k,f)$ in $\K\times_{\G_2} \G_2^I$
\begin{align*}
(\sigma F')_{(k,f)}=\sigma_{k_f}=\rho_2q_2(\hat f_k)\circ \sigma_k\circ \rho_1 q_1(\hat f_k)^\inv=({\sigma^{\rho_2}} T_2)_{(k,f)}\circ \sigma_k=({\sigma^{\rho_2}} T_2)_{(k,f)}\circ (\sigma T_1)_{(k,f)}=\sigma'_{(k,f)}.
\end{align*}
\end{proof}

\begin{remark}\label{rem:subcatrem}
The symmetric monoidal category $\mathrm{span}^{\mathrm{fib}}(\Grpd)$ from  Lemma \ref{def:spangrpd} can be viewed as a subcategory of $\mathrm{span}^{\mathrm{fib}}(\repGrpd)$, whose objects are pairs $(\G, \C:\G\to\Vect_\C)$ of a groupoid $\G$ and the constant functor $\C:\G\to\Vect_\C$ and whose morphisms are equipped with identity natural transformations.
\end{remark}

In the following, it  will sometimes be more practical to use  different representatives of the identity morphisms in $\mathrm{span}^\fib(\Grpd)$ and in $\mathrm{span}^\fib(\repGrpd)$. 
Those are obtained from the functor $\Pi_1:\Top\to \Grpd$ that assigns to a topological space its fundamental groupoid and to a continuous map the induced functor between fundamental groupoids. The inclusions $\iota_0: \bullet \to [0,1]$, $\bullet\mapsto 0$ and $\iota_1:\bullet\to[0,1]$, $\bullet\mapsto 1$ define functors $\Pi_1(\iota_0), \Pi_1(\iota_1):\bullet\cong \Pi_1(\bullet)\to \Pi_1[0,1]$ and
 for any  groupoid $\G$  a  span of groupoids
\begin{equation}\label{eq:def_ell}
\ell_\G:=\big (\G\cong \G^\bullet \xleftarrow{\Pi_1(\iota_0)^*} \G^{\Pi_1[0,1]} \xrightarrow{\Pi_1(\iota_1)^*} \G^\bullet\cong \G\big).
\end{equation}
As $\langle\iota_0,\iota_1\rangle:\bullet \sqcup \bullet \to \Pi_1[0,1]$ is a cofibration and $\Grpd$ a monoidal model category, with all objects fibrant,  the span $\ell_\G$ is fibrant for any groupoid $\G$.  {As $\Pi_1[0,1]$ is contractible,  $\G^{\Pi_1[0,1]}$ is essentially finite for any essentially finite groupoid $\G$}.

There is  a (unique) natural transformation $[\delta]:\Pi_1(\iota_0)\Rightarrow \Pi_1(\iota_1)$, defined by the homotopy class of any path $\delta:0\to 1$ in $[0,1]$.
Hence for any groupoid $\G$, we have 
a natural transformation 
$[\delta]^*: \Pi_1(\iota_0)^*\Rightarrow\Pi_1(\iota_1)^*$ with components $[\delta]^*_f=f([\delta])$ for any functor $f: \Pi_1[0,1] \to \G$.

\begin{lemma} \label{lem:pi1interval}For all essentially finite groupoids $\G$ and representations $\rho:\G\to\Vect_\C$  we have
\begin{align*}
&(\G \xleftarrow{\Pi_1(\iota_0)^*} \G^{\Pi_1[0,1]} \xrightarrow{\Pi_1(\iota_1)^*}  \G)\sim_{\mathrm{fib}} (\G\xleftarrow{p_0} \G^I\xrightarrow{p_1} \G),\\
&(\G \xleftarrow{\Pi_1(\iota_0)^*} \G^{\Pi_1[0,1]} \xrightarrow{\Pi_1(\iota_1)^*}  \G, \rho [\delta]^*)\sim_{\mathrm{fib}} (\G\xleftarrow{p_0} \G^I\xrightarrow{p_1} \G,\sigma^\rho).
\end{align*}
\end{lemma}
\begin{proof} We show that these spans are equivalent over $\G \times \G$.
For this, note that $I=\Pi_1([0,1],\{0,1\})$  and that the inclusion $F: I\to \Pi_1[0,1]$ defines a morphism of cospans of groupoids 
$$\xymatrix@R=2pt{ && I \ar[dd]^F \\
              &\bullet\ar[ur]^{\Pi_1(\iota'_0)} 
              \ar[dr]_{\Pi_1(\iota_0)\quad} && \bullet  \ar[ul]_{\Pi_1(\iota'_1)} 
              \ar[dl]^{\quad \Pi_1(\iota_1)} \\
                && \Pi_1[0,1],
              }$$
               where $\iota_i$ and $\iota'_i$ for $i=0,1$ denote the maps $\iota_i:\{\bullet\}\to[0,1]$, $\bullet\mapsto i$ and $\iota'_i: \{\bullet\}\to\{0,1\}$, $\bullet\mapsto i$.
We choose for each  $x\in[0,1]$ a path  $\gamma_x:x\to q_x$,  where $q_x \in \{0,1\}$ such that $q_0=0$ and $q_1=1$. By assigning to $x\in [0,1]$ the point $q_x\in\{0,1\}$  
and to the homotopy class $[\delta]:x\to y$ of a path $\delta$ in $[0,1]$ the homotopy class $[\gamma_y]\circ [\delta]\circ [\gamma_x]^\inv:q_x\to q_y$, we obtain a functor $F': \Pi_1[0,1] \to I$ such that $F' F=\id_I$. 
We also obtain a natural isomorphism $g: \id_{\Pi_1[0,1]} \Rightarrow F F'$ with components $g_x=[\gamma_x]: x\to q_x$, such that $g_0=1_0$ and $g_1=1_1$. 
This yields the following equivalence of groupoid spans over $\G \times \G $ 
$$\xymatrix{ && \G^I \ar@/_1pc/[dd]_{F'^*} \ar[ld]_{\Pi_1(\iota'_0)^*=p_0\quad} \ar[rd]^{\quad p_1=\Pi_1(\iota'_1)^*}  
\ar@{<-}@/^1pc/[dd]^{F^*}\\
              &\G  \cong \G^\bullet  &&\G^\bullet \cong \G.    \\
                && \G^{\Pi_1[0,1]}. \ar[ul]^{\Pi_1(\iota_0)^*} \ar[ur]_{\Pi_1(\iota_1)^*}
              } $$
     The unique isomorphism $d:0\to 1$ in  $I=\Pi_1([0,1],\{0,1\})$ is the homotopy class $d=[\delta]$ of any path $\delta:0\to 1$ in $[0,1]$.
            The natural transformation $\sigma^\rho:\rho p_0\Rightarrow \rho p_1$ for a representation $\rho:\G\to \Vect_\C$  is given by $\sigma^\rho=\rho [\delta]^*: \rho\Pi_1(\iota'_0)^*\Rightarrow\rho\Pi_1(\iota'_1)^*$.             With the definition of $F$ and $F'$ this yields 
            $$\sigma^\rho F^*=\rho [\delta]^*: \rho \Pi_1(\iota_0)^*\Rightarrow\rho\Pi_1(\iota_1)^*\qquad\qquad  \rho [\delta]^*F'^*=\sigma^\rho:\rho p_0\Rightarrow\rho p_1.$$
       \end{proof}

\subsection{The quantum Quinn construction}
\label{subsec:quantumquinn}

In this section we construct a symmetric monoidal functor $L: \mathrm{span}^{\mathrm{fib}}(\repGrpd)\to \Vect_\C$.  This functor will play an important role  in the construction of defect TQFTs. As we will show in the following, it
generalises a particular case of the  `degroupoidification' functor in \cite{BHW}. It 
can also be seen as  a natural generalisation of the groupoid-case of Quinn's finite total homotopy TQFT  \cite[Lecture 4]{Q}, see also \cite[Chapter 8]{G-CKT}, where representations of groupoids are incorporated. This motivates  the title of this subsection.

It is also related to the  parallel section functor $\mathrm{Par}: \mathrm{VecBun}_K\mathrm{Grpd}\to\mathrm{Vect}_K$ of Schweigert and Woike \cite[Theorem 3.17]{SW1}, where the category $\mathrm{VecBun}_K\mathrm{Grpd}$ from \cite[Def 3.7]{SW1} is the counterpart of  our category $\mathrm{span}^{\mathrm{fib}}(\repGrpd)$.  The work \cite{SW1} in turn uses a more general functor  constructed by Trova in \cite[Theorem 5.1]{T} and an older construction of such a parallel section functor by Willerton \cite{W}. {The articles \cite{SW1,T}  build on the work
\cite{FHLT} by Freed, Hopkins Lurie and Teleman, the work \cite{Mo} by Morton and the work \cite{Ha} by Haugseng on iterated spans in $(\infty,n)$-categories.}

Besides the fact that these references consider \emph{twisted} Dijkgraaf-Witten theory while we restrict attention to the \emph{untwisted} case, they consider categories built from general spans of groupoids, where the composition of  spans is defined using homotopy pullbacks \cite{BHW}. When using general spans, the  latter is required as otherwise the pullback groupoid obtained when composing spans would not necessarily be invariant under the equivalence of spans in Lemma~\ref{def:spangrpd}. As we take pullbacks along fibrations, we already obtain  that the pullbacks of spans are invariant under such fibred equivalence of spans, up to fibred homotopy. Therefore, our definition of composition of spans retain the required homotopy-theoretical information. For a  general discussion on why pullbacks along fibrations have the correct homotopy type see \cite[\S13.3.1.]{Hirschhorn}.

One advantage of our construction of $\mathrm{span}^{\mathrm{fib}}(\repGrpd)$ is that it 
gives a  concrete and explicit description. 
Moreover, all spans arising in our description of Dijkgraaf-Witten theory are fibrant, and fibrancy has a physical interpretation, namely that gauge transformations extend from the boundary to the bulk.

Our construction requires  some known results on groupoid fibrations by Brown \cite[\S 7.2.1 and  \S 7.2.9]{Br} that we summarise in the following proposition. 
Recall from Section \ref{sec:groupoidbasic} the notion of a fibration of groupoids. The \textbf{fibre} of a functor $F: \A \to \B$ between groupoids at   $B\in\Ob\B$ is the groupoid $\Fib(F,B)$ with 
\begin{align*}
&\Ob\, \Fib(F,B)=\{A\in\Ob\A\mid F(A)=B\}\qquad 
\hom_{\Fib(F,B)}(A,A')=\{f\in \hom_\A(A,A')\mid F(f)=1_B\}.
\end{align*}

For a groupoid $\G$ and  $G\in \Ob\G$ we set $\mathrm{Aut}_\G(G) = \hom_{\mathcal G}(G,G)$. We
write $[G]_\G=\{G'\in\Ob\G\mid G'\cong G\}\subset \Ob \G$ for the isomorphism class of $G$, also referred to as its path-component, and 
$\pi_0(\G)$ for its set of path-components. 

\begin{proposition}  \label{prop:omnibus} \cite[\S 7.2.1 and \S 7.2.9]{Br} Let $p: \A \to \B$ be a fibration of groupoids. 
\begin{compactenum}

\item\label{it2} For any  $A\in \Ob \A$ with $B=p(A)$ there is an exact sequence of groups
$$\{1\}\to \Aut_{\Fib(p,B)}(A)\xrightarrow{\iota} \Aut_{\A}(A) \xrightarrow{p_*} \Aut_{\B}(B),$$
where $\iota$ is induced by the inclusion $ \Fib(p,B) \hookrightarrow\A$, and $p_*$ by the fibration $p: \A \to \B$.

\item\label{actfib} For any  $B\in\Ob\B$ there is a left action $\rhd: \Aut_\B(B)\times \pi_0(\Fib(p,B))\to \pi_0 (\Fib(p,B))$ 
given by
$$
g\rhd [A]_{\Fib(p,B)}=[A_g]_{\Fib(p,B)}
$$
for any  lift $\hat{g}: A_g \to A$ of $g$. 
Furthermore, one has for all $A\in \Fib(p,B)$:
\begin{compactenum}
\item \label{triv}   The action of $p_* ( \Aut_{\A}(A))$ on $\left [ A \right ]_{\Fib(p,B)} $ is trivial.
\item \label{stab}The stabiliser of $\left [ A \right ]_{\Fib(p,B)} $  is $p_*(\Aut_\A(A))$,
\item \label{it2v} The orbit of $\left [ A \right ]_{\Fib(p,B)} $  
is 
$$\Aut_\B(B)\rhd[A]_{\Fib(p,B)}= \big\{ [A']_{\Fib(p,B)}\mid  A' \in [A]_{\A}, p(A)=B \big \},$$
and for any set   of representatives $\mathcal{R}(\B)$ of the path-components of $\B$,  there is a bijection
$$\amalg_{ B \in \mathcal{R}(\B)}\; \pi_0(\Fib(p,B))/ \Aut_{\B}(B)  \to \pi_0(\A),\quad {\Aut_\B(B)\rhd [A]_{\Fib(p,B)}\mapsto [A]_\A}.$$
\end{compactenum}

\item \label{essfin}If $\A$ and $\B$ are essentially finite, then so is $\Fib(p,B)$ for each $B \in \Ob\B$.
\item \label{it4} The cardinalities of the automorphism groups are related by
$$|\Aut_\A(A)|=\frac{|\Aut_{\Fib(p,B)}(A)| \cdot |\Aut_\B(B)|}{ |\Aut_{\B}(B)\rhd [ A  ]_{\Fib(p,B)}|}. $$
\item\label{it5} If $A,A'\in \Ob \Fib(p,B)$ are connected by a morphism in $\A$, then  $\Aut_{\Fib(p,B)}(A)\cong \Aut_{\Fib(p,B)}(A')$.
\end{compactenum}
\end{proposition}
\begin{proof}
1.~Claim \ref{it2}.~and claims 2.(a) and 2.(b)  are in   \cite[\S 7.2.1 and  7.2.9]{Br} and can be proven by elementary methods. For claim 2.(c), note that if $A, A' \in \Ob \Fib(p,B)$  {are} in the same path-component in $\A$, then there is a   morphism $f : A' \to A$ in $\A$, and clearly, $p(f) \rhd  [ A  ]_{\Fib(p,B)} =  [ A'  ]_{\Fib(p,B)}$. 

For  claim \ref{essfin}.~we only need to prove that $\Fib(p,B)$ has a finite number of path-components for all $B\in\Ob\B$. Let $E_B$ be the full {subgroupoid} of $\A$ with elements $e \in \Ob\A$ such that $p(e)\in [B]_\B$ {as objects}. If $A,A'\in \Ob E_B$ are in the same {path-component of} $\A$, then  they are in the same path-component {of} $E_B$. In particular, $E_B$ has only a finite number of path-components. Choose  representatives $A_1,\dots,A_n$ of the path-components of  $E_B$.  As $p: \A \to \B$ is a fibration, we can assume that $p(A_1)=\dots= p(A_n)=B$. It follows from 2.(c)  that the union of the orbits of $[A_1]_{\Fib(p,B)},\dots, [A_n]_{\Fib(p,B)}$
is $\pi_0(\Fib(p,B))$. As $\Aut_\B(B)$ is finite, this implies that $\Fib(p,B)$ is finite.

Claim \ref{it4}.~follows from the exact sequence in claim \ref{it2}.~together with the orbit-stabiliser theorem, by applying claim 2.(b). Claim \ref{it5}.~holds, because any morphism $f: A \to A'$ in $\A$ induces an isomorphism
$$\phi_f: \Aut_{\Fib(p,B)}(A)\to \Aut_{\Fib(p,B)}(A'),\; 
h \mapsto f h f^{-1}.$$ 

\end{proof}

We will now assign to every object $\rho:\G\to \Vect_\C$ of the symmetric monoidal category $\mathrm{span}^{\mathrm{fib}}(\repGrpd)$ 
a vector space, namely the limit $\lim \rho$, and to each morphism a linear map. To construct the latter, recall
 that the limit of a functor $\rho: \mathcal G\to\Vect_\C$ from a groupoid $\mathcal G$ is given by
\begin{align}
\label{eq:limexp}
\lim \rho=\bigoplus_{G\in R(\mathcal G)}\;\rho(G)^{\mathrm{Aut}(G)}\qquad  \rho(G)^\mathrm{Aut(G)}=\{x\in \rho(G)\mid \rho(g)x=x\,\forall g\in \mathrm{Aut}(G)\},
\end{align}
where $ R(\mathcal G)$ is a  set of representatives of the isomorphism classes of objects in $\mathcal G$. 
The inclusions and projections for the direct sums in \eqref{eq:limexp}
are denoted 
\begin{align}\label{eq:projinclids}
\iota_{G}:  \rho(G)^{\mathrm{Aut}(G)}\to \lim\rho \qquad \textrm{ and } \qquad \pi_{G}: \lim \rho\to \rho(G)^{\mathrm{Aut}(G)}. 
\end{align}
In this formulation the component morphisms $c_{G'}:\lim\rho\to \rho(G')$ of the limit cone are given by $c_{G'}\circ \iota_G=\rho(f)\vert_{\rho(G)^{\mathrm{Aut}(G)}}$ for any morphism $f: G\to G'$ with $G\in \mathcal R(\G)$ and $G'\in\Ob\G$.

Consider now a  pair $\big(\mathcal G_1\xleftarrow{P_1} \mathcal G\xrightarrow{P_2} \mathcal G_2,\sigma\big)\colon (\G_1,\rho_1) \to (\G_2,\rho_2) $ of a fibrant span $\mathcal G_1\xleftarrow{P_1} \mathcal G\xrightarrow{P_2} \mathcal G_2$ and a natural transformation $\sigma: \rho_1P_1\Rightarrow\rho_2P_2$
as in Definition \ref{def:spanrepgrpd}. Then for any pair of objects $G_1\in \Ob\G_1$ and $G_2\in \Ob\G_2$  and  any object $H\in\Ob \Fib\big ( \langle P_1,P_2 \rangle , (G_1,G_2)\big)$ 
we have a linear map 
$\sigma_H: \rho_1(G_1)\to \rho_2(G_2)$.

\begin{lemma}\label{lem:ind path comp} For any $H\in \G$ with $P_1(H)=G_1$ and $P_2(H)=G_2$ the linear map $\sigma_H: \rho_1(G_1) \to \rho_2(G_2)$ depends only on the path-component of $H$ in the  groupoid   $\Fib\big ( \langle P_1,P_2 \rangle , (G_1,G_2)\big)$.    
\end{lemma}
\begin{proof}
By  naturality of $\sigma_H$ the following diagram commutes for any morphism $f: H \to H'$ in $\Fib\big ( \langle P_1,P_2 \rangle , (G_1,G_2)\big)$ 
\begin{align}\label{eq:diagliftt}
\begin{gathered}\xymatrix{&&(\rho_1 P_1)(H) \ar[rr]^{\sigma_H} \ar[dd]_{\rho_1P_1(f)} &&(\rho_2 P_2)(H) \ar[dd]^{\rho_2P_2(f)}\\& \rho_1(G_1) \ar[ur]^{\id} \ar[dr]_{\id} &&&& \rho_2(G_2) \ar[ul]_{\id} \ar[dl]^{\id}\\
&&(\rho_1 P_1)(H') \ar[rr]_{\sigma_{H'}}  &&(\rho_2 P_2)(H').
}\end{gathered}
\end{align}
As $f: H \to H'$ is a morphism in the fibre, the  vertical arrows  are identities, which yields 
$\sigma_H=\sigma_{H'}$. 
\end{proof}

Let
$\mathcal G_1\xleftarrow{P_1} \mathcal \H\xrightarrow{P_2} \mathcal G_2$ be a fibrant span of essentially finite groupoids. As $\langle P_1, P_2\rangle \colon \H \to \G_1 \times \G_2$ is a fibration and all groupoids are essentially finite,  by Proposition \ref{prop:omnibus}, \ref{essfin}.~the fibre $$\{G_1|\H|G_2\}:=\Fib\big ( \langle P_1,P_2 \rangle , (G_1,G_2)\big)$$ is essentially finite  for all  $G_1 \in \Ob\G_1$ and $G_2\in \Ob\G_2$. Hence, each set $\mathcal{R}(\H,G_1,G_2)$  of representatives of its path-components is finite, and all automorphism groups $\Aut_{\{G_1|\H|G_2\}} (H)$   are finite.

\begin{lemma}\label{lem:smap}
Let 
$G_1 \in \Ob \G_1$ and $G_2\in\Ob\G_2$. The following  linear map  is independent of
the choice of  the set  $\mathcal{R}(\H,G_1,G_2)$ of representatives of  path-components of $\{G_1 | \H|G_2\}$ and takes values in  $\rho_2(G_2)^{\Aut_{\G_2}(G_2)}$
\begin{equation}\label{eq:def Ssigma} S_\sigma^{G_1,G_2}:= \frac{1}{|\Aut_{\G_2}(G_2)|}\sum_{H \in \mathcal{R}(\H,G_1,G_2)} \frac{1}{|\Aut_{\{G_1|\H|G_2\}} (H)|} \; \sigma_H: \rho_1(G_1)\to \rho_2(G_2).
\end{equation}
\end{lemma}

\begin{proof}
The first statement follows from Lemma \ref{lem:ind path comp}.

For the second choose $g \in \Aut_{\G_2}(G_2)$ and,
 for each $H \in \mathcal{R}(\H,G_1,G_2)$, a morphism  $g_H: H\to H'_g$ in $\H$ such that $P_1(g_H)=1_{G_1}$ and $P_2(g_H)=g$. This is possible 
because $\langle P_1, P_2\rangle: \H \to \G_1 \times \G_2$ is a fibration.

By applying Proposition \ref{prop:omnibus}, \ref{actfib}.~to the
$\Aut_{\G_1 \times \G_2}(G_1,G_2)$-action on $\pi_0(\{G_1|\H|G_2\})$ one finds
that the objects $H'_g$ for  $H \in \mathcal{R}(\H,G_1,G_2)$ form another set of representatives for the path-components of the fibre $\{G_1|\H|G_2\}$. 
We then consider the commuting diagram \eqref{eq:diagliftt} for $H'=H'_g$ and $f=g_H: H\to H'_g$. In this case, the vertical arrow on the left is given by the identity $\rho_1P_1(g_H)=1_{G_1}$ and the vertical arrow on the right by $\rho_2P_2(g_H)=\rho_2(g)$. 
So the linear map $\rho_2(g)\circ  S_\sigma^{G_1,G_2}: \rho_1(G_1) \to \rho_2(G_2)$ can be expressed  as

\begin{align*}
\rho_2(g) \circ  S_\sigma ^{G_1,G_2} &= \frac{1}{|\Aut_{\G_2}(G_2)|}\sum_{H \in \mathcal{R}(\H,G_1,G_2)} \frac{1}{|\Aut_{\{G_1|\H|G_2\}}(H)|}\; \rho_2(g) \circ \sigma_H\\
 & =  \frac{1}{|\Aut_{\G_2}(G_2)|} \sum_{H \in \mathcal{R}(\H,G_1,G_2)} \frac{1}{|\Aut_{\{G_1|\H|G_2\}}(H)|} \;\sigma_{H'_g}\\
  & =  \frac{1}{|\Aut_{\G_2}(G_2)|}\sum_{H \in \mathcal{R}(\H,G_1,G_2)} \frac{1}{|\Aut_{\{G_1|\H|G_2\}}(H'_g)|} \;\sigma_{H'_g}=  S_\sigma ^{G_1,G_2},
\end{align*}
where we used Proposition \ref{prop:omnibus}, \ref{it5}.~in the last step.
\end{proof}

\begin{proposition} \label{prop:lfunctor}
There is a symmetric monoidal functor $L: \mathrm{span}^{\rm fib}(\repGrpd)\to \mathrm{Vect}_\C$ that assigns:
\begin{compactitem}
\item to an object $(\mathcal G,\rho)$ the vector space $\lim\rho$;
\item to a morphism  $[(\mathcal G_1\xleftarrow{P_1}\mathcal H\xrightarrow{P_2} \mathcal G_2, \sigma)]: (\mathcal G_1,\rho_1)\to (\mathcal G_2,\rho_2)$ the linear map 
\begin{align}\label{eq:ldef}
\mathcal{L}(\sigma)&=\sum_{\substack{
G_1 \in \mathcal{R}(\G_1) \\ 
G_2 \in \mathcal{R}(\G_2)} } \iota_{G_2} \circ  S_\sigma^{G_1,G_2} \circ \pi_{G_1}: \lim \rho_1 \to \lim \rho_2.
\end{align}
\end{compactitem}
\end{proposition}
\begin{proof}
1.~We show that $L$ respects the composition of morphisms. For this, we consider a composite
$$[(\mathcal G_2\xleftarrow{Q_1}\mathcal K\xrightarrow{Q_2} \mathcal G_3,\tau )]\circ [(\mathcal G_1\xleftarrow{P_1}\mathcal H\xrightarrow{P_2} \mathcal G_2, \sigma)]  =[(\mathcal G_1\xleftarrow{P_1T_1} \mathcal H\times_{\mathcal G_2} \mathcal K\xrightarrow{Q_2T_2} \mathcal G_3, \tau T_2\circ \sigma T_1)].$$ 
given by the following commuting diagram, where the middle diamond is a pullback
$$\xymatrix{& &&\mathcal H\times_{\mathcal G_2} \mathcal K \ar[dl]_{T_1} \ar[dd]_{P} \ar[dr]^{T_2}   
\\ & &\H\ar[dl]_{P_1} \ar[dr]_{P_2} && \K \ar[dr]^{Q_2} \ar[dl]^{Q_1}\\ & \G_1 & & {\mathcal G_2}  & & \mathcal G_3. } $$ 
We then have a fibration $\langle P_1 T_1,P,Q_2T_2\rangle : \H \times_{\G_2} \K \to \G_1\times \G_2 \times \G_3$, by definition of the pullback. 
 Again by definition of the pullback, 
its fibre  at $(G_1,G_2,G_2)$ is the product $\{G_1\mid \H\mid G_2\}\times\{G_2\mid \K\mid G_3\}$
 of the 
 fibre of $\langle P_1,P_2\rangle$ at $(G_1,G_2)$ and the fibre  of $\langle Q_1,Q_2\rangle$ at $(G_2,G_3)$.
 Moreover,
the  restriction  $P_{G_1,G_3}: \{G_1|\H\times_{\G_2} \K|G_3\} \to \G_2$ of $P$ to the fibre $\{G_1|\H\times_{\G_2} \K|G_3\}$ of $\langle  P_1 T_1,Q_2T_2\rangle: \H\times_{\G_2} \K \to \G_1 \times {\G_3}$ is   a fibration.

For  $G_i\in\Ob\G_i$ choose sets of representatives $\mathcal{R}_{G_1,G_2}=\mathcal{R}(\H,G_1,G_2)$ and   $\mathcal{R}_{G_2, G_3}=\mathcal{R}(\K,G_2,G_3)$ of the path-components of $\{G_1\mid\H\mid G_2\}$ and $\{G_2\mid \K\mid G_3\}$. Choose   also
 a  set of representatives
 $\mathcal{R}(\G_2)$  of the path-components of $\G_2$. 
As $\langle P_1T_1,P,Q_2T_2\rangle \colon \H \times_{\G_2} \K \to \G_1\times \G_2 \times \G_3$ is a fibration, the set
$$Y_{G_1,G_3}:=\bigcup_{G_2\in \mathcal R(\G_2)} \mathcal{R}_{G_1,G_2}\times \mathcal{R}_{G_2,G_3}  $$
intersects each path-component of $\{G_1|\H\times_{\G_2} \K|G_3\}$. It is also clear that for distinct elements $G_2, G'_2\in \mathcal{R}(\G_2)$ the elements $(H,K)$ and $(H',K')$ are in different path-components of $\{G_1|\H\times_{\G_2} \K|G_3\}$ for all $H\in \mathcal{R}_{G_1,G_2}$, $H'\in \mathcal{R}_{G_1,G'_2}$  and $K\in \mathcal{R}_{G_2,G_3}$, $K'\in \mathcal{R}_{G'_2, G_3}$.
We can therefore choose for each $G_1\in\Ob\G_1$ and $G_3\in \Ob\G_3$ a set of representatives $\mathcal{R}_{G_1,G_3}$ of the path-components of $\{G_1|\H\times_{\G_2} \K|G_3\}$ that is contained in $Y_{G_1,G_3}$.

Choose now sets of representatives $\mathcal R(\G_1)$ and $\mathcal R(\G_3)$ of  the path-components in $\G_1$ and  $\G_3$. Then we compute for each 
 $G_1 \in \mathcal{R}(\G_1)$ and $G_3 \in \mathcal{R}(\G_3)$ 
\begin{align*}
 \pi_{G_3}\circ& \L(\tau)\circ \L(\sigma)\circ \iota_{G_1}=\sum_{
G_2 \in \mathcal{R}(\G_2) }   S_\tau^{G_2,G_3} \circ S_\sigma^{G_1,G_2} \\
 &=\sum_{G_2 \in \mathcal{R}(\G_2)} \quad \sum_{H \in \mathcal{R}_{G_1,G_2}} \quad \sum_{K \in \mathcal{R}_{G_2,G_3}} \quad  \frac{|\Aut_{\G_1}(G_1)|^{-1}}{|\Aut_{\{G_1|\H|G_2\}}(H)|}  
 \frac{|\Aut_{\G_2}(G_2)|^{-1}}{|\Aut_{\{G_2|\K|G_3\}}(K)|}  \; \tau_{K} \circ \sigma_{H}\\
  &=\sum_{G_2 \in \mathcal{R}(\G_2)} \quad \sum_{H \in \mathcal{R}_{G_1,G_2}} \quad \sum_{K \in \mathcal{R}_{G_2,G_3}} \quad  \frac{|\Aut_{\G_1}(G_1) |^{-1}}{|\Aut_{\{G_1|\H|G_2\}}(H)|} \frac{|\Aut_{\G_2}(G_2) |^{-1}}{|\Aut_{\{G_2|\K|G_3\}}(K)|} 
  \; \big ((\tau T_2)\circ (\sigma T_1 ) \big)_{(H,K)}. 
\end{align*} 
Applying Proposition \ref{prop:omnibus}, \ref{it4}.~to the fibration $P_{G_1,G_3}: \{G_1|\H\times_{\G_2} \K|G_3\} \to \G_2$ yields 
\begin{align}\label{eq:sumformula}
\big|\Aut_{\{G_1|\H|G_2\}}(H)\big|\,\, \big| \Aut_{\{G_2|\K|G_3\}}(K)\big|= \frac{\big |  \Aut_{\{G_1|\H\times_{\G_2} \K|G_3\}}  ( H,K) \big |\,\, \big|\Aut_{\G_2}(G_2)\rhd [(H,K)]  \big|}{\Aut_{\G_2}(G_2)}
\end{align}
for all  $H\in \mathcal R_{G_1,G_2}$ and $K\in \mathcal R_{G_2,G_3}$, 
where 
$[(H,K)]$ is taken in the $G_2$-fibre of $P_{G_1,G_3}$. 
Hence, we have
\begin{multline}\label{eq:composition-action}
 \pi_{G_3}\circ \L(\tau)\circ \L(\sigma)\circ \iota_{G_1}\\ =  \frac{1}{|\Aut_{\G_1}(G_1) |} \sum_{G_2 \in \mathcal{R}(\G_2)} \, \sum_{H \in \mathcal{R}_{G_1,G_2}} \, \sum_{K \in \mathcal{R}_{G_2,G_3}} \frac{1}{ \big|\Aut_{\G_2}(G_2)\rhd [(H,K)]  \big|}\frac{\big ((\tau T_2)\circ (\sigma T_1 ) \big)_{(H,K)} } 
{\big |  \Aut_{\{G_1|\H\times_{\G_2} \K|G_3\}}  ( H,K) \big |}.  
  \end{multline}
 
We now use the following elementary observation: For any finite group $A$ acting on a finite set $M$ and any $A$-invariant map $f: M \to V$ into some vector space $V$, one has
$$\sum_{[m]\in M/A} f(m)=\sum_{m \in M} \frac{f(m)}{|A \rhd m|}.$$ We apply this to 
the  $\Aut_{\G_2}(G_2)$-action on the path-components of the  $G_2$-fibre of $P_{G_1,G_3}$ in \eqref{eq:composition-action}, which is induced by the $\Aut_{\G_2}(G_2)$-actions on $\pi_0(\{G_1\mid \H\mid G_2\})$ and $\pi_0(\{G_2\mid \K\mid G_3\})$ in the product
$$\pi_0\big(\{G_1\mid \H\mid G_2\}\times \{G_2\mid \K\mid G_3\}\big)\cong\pi_0\big (\{G_1\mid \H\mid G_2\}\big)\times\pi_0\big (\{G_2\mid\K\mid G_3\}\big)\cong \mathcal{R}_{G_1,G_2} \times \mathcal{R}_{G_2,G_3}.$$
Proposition \ref{prop:omnibus}, 2.(c)  gives a bijection 
$$\coprod_{G_2 \in \mathcal{R}(\G_2)} \pi_0\big(\{G_1\mid \H\mid G_2\}\times \{G_2\mid \K\mid G_3\}\big) / \Aut_{\G_2}(G_2)  \to \pi_0\big (\{G_1|\H\times_{\G_2} \K|G_3\}  \big),$$
and applying this to  Equation \eqref{eq:composition-action} shows that $\L$ is compatible with compositions:
\begin{align*}
 \pi_{G_3}\circ \L(\tau)\circ \L(\sigma)\circ \iota_{G_1}=\frac{1}{|\Aut_{\G_1}(G_1)|} \sum_{X \in \mathcal{R}_{G_1,G_3}} \frac{1}{\big|\Aut_{\{G_1\mid \H\times_{\G_2} \K\mid G_3\}}(X)\big|}\;\big((\tau T_2)\circ (\sigma T_1)\big)_{X}.
 \end{align*}

3.~To show that $L$ respects identity morphisms, we consider the identity span $[(\mathcal G\xleftarrow{p_0}{\mathcal G}^I\xrightarrow{p_1}\mathcal G,\sigma^\rho)]$ on $(\mathcal G,\rho)$ with $\sigma^\rho_f=\rho(f)$ for all morphisms $f$ in $\mathcal G$. Choose a set $\mathcal{R}(G)$ 
 of representatives of the path-components of $\G$. Then we have $\hom_\mathcal G(G_1,G_2)=\emptyset$, for $G_1, G_2\in  R(\mathcal G)$, with $G_1\neq G_2$. It follows that
$$ \pi_{G_2}\circ \L(\sigma^\rho)\circ \iota_{G_1}=0.$$
 On the other hand,
{for any} $G\in \Ob\G$ the fibre $\{G| \G^I| G\}=\Fib(\langle p_0,p_1 \rangle, (G,G)) $ has only identity morphisms,  and exactly $|\Aut_\G(G)|$ components. We can thus choose
$\mathcal{R}(\G^I,G,G) = \Aut_\G(G)$ and obtain
\begin{align*}
 \pi_{G}\circ \L(\sigma^\rho)\circ \iota_{G}= \frac{1}{|\Aut_\G(G)|} \sum_{f \in \Aut_\G(G)} \sigma^\rho_f=\id_{\rho(G)^{\Aut_\G(G)}}.
 \end{align*}

4.~Let us now address independence on representatives.  Suppose that
$$(\mathcal G_1\xleftarrow{P_1} \mathcal G\xrightarrow{P_2} \mathcal G_2,\sigma)\sim_{\fib} (\mathcal G_1\xleftarrow{P'_1} \mathcal G'\xrightarrow{P'_2} \mathcal G_2,\sigma').$$
Then there are functors $F:\mathcal G\to\mathcal G'$ and $F':\mathcal G'\to \mathcal G$  with $P'_i F=P_i$, $P_i F'=P'_i$  that form an equivalence over $\mathcal G_1\times\mathcal G_2$ and  satisfy $\sigma' F=\sigma$ and $\sigma F'=\sigma'$. Then for any $x_i\in \G_i$ the restriction $F_{x_i}: \Fib(P_i,x_i)\to \Fib(P'_i, x_i)$ is an equivalence of categories. 
 Applying this to  \eqref{eq:def Ssigma} and using  $\sigma' F=\sigma$ and $\sigma F'=\sigma'$ shows that $\L(\sigma)=\L(\sigma')$.

5.~Compatibility with the symmetric monoidal structure is clear.
\end{proof}

The symmetric monoidal  functor $L:\mathrm{span}^{\mathrm{fib}}(\repGrpd)\to \Vect_\C$ from Proposition \ref{prop:lfunctor} generalises the symmetric monoidal  functor   $L_{\mathrm{class}}: \mathrm{span}(\Grpd)\to \Vect_\C$ for Dijkgraaf-Witten TQFT without defects, constructed in a more general setting by Quinn in \cite[Lecture 4]{Q}, see
also Freed and Quinn \cite{FQ} and Section 4 in Faria Martins and Porter  \cite{FMP}.
If we identify the  symmetric monoidal category $\mathrm{span}^{\mathrm{fib}}(\Grpd)$ from  Lemma \ref{def:spangrpd}  with a subcategory of $\mathrm{span}^{\mathrm{fib}}(\repGrpd)$ as in Remark \ref{rem:subcatrem}, then the restriction of $L$ to this subcategory defines a functor $L_{\mathrm{class}}: \mathrm{span}^{\mathrm{fib}}(\Grpd)\to \Vect_\C$. 

\begin{corollary}\label{cor:quinnclass} There is a symmetric monoidal functor $L_{\mathrm{class}}: \mathrm{span}^{\mathrm{fib}}(\Grpd)\to \Vect_\C$ that assigns
\begin{compactitem}
\item to an essentially finite groupoid $\G$ the free complex vector space $L_\G=\langle \pi_0(\G)\rangle_\C$ generated by $\pi_0(\G)$,
\item to a fibrant span $s=(\G_1\xleftarrow{p_1}\G\xrightarrow{p_2} \G_2)$ the linear map $L_s:\langle \pi_0(\G_1)\rangle_\C\to \langle \pi_0(\G_2)\rangle_\C$ with matrix elements
$$
\langle [G_1]\mid  L_s\mid[G_2]\rangle =\frac 1 {|\mathrm{Aut}_{\G_2}(G_2)|}\sum_{G\in\mathcal R(\G, G_1,G_2)}\frac 1 { |\mathrm{Aut}_{\{G_1\mid \G\mid G_2\}}(G)|}.
$$
\end{compactitem}
\end{corollary}

The formula for the matrix elements in Corollary \ref{cor:quinnclass} can be cast in a  simpler form by using  the homotopy content or groupoid cardinality  of an essentially finite groupoid $\H$, see Baez, Hoffnung and Walker \cite{BHW}
$$\chi^\pi(\H)=\sum_{H\in \pi_0(\H)} \frac 1 {|\mathrm{Aut}_\H(H)|}.$$ 
Denoting by $\PC_H(\H)$ for $H\in\Ob\H$ the full subgroupoid of $\H$ with object set $[H]_\H$, we then have 
$$
\langle [G_1]\mid  L_s\mid[G_2]\rangle =\chi^\pi\big(\PC_{G_2}(\G_2)\big) \cdot \chi^\pi\big(\{G_1\mid \G\mid G_2\}\big). 
$$
This 
arises as a specialisation of the functor constructed by Quinn in \cite[Lecture 4]{Q} to the case where the  target spaces are given by classifying spaces of groupoids. For a more detailed  discussion see \cite[\S 3.3, \S 4.3 and \S 8.2]{FMP} and \cite{FMP1}.

\section{Stratifications and graded graphs}
\label{sec:stratgraph}
\subsection{Graded graphs}
\label{subsec:graphs}

In the following all graphs are assumed to be finite. 
We consider the category $\mathrm{Graph}$, whose objects are finite graphs, $Q=(Q_0,Q_1, s,t: Q_1\to Q_0)$, and whose morphisms  $f: Q\to Q'$ are graph maps, given by  maps $f_0: Q_0\to Q'_0$ and $f_1: Q_1\to Q'_1$ with $s'\circ f_1={f_0\circ s}$ and $t'\circ f_1=f_0\circ t$.  The forgetful functor $R:\mathrm{Cat}\to \mathrm{Graph}$ has a left adjoint
 $L:\mathrm{Graph}\to \mathrm{Cat}$ that assigns to a graph $Q$ the
category $\mathcal Q=L(Q)$ freely generated by $Q$ and to a graph map $f: Q\to Q'$ the  induced functor $F=L(f):\mathcal Q\to\mathcal Q'$.

A \textbf{path} of length $k\in \Z_{\geq 0}$ in a graph $Q$ is a finite sequence $q=(e_k,\ldots, e_1)$ of edges $e_i\in Q_1$ with $s(e_{i})=t(e_{i-1})$ for $i=2,\ldots, k$. We write  $s(q)=s(e_1)$ and $t(q)=t(e_k)$ for the starting and target vertex of $q$ and  $q: s(q)\to t(q)$. For  $v\in Q_0$, we denote by $Q^k(v)$ the set of paths of length $k$ starting in $v$  and by $Q(v)=\bigcup_{k=0}^\infty Q^k(v)$ the set of all paths starting in $v$.  

The \textbf{subgraph generated} by $v\in Q_0$ is the  graph $Q^v$ with vertex and edge sets
$$
Q^v_0=\{t(q)\mid  q\in Q(v)\}\qquad Q^v_1=\{ e\in Q_1\mid s(e)\in Q^v_0\}.
$$
The maps $s,t: Q^v_1\to Q^v_0$ are  induced
by the ones of $Q$ and the inclusion $i_v: Q^v\to Q$.

\begin{definition}  A \textbf{graded graph} $(Q,\deg)$ is a   graph $Q=(Q_0,Q_1, s,t: Q_1\to Q_0)$ together with a  \textbf{degree map} $\deg: Q_0\to \Z_{\geq 0}$ such that $\deg\circ t=\deg\circ s-1$. 
\end{definition}

\begin{definition}\label{def:insertion} An \textbf{insertion} $f: (Q,\deg)\to (Q',\deg')$   is a graph map $f:Q\to Q'$  that 
\begin{compactenum}[(i)]
\item  preserves degree: $\deg'\circ  f_0=\deg$,
\item induces graph isomorphisms   $f_v: Q^v\to Q'^{f_0(v)}$ for all $v\in Q_0$. 
 \end{compactenum}
\end{definition}

Clearly, the identity map $\id_Q: (Q,\deg)\to (Q,\deg)$ on a graded graph $Q$ is an insertion. Condition (ii) implies that for an insertion  $f:Q\to Q'$, any edge $q'\in Q'_1$  with $s'(q')=v'\in f_0(Q_0)$ satisfies $q'\in f_1(Q_1)$. More specifically, for each $v\in f_0^\inv(v')$, there is a unique edge $q_v\in Q_1$ with $s(q_v)=v$ and $f_1(q_v)=q'$.
Hence, for any path $q':y'\to z'$ in $Q'$ with $y'\in f_0(Q_0)$ and any $y\in f_0^\inv(y')$, there is a unique path $q_y:y\to z$ whose image is $q$.  This implies that the functor $F:\maq\to\maq'$ induced by an insertion $f:Q\to Q'$ is a discrete opfibration in the sense of \cite[{Section 2.1}]{LR}.

\begin{definition}\label{def:insertive} A functor $F: \maq \to \maq'$ between small categories is called 
 a \textbf{discrete opfibration} if
for every  $q'\in\hom_{\maq'}(y',z')$  and   $y\in F^\inv(y')$ there is a unique  morphism
$q_y$ in $\maq$ with source $y$ and  $F(q_y)=q'$.
\end{definition}

\begin{lemma}\label{lem:I-faithfulandInsertive} 
For any insertion $f: (Q,\deg)\to (Q',\deg')$ the 
functor $F:\mathcal Q\to\mathcal Q'$ is  
a discrete opfibration.
\end{lemma}

In the following we will consider graded graphs that have at least one vertex of degree $0$ and such that each vertex of degree 1 has exactly two outgoing edges. These edges  come in two types.

\begin{definition}\label{def:reggraph} A graded graph $(Q,\deg)$ is called \textbf{regular}, if
 each vertex of degree $>0$ has at least one outgoing edge and each vertex of degree  $1$ has exactly two outgoing edges.

 A choice of  \textbf{normals} for a regular graph 
is a partition  $(\deg\circ s)^\inv(1)=L\dot\cup R$ such that
 each  vertex $ v\in Q_0$ with $\deg(v)=1$ has
 one outgoing edge, $l(v): v\to L(v)$, in $L$ and the other one, $r(v): v\to R(v)$, in $R$.
\end{definition}

\subsection{Stratifications}\label{subsec:strats}  

\subsubsection{Homogeneous stratifications}
\label{subsubsec:homogeneous}
Throughout the article we work in the piecewise linear (PL) framework, see for instance Rourke and Sanderson \cite{RS} and Friedman \cite[Appendix B]{Fb}. All manifolds and mappings are PL. Each compact PL manifold is a compact polyhedron and  the geometric realisation of a finite simplicial complex. For every
PL map  between  PL manifolds,  there is a locally-finite decomposition of the domain into simplexes such that 
the map is affine linear on each simplex. Every PL map arises from a simplicial map between  simplicial complexes.
    
 All manifolds in this article are  oriented PL manifolds of dimension $n\leq 3$, possibly open and  with boundaries. Their boundaries are oriented such that their normals
  point outwards.  For better legibility we write \textbf{$n$-manifold}  for an  oriented PL manifold of dimension $n\leq 3$. 
  A (not necessarily connected) closed 2-manifold  is called a \textbf{surface}.

We define a \textbf{filtration} of a Hausdorff space $X$ as in  \cite[Def 2.2.1]{Fb} as a sequence of closed subspaces $\emptyset=X^{-1}\subset X^0\subset \ldots\subset X^n=X$ for some integer $n\geq -1$. The space $X^k$ is called  \textbf{$k$-skeleton},  the connected components of $X^k\setminus X^{k-1}$ are called the  \textbf{$k$-strata} and $n$ is called the \textbf{formal dimension} of $X$.  A homeomorphism  $f: X\to Y$ between filtered spaces $X,Y$ of the same formal dimension is called a \textbf{filtered homeomorphism} if $f(X^k)\subset Y^k$ for all $k=0,\ldots, n$, see \cite[Def 2.3.2]{Fb}. Note that any open subset $U\subset X$ of a filtered space $X$ inherits a filtration with $k$-skeleta $U^k=U\cap X^k$.

\begin{definition} A \textbf{stratification} of an $n$-dimensional  PL  manifold $X$  is a filtration by  closed PL subspaces
$$
\emptyset=X^{-1}\subset X^0\subset X^1\subset \ldots\subset X^n=X
$$
 such that $X^k\setminus X^{k-1}$ is a $k$-dimensional PL manifold  with boundary  $\partial(X^k\setminus X^{k-1})\subset \partial X$.
The \textbf{codimension} of a $k$-stratum $s$ is $\mathrm{codim}(s)=n-k$, and the set of $k$-strata is denoted $S^X_k$.

We  call
0-strata   \textbf{vertices},  1-strata   \textbf{edges},    2-strata   \textbf{planes} and 3-strata  \textbf{regions} of $X$.  
For $n=3$, a 1-stratum that is PL homeomorphic to $S^1$  is called an \textbf{isolated loop}.
\end{definition}

Note that the empty manifold is a $k$-manifold for all $k\in\Z_{\geq 0}$. Thus $X_k=X_{k-1}$ is allowed, and in this case there are no $k$-strata.
As explained for instance in \cite{AH}, see also  \cite[Lemma 2.5.12]{Fb}, for each
stratification of a compact $n$-manifold $X$ there is a simplicial complex $C$ with underlying polyhedron $X$ such that all skeleta of $X$ are subcomplexes of $C$.  This follows, because $X$ and all $k$-skeleta $X^k$ are
compact polyhedra and hence  underlying polyhedra of  simplicial complexes $D$ and $D^k$ \cite[Theorem 2.11]{RS}. By \cite[Addendum 2.12]{RS} there are simplicial subdivisions $C$ of $D$ and $C^k$ of $D^k$  such that all $C^k$ are subcomplexes of  $C$. This implies that for each $k$ the number of $k$-strata is finite, in
 particular   $X^0$ is a finite discrete set.

We also need a condition that ensures  that  neighbourhoods of  points on a stratum are  isomorphic and that the stratification is transversal at the boundary. This is the homogeneity condition below, which
is equivalent to the condition called homogeneous in  \cite[Def 2.24]{BMS}, which in turn restricts the condition called CS in \cite{S,AH}.
For PL manifolds without boundary, imposing this condition yields a special case of the PL stratified pseudomanifolds in \cite[Def.~2.5.13]{Fb}.

In addition to  homogeneity we also impose  a further condition which is not present in \cite[Def 2.24]{BMS}. It is imposed to ensure
that every $k$-stratum is contained in the closure of at least one $l$-stratum for all $k\leq l\leq n$.

For an $m$-sphere $S$ we denote by $cS$ the (open) cone over $S$. Any stratification of $S$ induces a stratification of $cS$  with skeleta $(cS)^{k+1}=c(S^k)$ and the apex $v_c$ of the cone as $0$-skeleton.  The products $cS\times\R^n$ and $cS \times\R^n_+$ of this cone with $\R^n$ and with the closed half-space $\R^n_+=\{x\in \R^n\mid x_n\geq 0\}$ are also  stratified manifolds, in which each $(n+k)$-stratum is a product of $\R^n$ or $\R^n_+$ with a $k$-stratum of $cS$.

\begin{definition}\label{def:homog}
Let $X$ be a stratified manifold. A $k$-stratum $s$ is called \textbf{homogeneous} if there is a stratification of the $(n-k-1)$-sphere $S$ such that:
\begin{compactenum}
\item For any point $x\in s\cap \partial X$ there is an open neighbourhood $U_x$ of $x$ in $X$ and a filtered PL homeomorphism 
$h:(cS)\times\mathbb R^k_+ \to U_x$ with  $h(v_c,0)=x$, whose inverse is a filtered PL homeomorphism.

\item For any point $x \in s \setminus \partial X$ there  is an open neighbourhood $U_x$ of $x$ in $X$ and a filtered PL homeomorphism  $h:(cS)\times\mathbb R^k \to U_x$ with  $h(v_c,0)=x$, whose inverse is a filtered PL homeomorphism.
\end{compactenum}
The filtered PL homeomorphisms $h$ are called \textbf{local charts} and the open neighbourhoods $U_x$  \textbf{special neighbourhoods} of $x$. The stratified sphere $S$ is called the  \textbf{link} of $s$ in $X$.

A homogeneous $k$-stratum is called \textbf{saturated},  if the stratified $(n-k-1)$-sphere $S$
has  at least one $j$-stratum for $j=0,..., n-k-1$.
A stratification  is   homogeneous or saturated, if all  strata are  homogeneous or saturated.
\end{definition}

As we work in the PL setting, the link of each stratum $s$ is determined uniquely, up to isomorphisms of stratified spheres, see for instance \cite[Lemma 2.5.18]{Fb}.

As discussed in \cite{BMS}, the homogeneity condition is automatically satisfied for all $n$-strata and $0$-strata.  In particular, one has $S^0=\{-1,1\}$ and $S^{-1}=\emptyset$, whose cone is $\{\bullet\}$, so the homogeneity condition is void for $k=n$.
It  is equivalent to local flatness for $(n-1)$-strata, which holds for $n\leq 4$ by the Sch\"onflies theorem. Thus, every stratification of a manifold of dimension $n\leq 2$ is automatically homogeneous. A stratification of a  manifold of dimension $n=3$ can be made homogeneous by adding vertices to its 1-skeleton \cite{AH,BMS}. 

Likewise, all $n$-strata and $(n-1)$-strata in a stratification of an $n$-manifold are automatically saturated.  
Stratifications that are  triangulations  or dual to triangulations are always homogeneous and saturated.

As we consider such manifolds only for dimensions $2$ and $3$, we can list all types of neighbourhoods of strata and show examples in Figure \ref{fig:stratsph}.
For $n=2$ the stratified spaces $N=(cS)\times \R^k_+$ are 
\begin{compactitem}
\item $k=2$:  the half-space $N=\R^2_+$ with $N^0=N^1=\emptyset$, 
\item $k=1$: the strip  $N=[-1,1]\times \R_+$ with $N^0=\emptyset$ and  $N^1=\{0\}\times\R_+$,
\item $k=0$: a  disc $N=D^2$ with $N^0=\{0\}$ and radial  1-strata intersecting in $N^0$.
\end{compactitem}
For $n=3$, the stratified spaces $(cS)\times \R^k_+$ are 
\begin{compactitem}
\item  $k=3$:  the half-space $N=\R^3_+$ with $N^0=N^1=N^2=\emptyset$,
\item  $k=2$:  the  space $N=[-1,1]\times \R^2_+$ with $N^0=N^1=\emptyset$ and $N^2=\{0\}\times\R^2_+$,
\item  $k=1$:   $N=D^2\times\R_+$ with $N^0=\emptyset$, $N^1=\{0\}\times\R_+$ and  2-strata intersecting in $N^1$,
\item  $k=0$:  a ball  $N=D^3$ with $N^0=\{0\}$ and   2-strata intersecting in radial 1-strata.
\end{compactitem}

Note that the saturation condition requires that for  $n=2$ and $k=0$ the special neighbourhoods from Definition \ref{def:homog}  contain at least one $1$-stratum and for $n=3$ and $k=0,1$ they contain at least one 1-stratum and at least one 2-stratum.
Note also that the saturation condition  applies to all points in the special neighbourhoods from Definition \ref{def:homog}, including the ones on higher-dimensional strata. 
It follows that each $l$-stratum in the  special neighbourhoods from Definition \ref{def:homog} is contained in the closure of at least one $(l+1)$-stratum and that each $(l+1)$-stratum with $l\geq k$ contains at least one $l$-stratum in its closure. Note also that the unique $k$-stratum in these  neighbourhoods is contained in the closure of all $l$-strata with  $l\geq k$.

\begin{figure}
\begin{center}
\def\svgwidth{.6\columnwidth}
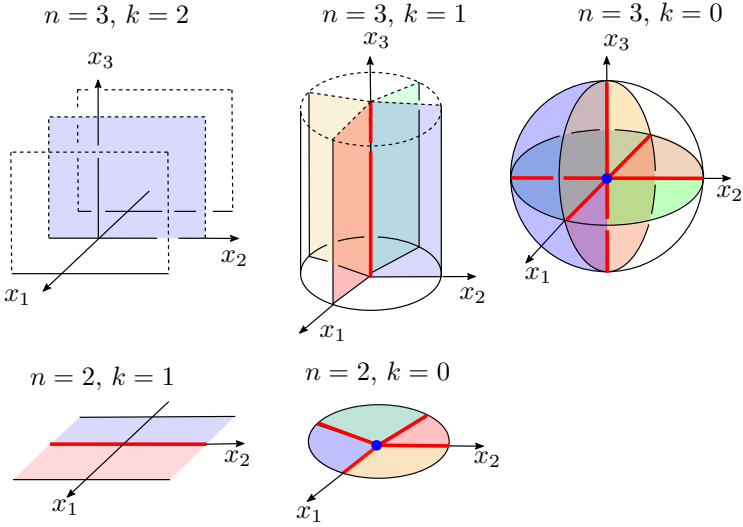
\end{center}
\caption{Neighbourhoods of $k$-strata in a stratified $n$-manifold for $k<n$. The maps $\iota_k: S_k^{\partial X}\to S_{k+1}^X$ from \eqref{eq:embedbound} are induced by the inclusion
 $\iota: \mathbb R^2\to \mathbb R^3$, $(x_1,x_2)\mapsto (x_1,x_2,0)$.}
\label{fig:stratsph}
\end{figure}

\subsubsection{Boundary stratifications}

Let $X$ be an $n$-manifold with a homogeneous and saturated stratification. 
Then the homogeneity condition implies that each $k$-stratum of $X$ is transversal to  $\partial X$, so that no $0$-stratum of $X$ is contained in  $\partial X$. 
It also ensures that a stratification of $X$  induces a homogeneous and saturated stratification of  $\partial X$. 
{If $x\in \partial X\cap s$ for a $k$-stratum $s$}, then any local chart  $h: cS\times\R^k\to  U_x$  restricts to a local chart
$h: cS\times \R^{k-1}\to U_x\cap \partial X$.

\begin{lemma}
A  homogeneous and saturated stratification of an $n$-manifold $X$  induces a homogeneous and saturated  stratification of $\partial X$
$$
\emptyset=\partial X^{-1}\subset \partial X^1\subset\ldots\subset \partial X^{n-1}=\partial X\qquad\qquad \partial X^k=X^{k+1}\cap \partial X.
$$
The subspaces $\partial X^k$ are called  \textbf{boundary $k$-skeleta}, and the connected components of $\partial X^k\setminus \partial X^{k-1}$ are called  \textbf{boundary $k$-strata}. We denote  by $S^{\partial X}_k$ the set of boundary $k$-strata.
\end{lemma}

Note that boundary $k$-strata of a stratified $n$-manifold $X$  \emph{are not} $k$-strata of $X$. In fact, they are associated with $(k+1)$-strata of $X$.
The homogeneity condition implies that for each boundary $k$-stratum $s\in S_k^{\partial X}$  there is a unique $(k+1)$-stratum  $t\in S_{k+1}^X$ with $s\subset t$. This defines (not necessarily injective) maps 
\begin{align}\label{eq:embedbound}
\iota_k: S_k^{\partial X}\to S_{k+1}^X
\end{align}
that respect the closures of strata:  if a boundary $k$-stratum $s$ is contained in the closure of a  boundary $l$-stratum $t$, then $\iota_k(s)$ is contained in the closure of  $\iota_l(t)$.

\subsubsection{Oriented, {framed} and fine stratifications}\label{subsec:orstrat}
In the following we   require that all strata of codimension $<3$ in a stratified $n$-manifold $X$ are oriented. The $n$-strata and boundary $(n-1)$-strata inherit an orientation from $X$. For the other strata, it is a datum.

\begin{definition} \label{def:orstrat} $\quad$
\begin{compactenum}
\item  An \textbf{oriented stratification} of a surface $\Sigma$ is a homogeneous  and saturated stratification of $\Sigma$ together with the choice of an orientation 
for each edge and vertex.  
\item  An \textbf{oriented stratification} of a  3-manifold $M$  
is a homogeneous and saturated  stratification of $M$ together with the choice of an orientation  for each plane and edge. 
\end{compactenum}
For $(n-1)$-strata, orientation is encoded in the choice of a normal. 
\end{definition}

Note that 
 an oriented stratification of $X$ induces an oriented stratification of  $\partial X$. 
Vertices of $\partial X$ that correspond to starting and target ends of edges in $X$ have negative and positive orientation, respectively, if the normal at $\partial X$ is oriented outwards.

For an $(n-1)$-stratum $s$ in an oriented stratification of an $n$-manifold $X$ we denote by $R(s)$ and $L(s)$ the $n$-strata at the starting and target end of the normal of $s$.   For an edge $e$ in a stratified 3-manifold $X$ that is not an isolated loop, we denote by
 $s(e)$, $t(e)$ the starting and target vertex or boundary vertex of  $e$.

We also impose that all strata of codimension 2 are equipped with a framing. This is needed to ensure that all the bundles of local strata defined in the following subsection are  trivial. As there are stratified manifolds with non-trivial bundles of local strata, cf.~Example \ref{ex:cylinderex0},  not all stratified manifolds admit a framing.

A framing of a codim 2-stratum $s$ is a coherent choice of an $n$-stratum in each special neighbourhood $U_x$ of each point $x\in s$ from Definition \ref{def:homog}. Recall that $U_x$ inherits a filtration from $X$ and that its $n$-strata are the connected components of $U_x\cap (X^n\setminus X^{n-1})$. 
Coherent means that these choices coincide, whenever the special neighbourhoods intersect in a point on $s$. 

\begin{definition} \label{def:framed} Let $X$ be an $n\leq 3$-manifold with an oriented homogeneous and saturated stratification. 
\begin{compactenum}
\item  A \textbf{framing} of an $(n-2)$-stratum $s\in S^X_{n-2}$ is a choice of an $n$-stratum $t_{U_x}$ in each special neighbourhood $U_x$ of each point $x\in s$, such that
 $t_{U_x}\cap W= t_{U_y}\cap W$ for all  special neighbourhoods $U_x$, $U_y$ of $x,y\in s$ and 
 neighbourhoods $W$ of a point $z\in s\cap U_x\cap U_y$. 
 \item The stratification is called \textbf{framed}, if each $(n-2)$-stratum is equipped with a framing.
 \end{compactenum}
\end{definition}

Note that any framed oriented homogeneous stratification of a 3-manifold $X$ induces a framed oriented homogeneous stratification of its boundary $\partial X$. The framing of the 0-strata of $\partial X$ is induced by the framing of the 1-strata of $X$.

\textbf{Convention:} In the following we mean by \textbf{stratified $n$-manifold}  an $n$-manifold with a \emph{framed} \emph{oriented} \emph{homogeneous} and \emph{saturated} stratification, where $n\in\{2,3\}$.
All stratifications considered hereafter  satisfy these conditions unless stated otherwise.
\textbf{Embeddings and homeomorphisms of stratified $n$-manifolds} are filtered PL embeddings and PL homeomorphisms that preserve the orientations of the manifolds and of all strata and the framings of all codimension 2-strata.

It will sometimes be useful to consider stratifications in which  all strata and boundary strata are open balls and every stratum is incident to at least one vertex or boundary vertex. This holds for instance for stratifications that are triangulations or dual to triangulations of $X$.

\begin{definition} A framed homogeneous and saturated stratification of an $n$-manifold $X$ is called  \textbf{fine}, if 
\begin{compactenum}[(i)]
\item all $k$-strata  and boundary $k$-strata of $X$ are open $k$-balls,
\item the closure of each stratum contains at least one vertex or boundary vertex,
\item it induces a fine stratification of $\partial X$.
\end{compactenum}
\end{definition}

\subsubsection{Bundles of local strata}

A stratified $n$-manifold $X$ comes with a notion of  \emph{local $j$-strata}  at any stratum  and at any point  of $X$. The latter are essentially germs, but can be described in a 
more direct way in our setting. 

Recall that any point $x\in X$ is contained in a unique $k$-stratum $s$ of $X$. Any  special neighbourhood $U_x$ of $x\in s$ as in Definition \ref{def:homog} inherits a filtration by closed subspaces from $X$. Its $j$-strata are the connected components of $U_x\cap (X^j\setminus X^{j-1})$ for $j\geq k$. They are in bijection with $(j-k)$-strata in the cone $cS$ and with $j$-strata in the spaces $cS\times \R^k$ or $cS\times \R^k_+$ from Definition \ref{def:homog}. 

Given two special neighbourhoods $U_x$ and $V_x$ of $x\in s$, we identify a $j$-stratum $p$ of $U_x$ and a $j$-stratum $q$ of $V_x$, whenever there is a neighbourhood $W$ of $x$ such that $p\cap W=q\cap W$. This is an equivalence relation on the set of all $j$-strata of special neighbourhoods of $x$.

\begin{definition} \label{def:local_stratum} Let $X$ be a stratified $n$-manifold,  $s$ a stratum of $X$ and $x\in s$.
A \textbf{local $j$-stratum} at $x$ is an equivalence class of  $j$-strata of special neighbourhoods of $x$.
\end{definition}

By the homogeneity condition, the number of local $j$-strata at a point on a $k$-stratum $s$ depends only on $s$. Hence, there is a bijection between local $j$-strata at $x$ and at $y$ for all $x,y\in s$. However, {for stratifications without framings or without orientations,  this bijection need not be} canonical.

The local $j$-strata at points $x\in s$ define a locally trivial bundle $\pi_j:E_j\to s$ with typical fibre $\{1,\ldots, l_j\}$, where $l_j$ is the number of local $j$-strata at any point $x\in s$.
A local trivialisation is given by the local charts of $s$. For each $x\in s$, there is a neighbourhood $W_x$ of $x$ and a PL map $H: W_x \times (cS)\times\mathbb R^k \to X$ such that for each $y \in W_x$ the map $z\mapsto H(y,z)$ is a local chart, with domain either $ (cS)\times\mathbb R^k $, if $z \in s\setminus \partial X$, or  $ (cS)\times\mathbb R^k_+ $, if $z \in \partial X \cap s$. This follows,
because it holds for the stratified manifolds  $(cS)\times\mathbb R^k_+$ in  Figure \ref{fig:stratsph}. The transition functions between the charts are given by the identifications of the local strata with strata in $(cS)\times\mathbb R^k$ or $(cS)\times\mathbb R^k_+$. The  fibre bundle construction theorem, see for instance \cite[Theorem 3.2]{Hus}, yields the fibre bundle.

In our situation the bundle $\pi_j: E_j\to s$ of local $j$-strata at a $k$-stratum $s$ is trivial, i.e.~of the form $E_j=s\times \pi_j^{-1}(x)$ for any $x \in s$.
 This holds trivially for $k=0$, where $s$ contains a single point, and for $k=j$, where there is a single local $k$-stratum at each point $x\in s$.
For $k=n-1$ and $j=n$, we have a canonical identification of local $n$-strata at any two points $x,y\in s$,  because the orientation of $s$ gives a  globally defined normal on $s$. This shows that $E_n= s\times \pi_n^{-1}(x)$ for all $x \in s$.
It remains to consider  $k=1$ and $n=3$, where $s$ is either an interval or an isolated loop.
If $s$ is an interval, then the bundle of local $j$-strata is trivial, as $s$ is contractible.  If $s$ is an isolated loop,  the framing defines a trivialisation.

The triviality of  $\pi_j: E_j\to s$ gives a canonical identification between local strata at any two points $x,y\in s$. 
This allows us to define $j$-local strata \emph{of the stratum} $s$ as path-components of the total space of the bundle $\pi_j: E_j\to s$ and to canonically identify them with 
 local strata \emph{at any point} $x\in s$. 

\begin{definition} \label{def:local_stratum_stratum} Let $X$ be a stratified $n$-manifold.
A \textbf{local $j$-stratum of a stratum $s$} is a path-component of the bundle $\pi_j: E_j\to s$. The set of local $j$-strata at  $s$ is denoted $S^{loc}_j(s)=\pi_0(E_j)$.
The unique element  of $q\cap \pi_j^\inv(x)$ for a local stratum $q$ at $s$ and $x\in s$  is called the \textbf{representative} of $q$ at $x$. 
\end{definition}

\begin{example}\label{ex:cylinderex0} The existence of framings is a necessary condition  {for the triviality of the bundles of local 2- and 3-strata of} an isolated loop. {Let} 
 $X=D^2\times S^1$ be the stratified 3-manifold with boundary $\partial X=S^1\times S^1$ obtained by gluing the top and bottom face of the following stratified solid cylinder with a rotation by $\pi$.
\begin{align*}
\begin{tikzpicture}[scale=.45]
\draw[color=black, line width=1pt] (0,0) ellipse (2cm and 1cm);
\draw[color=red, line width=1pt,  fill=red, fill opacity=.2] (0,4)--(2,4)--(2,0)--(0,0)--(0,4);
\draw[color=red, line width=1pt,  fill=red, fill opacity=.2] (0,4)--(-2,4)--(-2,0)--(0,0)--(0,4);
\node at (-1.2,4)[anchor=south, color=red]{$s$};
\draw[color=red,->,>=stealth, line width=1pt] (1,4)--(.4, 3.6);
\draw[color=red,->,>=stealth, line width=1pt] (-1,4)--(-.4, 4.4);
\draw[color=blue, line width=2pt] (0,4)--(0,0)  node[sloped, pos=0.5, allow upside down]{\arrowOut}  node[pos=.5, anchor=west]{$r$};  
\draw[color=black, line width=1pt] (-2,4)--(-2,0);
\draw[color=black, line width=1pt] (2,4)--(2,0);
\draw[color=black, line width=1pt, ->,>=stealth] (-2.6,0)  arc (180:360: 2.6cm and 1.4cm) node[pos=.5, anchor=north]{$\pi$};
\node at (0, -.5){$t$};
\draw[color=black, line width=1pt] (0,4) ellipse (2cm and 1cm);
\end{tikzpicture}
\end{align*}
Then  $X^1=r=S^1$, $X^2=s\cup r$ is a M\"obius strip and $X^3=D^2\times S^1$ a solid torus. Note that  $s=X^2\setminus X^1$ is orientable, so this stratification of $D^2\times S^1$ can indeed be oriented.
Every \emph{point} $x\in r$  has one local 1-stratum, two local 2-strata and two local 3-strata, but  
 the \emph{stratum} $r$ has  only one local 1-, 2- and 3-stratum. 
In this case, the isolated loop cannot be framed. 
\end{example}

Note  that local $j$-strata at a stratum $s$
are not necessarily in bijection with $j$-strata  $t$ such that $s\subset \bar t$. For  instance, for a vertex $v$ the \emph{local 1-strata} at $v$  are the  \emph{edge ends} incident at $v$, while the \emph{1-strata} at $v$ are the \emph{edges} incident at $v$.  For each  $k$-stratum $s$ and each  $j\geq k$ there is a map
$T_{s,j}: S^{loc}_j(s)\to S^X_j$
that assigns to a local $j$-stratum at $s$ the associated $j$-stratum  of $X$. In general, this map is not injective.  

\begin{definition} \label{eq:tmap}
We write $q:s\to t$ for a
local $j$-stratum $q\in S^{loc}_j(s)$  with associated $j$-stratum $t=T_{s,j}(q)$. 
\end{definition}

\pagebreak
\begin{example} \label{ex:locstrat}
Consider the following stratification of a 2-sphere  with a single 0-stratum $x$, 1-strata $t,u,v,w$ and 2-strata $S,T,U,V,W$.  \\[-35ex]

\begin{align*}\qquad\qquad\qquad\qquad\qquad\qquad
\begin{tikzpicture}[scale=.7]
\draw[color=red, line width=1pt,fill=red, fill opacity=.2] (0,0).. controls (11,-3) and (-3,11).. (0,0);
\draw[line width=1pt, red,->,>=stealth] (-.4,3)--(.1,2.7);
\draw[color=blue, line width=1pt,fill=white] (0,0).. controls (5,1) and (1,5).. (0,0);
\draw[color=blue, line width=0pt,fill=blue, fill opacity=.2] (0,0).. controls (5,1) and (1,5).. (0,0);
\draw[color=cyan, line width=1pt,fill=cyan, fill opacity=.2] (0,0).. controls (-5, 3) and (-5,-3).. (0,0);
\draw[color=violet, line width=1pt,fill=violet, fill opacity=.2] (0,0).. controls (-3,-5) and (3,-5).. (0,0) ;
\draw[fill=black](0,0) circle (.15);
\node at (0,-.3)[anchor=north]{$x$};
\draw[dashed, line width=.5pt] (0,0) circle (2.3);
\node at (-.3,-3.6) [anchor=north, color=violet, anchor=north east]{$w$};
\draw[line width=1pt, violet,->,>=stealth] (0,-3.7)--(0,-4.3);
\node at (0,-3) [color=violet]{$W$};
\node at (-1.1,-1.5) [color=violet]{$w_1$};
\node at (1.1,-1.5) [color=violet]{$w_2$};
\node at (2.5,2.5) [color=blue]{$t$};
\node at (2,2) [color=blue]{$T$};
\draw[line width=1pt, blue,->,>=stealth] (.7,1.8)--(1.2,1.5);
\node at (3,1) [color=red]{$U$};
\node at (4.5,1.3) [color=red]{$u$};
\node at (-4.2,0)[color=cyan, anchor=south]{$v$};
\draw[line width=1pt, cyan,->,>=stealth] (-3.8,0)--(-4.3,0);
\node at (-3,0)[color=cyan]{$V$};
\node at (-1,.8)[color=cyan]{$v_1$};
\node at (-1,-.8)[color=cyan]{$v_2$};
\node at (-3,2.5) {$S$};
\node at (1.5,.1)[color=red]{$U_1$};
\node at (0,1.8)[color=red]{$U_2$};
\node at (1.5,-.7)[color=red]{$u_1$};
\node at (1.3,.9) [color=blue]{$t_1$};
\node at (.8,1.3) [color=blue]{$t_2$};
\node at (-.8,1.5)[color=red]{$u_2$};
\node at (-2, 2)[color=black] {$S_1$};
\node at (-2,-2)[color=black] {$S_2$};
\node at (2,-2)[color=black] {$S_3$};
\end{tikzpicture}
\end{align*}\\[-6ex]
The dashed disc is a neighbourhood of $x$ as in Definition \ref{def:homog}, and the local strata at $x$ are
\begin{compactitem}
\item a single local 0-stratum $x: x\to x$,
\item local 1-strata $t_1, t_2: x\to t$, $u_1,u_2: x\to u$, $v_1,v_2: x\to v$, $w_1,w_2: x\to w$,
\item local 2-strata $S_1,S_2,S_3: x\to S$,  $T: x\to T$, $U_1,U_2: x\to U$, $V: x\to V$ and $W: x\to W$.
\end{compactitem}
\end{example}

\subsubsection{Maps between sets of local strata}

For any local $k$-stratum $q:s\to t$ and $j\geq k$, any local $j$-stratum at $t$ defines a local $j$-stratum at $s$. Intuitively, this follows, because $s\subset \bar t$ and hence any stratum at $t$ is also a stratum at $s$, {and because all bundles of local strata considered in this article are trivial.
 
\begin{lemma}\label{lem:stratinj}  Every local stratum $q:s\to t$ in a stratified $n$-manifold $X$  induces maps
$$
S^q_j: S^{loc}_j(t)\to S^{loc}_j(s)\qquad\qquad {j\in \{0,1,2,3\}.}
$$
\end{lemma}
\begin{proof}
 Let $t$ be a $k$-stratum of $X$. 
 For each $x\in s$ there is a unique local $k$-stratum at $x$ that represents $q$, and for each special neighbourhood $U_x$ as in Definition \ref{def:homog}, there is a unique stratum $q'$ of $U_x$ that represents $q$. 
 
  For any $y\in q'$ we can choose a neighbourhood $U_y$ as in Definition \ref{def:homog} with $U_y\subset U_x$. This identifies every $j$-stratum $w_y$ in $U_y$ with a unique $j$-stratum $w_x$ in $U_x$ and defines a map $S^{q}_{j}: S^{loc}_j(t,y)\to S^{loc}_j(s,x)$, $w_y\mapsto w_x$, where 
$S^{loc}_j(t,y)$ and $S^{loc}_j(s,x)$ denote the sets of local $j$-strata at $x$ and $y$, respectively. Clearly, this does not depend on the choice of the special neighbourhoods $U_x$ and $U_y$.  As the bundle of local $j$-strata of  $t$ is trivial, it also does not depend on the choice of $y\in q'$ and defines a map $S^q_{j}: S^{loc}_j(t)\to \pi_j^\inv(x)$.
 By passing to the path-components of the bundle of local $j$-strata over $s$, we obtain a
  map $S^q_{j}: S^{loc}_j(t)\to S_j^{loc}(s)$  that does not depend on the choice of $x\in s$, as this bundle is trivial.
 \end{proof}

\begin{example}\label{ex:locmaps} For the stratification in Example \ref{ex:locstrat} 
we have 
$$ S^{loc}_0(x)=\{x\}\qquad 
S^{loc}_1(x)=\{t_1,t_2, u_1,u_2, v_1,v_2, w_1,w_2\}\qquad S^{loc}_2(x)=\{S_1,S_2,S_3,T, U_1,U_2, V,W\}.$$ 
The maps associated with the local 1-stratum $u_1: x\to u$  are given by 
$$
S_0^{u_1}=\emptyset:\emptyset\to S_0^{loc}(x),\qquad S_1^{u_1}: \{u\}\to S^{loc}_1(x),\; u\mapsto u_1 \qquad S^{u_1}_2:\{U,S\}\to S^{loc}_2(x),\; U\mapsto U_1, S\mapsto S_3,
$$ and the maps associated with the local 1-stratum $u_2: x\to u$   by 
$$
S_0^{u_2}=\emptyset:\emptyset\to S_0^{loc}(x),\qquad S_1^{u_2}: \{u\}\to S^{loc}_1(x),\; u\mapsto u_2 \qquad S^{u_2}_2:\{U,S\}\to S^{loc}_2(x),\; U\mapsto U_2, S\mapsto S_1.
$$ 
\end{example}

In the following we will sometimes describe local strata $q: s\to t$ at a $k$-stratum $s$ by composable sequences of local strata.
In particular, we consider sequence of local strata that each raise the dimension by one.  

\begin{definition}\label{def:repoflocstrat}
A sequence of local strata
$
s=u_0\xrightarrow{q_1} u_1\to \ldots\xrightarrow{q_m} u_m=t
$ in a stratified $n$-manifold $X$
\textbf{represents} a local stratum $q:s\to t$, if the  induced maps between the sets of local $j$-strata  agree:
$$
S_j^q=S_j^{q_1}\circ \ldots\circ  S_j^{q_m}: S^{loc}_j(t)\to S^{loc}_j(s)\qquad \forall j\in\Z_{\geq 0}.
$$
\end{definition}

\begin{lemma}\label{lem:repsequenceexists}
For each $k$-stratum $s$ of $X$ and  local $l$-stratum $q:s\to t$, there is a sequence  $$s=u_k\xrightarrow{q_{k+1}} u_{k+1}\xrightarrow{q_{k+2}} \ldots \xrightarrow{q_{l-1}} u_{l-1}\xrightarrow{q_{l}} u_{l}=t$$ of local strata  that represents $q$ and such that each local stratum raises the dimension  by one. 
\end{lemma}

\begin{proof} This is a consequence of the saturation condition. 
It holds  trivially for all local $l$-strata $q$ with $l= k+1$.  If $l>k+1$, choose a neighbourhood $U_x$ of a point $x\in s$ as in Definition \ref{def:homog}  and  an  $l$-stratum $q'$ in $U_x$ that represents  $q$. Then by the discussion at the end of Section \ref{subsubsec:homogeneous} there is an  $(l-1)$-stratum $u'_{l-1}\subset \overline q'$  in $U_x$ with $ s\subset \overline u'_{l-1}$.  If $l=k+2$, then $s\to u_{l-1}\to t$ is the required sequence of local strata, where $u_{l-1}=[u'_{l-1}]$ denotes the path-component of $u'_{l-1}$. 
If $l=k+3$,  there is an  $(l-2)$-stratum $u'_{l-2}\subset \overline u'_{l-1}$ in $U_x$ with $ s\subset \overline u'_{l-2}$,
  and 
$s\to u_{l-2}\to  u_{l-1}\to t$ with $u_i=[u'_i]$ is the  required sequence of local strata.
\end{proof}

\begin{example} For  the stratification in Example \ref{ex:locstrat} 
\begin{compactitem}
\item the local 2-stratum $U_1: x\to U$ is represented by 
$x\xrightarrow{u_1} u\xrightarrow{U} U$ and $x\xrightarrow{t_1} t\xrightarrow{U} U$, 
\item the local 2-stratum $S_3: x\to S$ is represented by 
$x\xrightarrow{u_1} u\xrightarrow{S} S$ and $x\xrightarrow{w_2} w\xrightarrow{S} S$,
\item  the local 2-stratum $W: x\to W$ is represented by 
$x\xrightarrow{w_1} w\xrightarrow{W} W$ and $x\xrightarrow{w_2} w\xrightarrow{W} W$.
\end{compactitem}
\end{example}

Finally,  note that for a stratified 3-manifold $X$ and the associated boundary stratification of $\partial X$, the inclusions
$\iota_k: S_k^{\partial X}\to S_{k+1}^X$ from  \eqref{eq:embedbound} respect  local strata. For any boundary $k$-stratum $s$, there are canonical bijections
\begin{align}\label{eq:boundinj-locstrat}
\iota_{kj}:  S^{loc}_j(s)\to S^{loc}_{j+1}(\iota_k(s))
\end{align}
that send local $j$-strata at $s$ to local $(j+1)$-strata at the unique $(k+1)$-stratum $\iota_k(s)$ of $X$ that intersects $s$.  
The bijections in \eqref{eq:boundinj-locstrat} are apparent in 
 Figure \ref{fig:stratsph}.

\subsection{Graded graphs from stratifications}
\label{sec:grgraph}

A  compact stratified $n$-manifold $X$   defines a graded graph.
Its vertices  are the strata of $X$, and the degree of a $k$-stratum $s$ is $\deg(s)=\mathrm{codim}(s)=n-k$. Edges starting at a $k$-stratum $s$  are  local $(k+1)$-strata  $q:s\to t$ at $s$ and connect $s$ to the $(k+1)$-stratum $t$.  The normals of  $(n-1)$-strata equip the graph with normals.

\begin{definition}\label{def:embquiv} Let $X$ be a compact stratified $n$-manifold. 
The \textbf{graded graph}  associated to $X$ is the regular graph $Q^X$ 
with vertex set $Q^X_0$ and edge set $Q^X_1$ given by
$$
Q^X_0=\bigcup_{j=0}^n S_j^X, \qquad \textrm{ and } \qquad  Q^X_1=\bigcup_{j=0}^{n-1} \bigcup_{s\in S^X_j} S^{loc}_{j+1}(s).
$$
The maps $s,t: Q_1^X\to Q_0^X$ are given by $s(q)=u$, $t(q)=v$ for a local $(k+1)$-stratum
$q: u\to v$ at $u\in S_k^X$.
The degree map 
  $\deg: Q^X_0\to\Z_{\geq 0}$ is given by $\deg(u)=\mathrm{codim}(u)=n-k$ for $u\in S_k^X$. The normals of the $(n-1)$-strata of $X$  define its normals.
\end{definition}

Note that the saturation condition implies that each vertex $s\in Q^X_0$ with $\deg(s)>0$ has at least one outgoing edge.  Also, each  vertex $s\in Q^X_0$ with $\deg(s)=1$ has exactly two outgoing edges, so $Q^X$ is indeed regular. 
As the  stratification of a $3$-manifold $X$ induces the boundary  stratification of  $\partial X$, it also  defines a graded graph $Q^{\partial X}$ for  $\partial X$. 
Because the inclusions  \eqref{eq:embedbound} preserve degrees and induce the bijections between the sets of local $j$-strata in \eqref{eq:boundinj-locstrat}, they define an insertion.

\begin{lemma}\label{lem:inclusions are insertive}
For any homogeneous and saturated stratification of  a compact $n$-manifold $X$, the maps from \eqref{eq:embedbound} and \eqref{eq:boundinj-locstrat} define an insertion
$i_\partial:  Q^{\partial X}\to  Q^X$, 
and hence   $I_\partial:\maq^{\partial X}\to  \maq^X$ is a discrete opfibration.
\end{lemma}
For a fine stratification of a compact  $n$-manifold $X$ the graph $ Q^X$ can be visualised by thickening the stratification.
This can be achieved for instance   by choosing a triangulation that contains all skeleta as subcomplexes and taking the second barycentric subdivision.

\begin{figure}
\begin{center}
\def\svgwidth{.6\columnwidth}
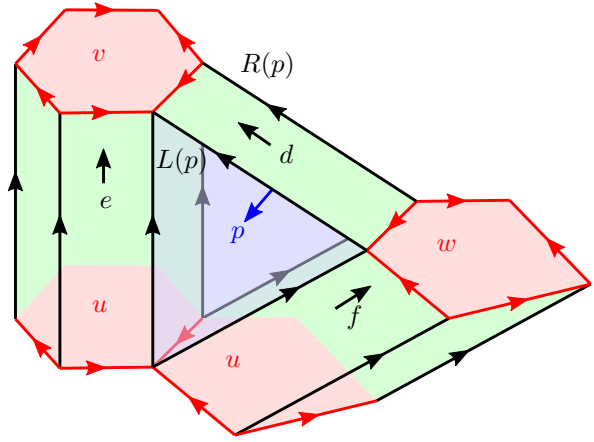
\end{center}
$\quad$\\[-12ex]
\caption{Thickening $M_{th}$ of a stratified  3-manifold $M$ with vertices $u,v,w$, edges $d,e,f$ and a plane $p$.
 }
\label{fig:defect_thick}
\end{figure}

\begin{definition} \label{def:thicksurf}Let $\Sigma$ be a  surface with a fine  stratification. The \textbf{thickened stratification} $\Sigma_{th}$ is the stratification obtained by thickening each edge of $\Sigma$  to a rectangle and each vertex of $\Sigma$ to an $n$-gon, where $n$ is the number of local 1-strata at $v$. Its sets  of planes, edges and vertices  are given by
\begin{align*}
&S^{th}_{k}=\bigcup_{j=0}^{k} \bigcup_{s\in S^\Sigma_j} S^{loc}_{j-k+2}(s),\qquad \textrm{ for } k=0,1,2.
\end{align*}
\end{definition}

The formula in
Definition \ref{def:thicksurf} states that each plane of the thickening $\Sigma_{th}$ corresponds to a unique  vertex, edge or plane of $\Sigma$. We refer to them as red, green and white planes, respectively. 

The edges of $\Sigma_{th}$ are partitioned into two sets. Edges in the first set are called red edges and correspond to pairs  of a vertex $v$ of $\Sigma$ and a local  edge  at $v$. Edges in the second set are called black edges and correspond to pairs  of an edge $e$ of $\Sigma$ and a local plane  at $e$. 
Vertices of $\Sigma_{th}$ are in bijection with pairs  of a vertex  $v$ of $\Sigma$ and a local plane at $v$. The vertex of $\Sigma_{th}$  is contained in the associated plane of $\Sigma$.

Note that each red plane has only red edges, each green plane is a rectangle with two black and two red edges, and each white plane has only black edges. 
Each black or red edge is  in a unique green rectangle.
The orientation of the stratification of $\Sigma$ induces an orientation of the  red edges of $\Sigma_{th}$. They are oriented parallel to the  normal of the  associated edge.

\begin{definition} \label{def:thick3d}Let $M$ be a compact 3-manifold with a fine stratification. The \textbf{thickened stratification} $M_{th}$ is the stratification obtained by thickening each edge and plane to a polygonal cylinder and each vertex to a polyhedron. Its sets of regions, planes, edges and vertices  are
\begin{align*}
S^{th}_{k}=\bigcup_{j=0}^{k}\bigcup_{s\in S^M_j} S^{loc}_{j-k+3}(s)\qquad k=0,1,2,3.
\end{align*}
\end{definition}

The formula in Definition \ref{def:thick3d} states that each region of  the thickening corresponds to a unique region, plane, edge or vertex of $M$. 
Planes of the thickening  correspond to pairs  (i)
  of a vertex $v$ of $M$ and a local edge at $v$,  (ii)   of an edge $e$ of $M$  and a local plane at $e$ and (iii)   of a plane $p$ of $M$ and a local region  at $p$. We refer to them as  red, green and  blue planes of $M_{th}$, respectively.

Edges of $M_{th}$  correspond to pairs (i)   of a vertex $v$ of $M$ and a local plane  at $v$ and (ii)   of an edge $e$ of $M$ and a local region  at $e$. We refer to them as  red and  black edges, respectively. Vertices of $M_{th}$ correspond to pairs  of a  vertex $v$ of $M$ and a local region at $v$. The vertex  of $M_{th}$  is contained in the associated region of $M$.
Red planes contain only red edges,  green planes are  rectangles with two red and two black edges and  blue planes contain only black edges.  
Each black or red edge is contained in a unique green rectangle.

The orientation of the stratification of $M$ defines an orientation for the edges of $M_{th}$.  Red  edges of $M_{th}$ are oriented parallel to the  normal  of the associated plane. Black edges of $M_{th}$ are oriented parallel to the orientation of the associated edge.

Note that a thickening of $M$ defines  a thickening $\partial M_{th}$ of its  boundary stratification such  that the color conventions and orientations match. They are  illustrated in  Figure \ref{fig:defect_thick}.
The relations between (local) strata  of a stratified manifold $X$,  the strata of  $X_{th}$ and  the  graded graph $Q^X$ are summarised in Table \ref{tab:strat}.

\begin{table}
\begin{tabular}{|l|l|l||l|l|l|}
$M$ & $M_{th}$ & $Q^M$ & $\Sigma$ & $\Sigma_{th}$ & $Q^{\Sigma}$\\
\hline & & & & &\\[-1.2ex]
strata & regions &  $Q_0^M$ & strata & planes & $Q^\Sigma_0$\\[+.5ex]
\hline & & & & &\\[-1ex]
$v\in S_0^M$ &$v_{th}$ & $\deg(v)=3$ & $v\in S^\Sigma_0$  & red $v_{th}$ & $\deg(v)=2$\\[+.5ex]
$e\in S_1^M$ & $e_{th}$  & $\deg(e)=2$ & $e\in S^\Sigma_1$   &  green $e_{th}$ & $\deg(e)=1$\\[+.5ex]
$p\in S_2^M$ & $p_{th}$  & $\deg(p)=1$ & $p\in S^\Sigma_2$   &  white $p_{th}$ & $\deg(p)=0$\\[+.5ex]
$r\in S_3^M$ & $r_{th}$  & $\deg(r)=0$ & & &\\[+1ex]
\hline & & & & &\\[-1ex]
pairs  $\Delta\dim=1$ & planes &  $Q^M_1$ & pairs $\Delta \dim=1$& edges &  $Q^\Sigma_1$\\[+.5ex]
\hline & & & & &\\[-1ex]
$(v,q)$, $v\in S^M_0$, $q:v\to e$ & red  & $q:v\to e$  & $(v,q)$, $v\in S^\Sigma_0$, $q:v\to e$ & red  & $q:v\to e$
\\[+.5ex]
$(e,q)$, $e\in S^M_1$, $q:e\to p$ & green  & $q:e\to p$ & $(e,p)$, $e\in S^\Sigma_1$, $q:e\to p$ & black  & $q:e\to p$\\[+.5ex]
$(p,q)$, $p\in S^M_2$, $q:p\to r$ & blue  &  $q:p\to r$ & & & \\[+1ex]
\hline & & & & &\\[-1ex]
pairs $\Delta \dim=2$& edges &   & pairs $\Delta \dim=2$& vertices & \\[+.5ex]
\hline & & & & &\\[-1ex]
$(v,q)$, $v\in S^M_0$, $q:v\to p$ & red  &  & $(v,q)$, $v\in S^\Sigma_0$, $q:v\to p$ & in $p$ & \\[+.5ex]
$(e,q)$, $e\in S^M_1$, $q:e\to r$ & black  &  & & &\\[+1ex]
\hline & & \\[-1ex]
pairs $\Delta \dim=3$ & vertices & \\[+.5ex]
\cline{1-3} & &\\[-1ex]
$(v,q)$, $v\in S^M_0$, $q:v\to r$ & in $r$ & \\[+1ex]
\cline{1-3} 
\end{tabular}
\caption{Stratification, thickening and the associated graded graph for a 3-manifold $M$ and a surface $\Sigma$.}
\label{tab:strat}
\end{table}

\section{Defect cobordism category for untwisted Dijkgraaf-Witten TQFT}\label{sec:defdata-anddefcob}
\label{sec:cobcat}
In this section, we introduce the defect data for Dijkgraaf-Witten TQFT \cite{DW} with defects and the associated category of defect cobordisms.  As Dijkgraaf-Witten TQFT is a special case of Turaev-Viro-Barrett-Westbury TQFT \cite{TV,BW}, for a detailed discussion see Petit \cite{P}, its  defect data  is a special case of the defect data for  general Turaev-Viro-Barrett-Westbury TQFTs \cite{TV,BW}. Turaev-Viro-Barrett-Westbury TQFT with defects was considered in \cite{M, CM} and  assigns the following data to the strata of a manifold of dimension $n=2,3$ 
\begin{compactitem} 
\item codim 0-strata: spherical fusion categories over $\C$,
\item codim 1-strata: finite  semisimple $\C$-linear bimodule categories with bimodule traces,
\item codim 2-strata: $\C$-linear exact bimodule functors,
\item codim 3-strata: bimodule natural transformations.
\end{compactitem}
Here, the spherical fusion categories acting on the bimodule category for a codim 1-stratum are the ones of its local codim 0-strata.
The bimodule functor for  a codim 2-stratum relates the bimodule categories of its local codim 1-strata, and the  bimodule natural transformation for a codim 3-stratum relates the bimodule functors of its local codim 2-strata. The bimodule traces, introduced by Schaumann \cite{S13,S15} and developed further by Fuchs, Schweigert and Schaumann \cite{FSS}, are required for the state sum formulation of the TQFT.

Dijkgraaf-Witten TQFT without defects amounts to specialising the spherical fusion categories to the spherical fusion categories $\mathrm{Vec}^\omega_G$ of $\C$-vector spaces graded by a finite group $G$
and linear maps that preserve the grading. The monoidal structure is given by the 3-cocycle $\omega: G\times G\times G\to\C^\times$ with the trivial cocycle  $\omega\equiv 1$ for  the untwisted case. For $\omega\equiv 1$,  $\mathrm{Vec}_G$ is has a spherical structure induced by the duals in $\Vect_\C$,  see \cite{P}. 

As shown in \cite{EGNO}, Example 7.4.10 and Exercise 7.4.11, finite-dimensional module categories over $\mathrm{Vec}_G$ are given by finite $G$-sets and certain 2-cocycles.  We only consider trivial 2-cocycles in the following. In this case, a finite semisimple $(\mathrm{Vec}_G, \mathrm{Vec}_H)$-bimodule category $\mathcal M$ corresponds to a finite $G\times H^{op}$-set $M$, whose elements are the simple objects   of $\mathcal M$.  
Bimodule functors and natural transformations between such bimodule categories  can be analysed  by considering the images 
and coherence isomorphisms on simple objects. It is shown in \cite[Example 5.11]{L} that bimodule functors $F:\mathcal M\to \mathcal N$ between $(\mathrm{Vec}_G, \mathrm{Vec}_H)$-bimodule categories $\mathcal M$, $\mathcal N$ correspond to representations $\rho:(M\times N)\sslash (G\times H)\to \Vect_\C$ of the associated action groupoid and bimodule natural transformations $\eta: F\Rightarrow F'$ between such functors to intertwiners of such representations.

In the following we do not consider Turaev-Viro-Barrett-Westbury TQFT with or without defects beyond the Dijkgraaf-Witten case and also do not consider twisting by cocycles. We thus assign the following defect data: 
\begin{compactitem}
\item codim 0-strata: finite groups,
\item codim 1-strata: finite sets with group actions,
\item codim 2-strata: representations of action groupoids,
\item codim 3-strata: intertwiners of representations of action groupoids.
\end{compactitem} 
For lack of a better terminology, the  first two layers of data are called  \textbf{classical} defect data,  because they are given by groups and group actions, the last two  \textbf{quantum} defect data,  because they are linear. 

We start by considering the classical defect data. 
The groups acting on the finite set assigned to a codim 1-stratum are specified by its normals from
Definition \ref{def:orstrat}. In the associated graded graph $Q^X$ this corresponds to the edges $l(s): s\to L(s)$ and $r(s): s\to R(s)$ from Definition \ref{def:reggraph}.

\begin{definition}\label{def:class-def-data} Let $X$ be a compact stratified $n$-manifold. 
An \textbf{assignment of classical defect data}  to $X$
\begin{compactitem}
\item {assigns} to each $n$-stratum $t$ a finite group $G_t$,
\item {assigns} to each $(n-1)$-stratum $s$ a finite $G_{L(s)}\times G_{R(s)}^{op}$-set $M_s$.
\end{compactitem}
\end{definition}

Note that an assignment of classical defect data to a compact stratified $n$-manifold $X$ defines an assignment of classical defect data to its boundary $\partial X$. 

We  now show it assigns to each stratum $r$ of $X$ an action groupoid $\mathcal D_r=M_r\sslash G_r$ and to each local stratum $p:u\to v$ a functor $D_p:\mathcal D_u\to \mathcal D_v$.

\textbf{Action groupoids for strata:}
 The  set $M_r$ in the groupoid 
$\mathcal D_r=M_r\sslash G_r$ is  the product of the sets $M_s$ for all local $(n-1)$-strata $x:r\to s$. The group $G_r$ is the product of the groups $G_t$ for all local $n$-strata $q:r\to t$
\begin{align}\label{eq:groupoiddef}
 \mathcal D_r:=M_r\sslash  G_r=\left(\textstyle \prod_{ S^{loc}_{n-1}(r)\ni x:r\to s}  M_s\right) {\sslash} \left(\textstyle \prod_{S^{loc}_n(r)\ni q:r\to t} G_t\right).
 \end{align}
 The action of $G_r$ on $M_r$ is given by the commuting diagram
\begin{align}\label{eq:grpdactdef}
 \xymatrix{
\left(\prod_{S^{loc}_n(r)\ni q:r\to t} G_t\right)\times \left(\prod_{ S^{loc}_{n-1}(r)\ni x:r\to s}  M_s\right) \ar[r]^{\qquad\qquad\rhd}\ar[d]^{\langle \pi_{lx}, \pi_{rx}, \pi_x\rangle }  &   \left(\prod_{ S^{loc}_{n-1}(r)\ni x:r\to s}  M_s\right) \ar[d]^{\pi_x}\\
G_{L(s)}\times G_{R(s)}\times M_s \ar[r]_{\qquad (g_L,g_R, m)\mapsto g_L\rhd m\lhd g_R^\inv}& M_s,
 }
 \end{align}
where $lx=S_n^x(l(s))$ and $rx=S_n^x(r(s))$ are given by the map    $S_n^x: S^{loc}_n(s)=\{l(s), r(s)\}\to S^{loc}_n(r)$  from Lemma \ref{lem:stratinj} for the local $(n-1)$-stratum $x:r\to s$.

 In particular,  $\mathcal D_t=\bullet\sslash G_t$ for each $n$-stratum $t$ and  $\mathcal D_s=M_s\sslash G_{L(s)}\times G_{R(s)}$ for each $(n-1)$-stratum $s$. For an $(n-1)$-stratum $s$ that occurs in several local $(n-1)$-strata $q:r\to s$, the associated set $M_s$ appears several times as a factor in $M_r$. Likewise, if
an $n$-stratum $t$  corresponds to several local $n$-strata $q: r\to t$, the group $G_t$ appears several times  in $G_r$.

\begin{example} For the stratification in Example \ref{ex:locstrat} the action groupoid associated to the $0$-stratum $x$ is
$$\mathcal D_x=M_x\sslash G_x=(M_t^{\times 2}\times M_u^{\times 2}\times M_v^{\times 2}\times M_w^{\times 2})\sslash (G_S^{\times 3}\times G_T\times G_U^{\times 2}\times G_V\times G_W).$$
The factors in the set $M_x$ are in bijection with the local 1-strata at $x$ and the factors in the group $G_x$ with the local 2-strata at $x$. If
we index elements  by the associated local strata,  the  action for $\mathcal D_x$ reads
\begin{align*}
&m_{t_1}\mapsto g_T\rhd m_{t_1}\lhd g_{U_1}^\inv & &m_{t_2}\mapsto g_T\rhd m_{t_2}\lhd g_{U_2}^\inv & &m_{u_1}\mapsto g_{U_1}\rhd m_{u_1}\lhd g_{S_3}^\inv & &m_{u_2}\mapsto g_{U_2}\rhd m_{u_2}\lhd g_{S_1}^\inv
\\
&m_{v_1}\mapsto g_{S_1}\rhd m_{v_1}\lhd g_V^\inv & &m_{v_2}\mapsto g_{S_2}\rhd m_{v_2}\lhd g_V^\inv
& &m_{w_1}\mapsto g_{S_2}\rhd m_{w_1}\lhd g_W^\inv & &g_{w_2}\mapsto g_{S_3}\rhd m_{w_2}\lhd g_W^\inv.
\end{align*}

\end{example}

\textbf{Functors for local strata:}
For a local $l$-stratum $p:u\to v$, the   maps 
$$S_n^p: S^{loc}_n(v)\to S^{loc}_n(u)\qquad S_{n-1}^p: S^{loc}_{n-1}(v)\to S^{loc}_{n-1}(u)$$ from Lemma \ref{lem:stratinj} assign
to each local $n$-stratum  $q:v\to t$  a local $n$-stratum  $q^p:=S_{n}^p(q): u\to t$ and to each  local $(n-1)$-stratum $x:v\to s$  a local $(n-1)$-stratum $x^p:=S_{n-1}^p(x):u\to s$. 
This defines 
 a map $f_p:  M_u\to  M_v$   and
  a group homomorphism $\phi_p:  G_u\to G_v$, given by the commuting diagrams  
\begin{align}\label{eq:fpdef}
\xymatrix{ \displaystyle \prod_{S_{n-1}^{loc}(u)\ni x':u\to s} \!\!\!\!\!\!\!\! M_s \ar[rd]_{\pi_{x^p}} \ar[rr]^{f_p} & & \displaystyle  \prod_{S_{n-1}^{loc}(v)\ni x: v\to s} \!\!\!\!\!\!\!\! M_s\ar[ld]^{\pi_{x}}\\
& M_s
}\qquad
\xymatrix{ \displaystyle \prod_{S_{n}^{loc}(u)\ni q':u\to t} \!\!\!\!\!\!\!\! G_t \ar[rd]_{\pi_{q^p}} \ar[rr]^{\phi_p} & & \displaystyle  \prod_{S_{n}^{loc}(v)\ni q: v\to t} \!\!\!\!\!\!\!\! G_r\ar[ld]^{\pi_{q}}\\
& G_t
}
\end{align}
 for all local $(n-1)$-strata $x: v\to s$ and $n$-strata $q: v\to t$.
As the maps  from Lemma \ref{lem:stratinj} satisfy 
\begin{align}\label{eq:snpid}
S_n^p\circ S^x_{n-1}=S_n^{x^p}: S_n^{loc}(s)\to S_n^{loc}(u)
\end{align}
for all local $(n-1)$-strata $x:v\to s$, 
 and by definition of the action in \eqref{eq:grpdactdef} the
  diagram
\begin{align}\label{eq:locstrat}
\xymatrix{
G_u\times M_u \ar[rrrr]^{\langle\phi_p,f_p\rangle} \ar[rrd]^{\quad\langle \pi_{l(x^p)}, \pi_{r(x^p)},\pi_{x^p} \rangle } \ar[ddd]_{\rhd} & & & &  G_v\times M_v \ar[ddd]^{\rhd} \ar[lld]_{\langle \pi_{lx},\pi_{rx},\pi_x\rangle\quad}\\
& & G_{L(s)}\times G_{R(s)}\times M_s \ar[d]^\rhd\\
& & M_s\\
M_u \ar[rrrr]_{f_p}\ar[rru]^{\pi_{x^p}} && & & M_v\ar[llu]_{\pi_x}
}
\end{align}
 commutes for all local $l$-strata $p: u\to v$, where $l(x^p)=S_n^{x^p}(l(s))$, $r(x^p)=S_n^{x^p}(r(s))$, $lx=S_n^x(l(s))$ and $rx=S_n^x(r(s))$. The lower triangle commutes by \eqref{eq:fpdef}, the quadrilaterals by \eqref{eq:grpdactdef} and the upper triangle by definition of $\phi_p$ and $f_p$ in \eqref{eq:fpdef} and due to identity \eqref{eq:snpid}, which implies $l(x^p)=(lx)^p$ and $r(x^p)=(rx)^p$.

Applying the functor $\Act:\Mod\to \Grpd$ from Lemma \ref{lem:actfunctset}  yields a functor $D_p=\Act(\phi_p,f_p): \mathcal D_u\to\mathcal D_v$.

\textbf{Boundary stratifications:}  For each boundary $k$-stratum $s$ with associated $(k+1)$-stratum $t=\iota_k(s)$ from \eqref{eq:embedbound}, we have the  bijections  $\iota_{kj}: S^{loc}_{j}(s)\to S^{loc}_{j+1}(t)$ from \eqref{eq:boundinj-locstrat}. This implies $M_s=M_t$ and $G_s=G_t$ as well as $f_q=f_p$ and  $\phi_q=\phi_p$ for  local boundary $l$-strata $p:s\to u$ with  associated local $(l+1)$-strata $q:t\to v$. 

Combining these statements and formulating them  in terms of the  graded  graph $Q^X$, the associated category $\maq^X$  and the categories and functors from Definition \ref{def:mod} and Lemma \ref{lem:actfunctset} gives the following.

\begin{lemma} \label{cor:deltabound} An assignment of classical defect data to a compact stratified $n$-manifold $X$  defines 
a functor\linebreak $E^X: \maq^X\to\Mod$ that assigns
\begin{compactitem}
\item to each $k$-stratum $u$ of $X$ the pair $(G_u,M_u)$ defined by \eqref{eq:groupoiddef} and \eqref{eq:grpdactdef},
\item to each local $(k+1)$-stratum $p:u\to v$ the pair $(\phi_p,f_p)$ given by \eqref{eq:locstrat},
\end{compactitem}
and a  functor $D^X=\Act E^X:\mathcal Q^X\to \Grpd$ that assigns
\begin{compactitem}
\item to each $k$-stratum $u$ of $X$ the action groupoid  $\mathcal D_u=\Act(G_u,M_u)=M_u\sslash  G_u$, 
\item to each local $(k+1)$-stratum $p:u\to v$ the functor $D_{p}=\Act(\phi_p,f_p): \mathcal D_u\to\mathcal D_v$.
\end{compactitem}
If $X$ has a boundary, then $E^{\partial X}=E^XI_\partial$ and $D^{\partial X}=D^XI_\partial$, where  $I_\partial:\mathcal Q^{\partial X}\to \mathcal Q^X$ is given in Lemma \ref{lem:inclusions are insertive}.
\end{lemma}

The full assignment of defect data for a  stratified $n$-manifold $X$  is  obtained by supplementing the classical defect data with representations of  action groupoids and intertwiners between them.

\begin{definition} \label{def:defectdata} Let $X$ be a compact stratified $n$-manifold. 
 An \textbf{assignment of defect data} to  $X$  assigns
\begin{compactitem}
\item to each $n$-stratum $t$  a finite group $G_t$,
\item to each $(n-1)$-stratum $s$  a finite $G_{L(s)}\times G_{R(s)}^{op}$-set $M_s$,
\item to each $(n-2)$-stratum $u$   a representation $\rho_u:\mathcal D_u\to \Vect_\C$,
\item to each $(n-3)$-stratum $v$ an intertwiner $\sigma_v: \rho_v^{t}\Rightarrow\rho_v^{s}$, where 
\begin{align}\label{eq:rhovdef}
\rho^s_v=\bigotimes_{s(e): v\to e}\, \rho_{e}D_{s(e)}: \mathcal D_v\to \Vect_\C    \qquad\qquad
 \rho_v^t=\bigotimes_{t(e): v\to e}\, \rho_eD_{t(e)}: \mathcal D_v\to \Vect_\C.
\end{align}
\end{compactitem}
\end{definition}
The representations $\rho_v^s: \mathcal D_v\to \Vect_\C$ in \eqref{eq:rhovdef} are constructed as follows. We consider all outgoing edge ends at $v$, or, equivalently, local 1-strata $s(e): v\to e$ associated with  starting ends of  edges $e$ at $v$. Applying  $D^X:\maq^X\to\Grpd$ associates to each of them a functor $D_{s(e)}: \mathcal D_v\to\mathcal D_e$. Post-composing it with the representation $\rho_e: \mathcal D_e\to \Vect_\C$ yields a representation $\rho_e D_{s(e)}:\mathcal D_v\to \Vect_\C$, and taking the tensor product of these representations yields the representation $\rho^s_v:\mathcal D_v\to \Vect_\C$. The construction of the representation $\rho^t_v:\mathcal D_v\to \Vect_\C$ is analogous but considers  the \emph{incoming} instead of the \emph{outgoing} edge ends at $v$.

Note that  an assignment of defect data to a compact stratified $n$-manifold $X$ defines an assignment of defect data to $\partial X$. This allows one to extend the  notion of a cobordism to stratified manifolds with defect data. Recall from Section \ref{subsec:orstrat} that embeddings and isomorphisms of stratified manifolds are required to respect the orientations of the manifolds and of all strata and the framings of  the codimension 2-strata. 
For stratified manifolds with defect data, we also require that they
respect the defect data:

\begin{definition} Let $\Sigma_0,\Sigma_1$ be stratified surfaces with  defect data. 
A \textbf{defect cobordism} from $\Sigma_0$ to $\Sigma_1$ is a compact stratified  3-manifold $M$ with defect data, together with embeddings of stratified manifolds with defect data $\overline \Sigma_0 \xrightarrow{\iota_0} M \xleftarrow{\iota_1} \Sigma_1$, {that induces an isomorphism of stratified manifolds $\langle \iota_0, \iota_1\rangle: \overline \Sigma_0\amalg \Sigma_1 \to \partial M$.} 
\end{definition}

\begin{definition}\label{def:coddef} The \textbf{defect cobordism category} $\mathrm{Cob}^{\mathrm{def}}_3$ is the symmetric monoidal category with
\begin{compactitem}
\item stratified surfaces with  defect data as objects,
\item equivalence classes of defect cobordisms $ \overline \Sigma_0\xrightarrow{\iota_0} M \xleftarrow{\iota_1} \Sigma_1$  as morphisms from $\Sigma_0$ to $\Sigma_1$, where  $(M,\iota_0,\iota_1)\sim (M',\iota'_0,\iota'_1)$ 
 if there is an isomorphism of stratified manifolds $\phi: M\to M'$ that respects {the defect data} 
 and makes  the following diagram commute
\begin{align}\label{eq:equivcob}\vcenter{
\xymatrix{ &  M\ar[dd]^\phi_\cong &\\
\overline \Sigma_0 \ar[ru]^{\iota_0} \ar[rd]_{\iota'_0} & & \Sigma_1 \ar[lu]_{\iota_1} \ar[ld]^{\iota'_1}\\
& M',
}}
\end{align}
\item composition of morphisms by gluing of stratified manifolds with defect data,
\item the identity morphism on $\Sigma$  given by  $\Sigma\xrightarrow{\iota_0} \Sigma\times[0,1]\xleftarrow{\iota_1} \Sigma$,
where {$\iota_j: \Sigma\to \Sigma\times [0,1]$, $x\mapsto (x,j)$},  $\Sigma\times[0,1]$ has a $(k+1)$-stratum $s\times[0,1]$ for each $k$-stratum $s$ of $\Sigma$, labeled by the same data and {with the  induced  framing}, if 
$s$ is a stratum of codimension 2.

\item the symmetric monoidal structure given by disjoint unions.
\end{compactitem}
\end{definition}

Here and in the following, $\overline \Sigma$ denotes the surface $\Sigma$ with the opposite orientation. In the following, we suppress this orientation reversal when denoting a cobordism by $\Sigma_0\xrightarrow{\iota_0} M\xleftarrow{\iota_1} \Sigma_1$ for better legibility.

Given this definition of the defect cobordism category, we define  an oriented  3d defect TQFT as a symmetric monoidal functor from the defect cobordism category into the category of finite-dimensional vector spaces.

\begin{definition} A 3d \textbf{defect TQFT} is a symmetric monoidal functor $\mathcal Z: \mathrm{Cob}_3^{\mathrm{def}}\to \Vect_\C$.
\end{definition}

We compare our definitions of the defect cobordism category and the resulting notion of a defect TQFT to other versions of  defect cobordism categories   
considered by Carqueville, Meusburger and Schaumann in \cite{CMS},  by Carqueville, Runkel and Schaumann  in \cite{CRS,CRS2, CRS3}, by Carqueville, Mulevi\v{c}ius, Runkel, Schaumann and Scherl \cite{CMR+} and by Carqueville and M\"uller \cite{CM}.
The first difference is that we work with stratified PL manifolds, while these publications consider 
 smooth strata in topological manifolds. We prefer to work in the PL setting, firstly because this is the setting used in \cite{BMS} which is the basis of these constructions. More importantly, the constructions in Section \ref{sec:fundgrpd}, are  simpler and more intuitive in the PL setting. Performing them in the smooth setting would require more work and add technical complications.

The other substantial difference to the defect cobordism category used in \cite{CMS,CRS,CRS2,CRS3, CMR+} is the saturation condition on stratifications in Definition \ref{def:homog}, which is not imposed in those publications. Omitting the saturation condition leads to the appearance of  defect points that are not contained in the closures of 1- or 2-strata and of defect lines that are not contained in the closures of  2-strata. This forces one to either artificially restrict the defect data associated to those points and lines or to 
introduce additional framings by hand, see for instance the discussion on defect lines in \cite[Remark 5.9 (ii)]{CRS}.  

We feel that it is more natural  to exclude such defect lines and points. Firstly, because this avoids the mentioned issues with the framings and restrictions on the defect data. Secondly, the algebraic data on the higher-dimensional strata that contain a given stratum in their closures define the source and target of the higher categorical data on this stratum.  Finally, forbidding such strata does not restrict generality, as such strata can be realised by labelling the incident  strata with \emph{transparent defect data}, see Definition \ref{def:transpdef}.

Our framings from Definition \ref{def:framed} resemble the
$*$-decorations considered in \cite[Def.~5.4]{CRS}, but play a different role. As explained in Section \ref{subsec:orstrat}, we need the framing  to ensure that the bundles of local strata of isolated loops are trivial. As the gluing of cobordisms  respects the framing, 1-strata cannot glue to isolated loops with non-trivial bundles of local strata. 
In contrast, the $*$-decorations in \cite{CRS} are used to consistently assign algebraic data to 1-strata \cite[Remark 5.3]{CRS}. It is then shown that the resulting defect TQFT does not depend on the choice of these $*$-decorations, due to the algebraic properties of the defect data. We do not require the framing to assign defect data. However,  we do not know if our formalism, in particular Theorem \ref{th:classcobfunc}, extends to isolated loops with non-trivial bundles of local strata.

\section{Gauge configurations and gauge transformations}
\label{sec:gauge}

In this section, we associate to each stratified surface and each stratified cobordism with classical defect data a groupoid, the gauge groupoid of the surface or cobordism.  
Its objects are called gauge configurations and its morphisms gauge transformations. 
We show that this groupoid is an action groupoid. 

This groupoid is constructed as follows. We describe each compact stratified $n$-manifold $X$ by a graded graph $Q^X$ as in Section \ref{sec:grgraph}
 and   the associated  category $\maq^X$. We consider the functor $D^X:\maq^X\to \Grpd$ from Lemma \ref{cor:deltabound} 
 that encodes  the algebraic content of the theory, the classical defect data. 
In  Section \ref{sec:fundgrpd} we construct another functor, $T^X:\maq^X\to \Grpd$, that encodes the topological content of the theory,   fundamental groupoids associated to strata, and functors between them associated to local strata. 
Gauge configurations and gauge transformations  assign algebraic content to the topological content  and are given by  natural transformations from the functor $T^X$ to the functor $D^X$ and modifications between them.  

To implement this, we  introduce  in Section \ref{subsec:pack} a category whose objects are triples of a graded graph $Q$ and functors $D,T:\maq\to \Grpd$ and whose morphisms are given by  graph maps and certain natural transformations. The gauge groupoid of a stratified $n$-manifold $X$ is then defined and analysed in Section \ref{subsec:gaugegrpd}.  Sections \ref{subsec:redgaugegrpd} and \ref{subsec:geomdescript}  give a description of the gauge groupoids in terms of fundamental groupoids with basepoints  and a fully combinatorial description for fine stratifications.

\subsection{Fundamental groupoids associated with  stratifications}
\label{sec:fundgrpd}

We now associate to each compact  stratified $n$-manifold $X$ with associated graded graph $Q^X$ a functor $B^X: \maq^X\to \Top$ and a functor $T^X=\Pi_1B^X:Q^X\to \Grpd$. As these functors must also take into account the inclusions of lower-dimensional skeleta into  higher-dimensional ones, it is not possible to work directly with the  strata. Instead, we  associate to each stratum $t$  a  compact {PL manifold $\hat t$, whose interior coincides with $t$}, and to each local stratum $q:t\to u$ a {PL} map $\iota_q:\hat t\to\hat u$ as follows.

 Recall from Section \ref{subsec:strats} that each compact stratified $n$-manifold $X$  has a  triangulation $T$  that contains all skeleta of $X$ as subcomplexes. We fix such a triangulation $T$. When we consider  neighbourhoods $U_x$ as in Definition \ref{def:homog}, we always assume that they are chosen small enough such that each $k$-simplex of $T$ intersects at most one $k$-stratum of $U_x$.

\textbf{Spaces assigned to strata:} The spaces $\hat t$ assigned to $k$-strata $t$ are obtained by first cutting the $k$-skeleton of $X$ along the $(k-1)$-skeleton. 
More precisely, to construct  $\hat t$, we  
 glue all $k$-simplexes $\tau$ of $T$ with $\tau\cap t\neq \emptyset$ along those $(k-1)$-faces $\rho$ with $\rho\cap t\neq \emptyset$. 
This yields  a compact PL manifold with boundary $\hat t$ whose interior is PL homeomorphic to $t$. It is the geometric realisation of a finite simplicial complex.

Sending each $k$-simplex $\hat \tau$ of $\hat t$ to the associated $k$-simplex $\tau$ of $t$ yields a continuous map $f_t: \hat t\to\bar t$ that is injective in the interior of the simplexes. 
Neither the spaces $\hat t$ nor the maps $f_t$ depend on  $T$.

\textbf{Maps assigned to local strata:}
This construction assigns to each local $l$-stratum $q:s\to t$  a {PL map} $\iota_q:\hat s\to \hat t$ as follows.  
For any $k$-simplex $\sigma$ of $\bar s$ and any  $x\in \mathring \sigma$ we can choose a neighbourhood $U_x$ as in Definition \ref{def:homog}, such that each $l$-simplex $\tau$ of $\bar t$ intersects at most one $l$-stratum of $U_x$. Let  $\tau_1,\ldots, \tau_r$ be the $l$-simplexes of $\bar t$ that contain $\sigma$ as a $k$-face and satisfy $q'\cap \tau_j\neq \emptyset$ for 
the unique
$l$-stratum $q'$ of $U_x$ representing $q$. 
Denote by $\hat\tau_1,\ldots,\hat \tau_r$ the associated $l$-simplexes of $\hat t$ with $f_t(\hat \tau_j)=\tau_j$. The simplex $\sigma$ corresponds to a unique $k$-simplex $\hat \sigma$ of $\hat s$, and all $k$-faces $\hat\rho$  of $\hat \tau_1,\ldots,\hat \tau_m$ with $f_t(\hat\rho)=\sigma$  correspond to a single $k$-simplex $\hat\sigma'$ of $\hat t$. 
The assignments $\hat \sigma\to\hat \sigma'$ then define a PL map $\iota_q:\hat s\to\hat t$
that is independent of  $T$ and such that the following diagram commutes for
the inclusions $\iota_{st}:\bar s\to\bar t$ 
\begin{align}\label{eq:commtop}
\vcenter{\xymatrix{ \hat s\ar[d]_{f_s} \ar[r]^{\iota_q}  & \hat t \ar[d]^{f_t} \\
\bar s \ar[r]_{\iota_{st}} & \bar t.
}}
\end{align}

\textbf{Boundary strata:}
 Let $s$ be a  $k$-stratum of $\partial X$ and
 $t=\iota_k(s)$ the
 associated $(k+1)$-stratum of $X$ from \eqref{eq:embedbound}. Then for each $k$-simplex $\sigma$
in  $\partial X$ with $\sigma\cap s\neq \emptyset$ there is a unique 
$(k+1)$-simplex 
$\tau$ in $X$ with $\tau\cap t\neq \emptyset$ and $\tau\cap \partial X=\sigma$. Denote by $\hat \sigma$ and $\hat \tau$ the associated simplexes in $\hat s$ and $\hat t$. Sending $\hat \sigma$  to the corresponding  $k$-face of $\hat \tau$ defines a PL embedding 
 $\iota_{s\partial}: \hat s\to \hat t$ {that} does not depend on $T$.

For a local $l$-stratum $p: s\to u$ of $s$ and the associated local $(l+1)$-stratum $q: t\to v$ at $t$ from \eqref{eq:boundinj-locstrat} this yields  the 
following commuting diagram, where the arrows in the lower quadrilateral are the obvious inclusions
\begin{align}\label{eq:stratboundemb}
\xymatrix{ 
\hat u \ar[d]_{f_u} \ar@/^4ex/[rrr]^{\iota_{u\partial}} & \hat s \ar[l]^{\iota_p} \ar[d]_{f_s} \ar[r]_{\iota_{s\partial}} & \hat t\ar[d]^{f_t} \ar[r]_{\iota_q} & \hat v \ar[d]^{f_v}\\
\bar u \ar@/_4ex/[rrr]_{\iota_{uv}} & \bar s \ar[l]_{\iota_{su}}\ar[r]^{\iota_{st}} & \bar t \ar[r]^{\iota_{tv}} & \bar v.
}
\end{align}
We now describe these constructions in terms of the graded graphs $Q^X$ and $Q^{\partial X}$ and the associated categories $\maq^X$ and $\maq^{\partial X}$. 
The spaces $\hat s$ associated to the strata and the maps $\iota_q: \hat s\to \hat t$ associated to local strata $q:s\to t$ define a functor $B^X:\mathcal Q^X\to \Top$.
Likewise, the corresponding data for the boundary defines  a functor $B^{\partial X}:\maq^{\partial X}\to \Top$.  
The upper quadrilateral in the commuting diagram \eqref{eq:stratboundemb} states that the maps $\iota_{s\partial}: \hat s\to \hat t$  are natural with respect to the morphisms in $\maq^{\partial X}$. 
Denoting by $I_\partial: \mathcal Q^{\partial X}\to \mathcal Q^X$  the functor from Lemma \ref{lem:inclusions are insertive} induced by the insertion $i_\partial: Q^{\partial X}\to Q^X$,  we then obtain the following proposition.

\begin{proposition}\label{prop:tfunctor}  Let $X$ be a compact stratified $n$-manifold. 
There is a functor $B^X:\mathcal Q^X\to \Top$ that assigns
\begin{compactitem}
\item to a $k$-stratum $s$ of $X$ the space $\hat s$,
\item to a local $(k+1)$-stratum $q:s\to t$ the map $\iota_q:\hat s\to\hat t$ from \eqref{eq:commtop}.
\end{compactitem}
Post-composition with  $\Pi_1: \Top\to\Grpd$ yields the functor $T^X=\Pi_1B^X:\mathcal Q^X\to \Grpd$ that assigns
\begin{compactitem}
\item to a $k$-stratum $s$ the groupoid $\Pi_1(\hat s)$,
\item to a local $(k+1)$-stratum $q: s\to t$  the functor $\Pi_1(\iota_q): \Pi_1(\hat s)\to \Pi_1(\hat t)$.
\end{compactitem}
The maps from \eqref{eq:stratboundemb} define natural transformations $\beta^X: B^{\partial X}\Rightarrow B^XI_\partial$ and $\tau^X=\Pi_1\beta^X: T^{\partial X}\Rightarrow T^XI_\partial$
with 
\begin{align}\label{eq:tnatdef}
&\beta^X_s= \iota_{s\partial}: \hat s \to  \hat t, &
&\tau^X_s=\Pi_1(\iota_{s\partial}): \Pi_1(\hat s)\to \Pi_1(\hat t).
\end{align}
\end{proposition}

\subsection{Categorical description}
\label{subsec:pack}

We  now introduce a category that encodes  the graded graphs $Q^X$  for compact stratified $n$-manifolds $X$ and the associated functors $D^X,T^X:\maq^X\to \Grpd$ from Lemma \ref{cor:deltabound} and Proposition \ref{prop:tfunctor}. This category must take into account the relation between the data for  $X$  and  the data for its boundary $\partial X$. 
While the functors $D^X$ and $D^{\partial X}=D^XI_\partial$ are directly related by the  functor $I_\partial:\maq^{\partial X}\to \maq^X$ from Lemma \ref{cor:deltabound}, the functors $T^X$ and $T^{\partial X}$  involve an additional  natural transformation $\tau^X: T^{\partial X}\Rightarrow T^XI_\partial$ given in \eqref{eq:tnatdef}. Thus, the morphisms in this category are  given by graph maps, together with  natural transformations.

\begin{definition} \label{def:pack}We denote by $\pack$ the category whose 
\begin{compactitem}
\item objects are triples $(Q, T, D)$ of a graph $Q$  and functors $T,D:\mathcal Q\to \Grpd$,
\item  morphisms from $(Q, T, D)$  to $(Q',T',D')$ are pairs $(f, \tau)$ of a graph map $f: Q\to Q'$ with associated functor $F:\mathcal Q\to\mathcal Q'$ satisfying  $D=D'F$, and a natural transformation $\tau\colon T\Rightarrow T'F$,
\item composition of morphisms is given by $(f',\tau')\circ (f,\tau)=\big (f'\circ f, (\tau' F)\circ \tau\big)$.
\end{compactitem}
\end{definition}

Inserting  functors $D,T:\maq\to \Grpd$  into the functor $\GRPd(-,-): \Grpd^{op}\times\Grpd\to \Grpd$ yields a functor $\GRPd(T,D):\maq^{op}\times\maq\to \Grpd$. 
By taking the ends of such functors we then obtain the following proposition.
\begin{proposition}\label{prop:gconfiggrp}There is a functor $G: \pack^{op}\to \Grpd$ that assigns
\begin{compactitem}
\item to an object $\mathcal P=(Q,T,D)$ the end of the functor $\GRPd(T,D): \mathcal Q^{op}\times \mathcal Q\to \Grpd$
$$
G(\mathcal P)=\int_{v\in Q_0} \GRPd\big(T(v),D(v)\big),
$$
\item to a morphism $\phi=(f,\tau):  \mathcal P\to\mathcal P'$ the induced morphism $G(\phi):  G(\mathcal P')\to G(\mathcal P)$ between ends.
\end{compactitem}
\end{proposition}

\begin{proof} We write $D(u)^{T(v)}=\GRPd\big(T(v),D(u)\big)$ 
and denote the morphisms that characterise the ends by 
$$e_v: G(\mathcal P)\to D(v)^{T(v)}\qquad \qquad e'_w: G(\mathcal P')\to D'(w)^{T'(w)}.$$  Then the functor $G(\phi): G(\mathcal P')\to G(\mathcal P)$ is induced by the universal property of the end via the following diagram that commutes for every morphism $h: u\to v$ in $\mathcal Q$
\begin{align}\label{eq:commend}\begin{gathered}
\xymatrix@C=40pt{
G(\mathcal P') \ar@{-->}[rd]^{G(\phi)}\ar[ddd]_{e'_{f(u)}} \ar[rrr]^{e'_{f(v)}}& & & D(v)^{T'(f(v))}\ar[ddd]^{D(1_{v})^{T'(f(h))}} \ar[ld]_{ D(1_v)^{(\tau_v)}}\\
& G(\mathcal P)\ar[d]_{e_u} \ar[r]^{e_v}& D(v)^{T(v)} \ar[d]^{D(1_v)^{T(h)}}\\
& D(u)^{T(u)} \ar[r]_{D(h)^{T(1_u)}}& D(v)^{T(u)}\\
D(u)^{T'(f(u))}\ar[rrr]_{D(h)^{T'(1_{f(u)})}}   \ar[ru]^{ D(1_u)^{(\tau_u)}} & & & D(v)^{T'(f(u))}. \ar[lu]_{D( 1_v)^{(\tau_u)}}
}
\end{gathered}
\end{align}
Functoriality follows  by considering this  for composites in $\pack$ and for $(\id,\id): (Q,T,D)\to (Q,T,D)$.
\end{proof}

\begin{remark}\label{rem:gfunc} The functor $G(\phi): G(\mathcal P')\to G(\mathcal P)$  for a morphism $\phi=(f,\tau):(Q,T,D) \to (Q',T',D')$ from diagram \eqref{eq:commend}
can be described equivalently as the  following composite
\begin{align*}
    &\int_{x' \in \maq'} \GRPd\big(T'(x'),D'(x')\big)
    \xrightarrow{\End_{\maq'}\big(\GRPd(\hat \tau,D')\big)}\\
     &\int_{x'\in \maq'} \GRPd\big(\Lan_F T(x'), D'(x')\big)
    \cong \int_{x \in \maq} \GRPd\big(T(x),D'F(x)\big)= \int_{x \in \maq} \GRPd\big(T(x),D(x)\big).  
\end{align*}
Here, $\hat{\tau}: \Lan_F T\Rightarrow T'$ is induced by 
$\tau: T\Rightarrow T' F$ via the universal property of the left Kan extension, and we applied Lemma \ref{lem:enriched Kan prop} for $\V=\Grpd$. 
\end{remark}

We  give a more explicit construction of the groupoid $G(\mathcal{P})$. Recall from Section \ref{sec:groupoidbasic} the interval groupoid $I$. Taking its tensor product with the identity functor with respect to the cartesian monoidal structure of $\Grpd$ defines
 a functor $(-)\times I : \Grpd \to \Grpd$. Its composites with  a functor $T: \maq \to \Grpd$ and with a natural transformation $\tau: T\Rightarrow T'$ are   denoted  $T\times I=(-\times I)T:\maq\to\Grpd$ and $\tau\times I=(-\times I)\tau: T\times I\Rightarrow T'\times I$. The inclusions $\iota^0_\G,\iota^1_\G:\G\to\G\times I$ for each groupoid $\G$ define natural transformations $\iota^0,\iota^1: T\Rightarrow T\times I$.  Recall also from Section \ref{sec:groupoidbasic} that functors $F: \mathcal G\times I\to \mathcal H$ correspond to  natural transformations between the functors $F^0=F\iota^0,F^1=F\iota^1: \mathcal G\to\mathcal H$. We thus denote them by the same letters.

\begin{lemma} \label{lem:gaugesimplify}$\quad$ 
\begin{compactenum}
\item The groupoid $G(Q,T,D)$ for an object $(Q,T,D)$ from Proposition \ref{prop:gconfiggrp} has 
\begin{compactitem}
\item as objects natural transformations, $A\colon T\Rightarrow D$,
\item as morphisms $\gamma: A^0\Rrightarrow A^1$ natural transformations $\gamma: T\times I\Rightarrow D$ with $\gamma\circ \iota^0=A^0$ and $\gamma\circ \iota^1=A^1$ 
 \begin{align}\label{eq:gaugetrafodiag0}\begin{gathered}
 \xymatrix{ T\ar@{=>}[rd]_{A^0}\ar@{=>}[r]^{\iota^0} & T\times I \ar@{=>}[d]^\gamma & \ar@{=>}[l]_{\quad \iota^1} T \ar@{=>}[ld]^{A^1}\\
 & D.
 }
 \end{gathered}
 \end{align}
 \end{compactitem}
 \item The functor $G(f,\tau): G(Q',T',D')\to G(Q,T,D)$ 
  for a morphism $(f,\tau): (Q,T,D)\to (Q',T',D')$ sends
 \begin{compactitem}
\item  a natural transformation $A': T'\Rightarrow D'$ 
to the natural transformation $ (A'F)\circ \tau: T\Rightarrow D$, 
\item a natural transformation $\gamma': T'\times I\Rightarrow D'$   
to $\gamma=(\gamma'F)\circ (\tau\times I): T\times I\Rightarrow D$. 
\end{compactitem}
 \end{compactenum}
\end{lemma}

\begin{proof}  1.~The functor ${\GRPd(T,D)}: \mathcal Q^{op}\times\mathcal Q\to \Grpd$   in Proposition \ref{prop:gconfiggrp} assigns to a {pair of  vertices $(v,v')\in Q_0\times Q_0$ the groupoid $\Grpd\big(T(v), D(v')\big)=D(v')^{T(v)}$} with functors $A_{v,v'}: T(v)\to D(v')$ as objects and natural transformations $\gamma_{v,v'}: A^0_{v,v'}\Rightarrow A^1_{v,v'}$ as morphisms from $A^0_{v,v'}$ to $A^1_{v,v'}$. 

Its end $\End_{\maq}\GRPd(T,D)$  is the subgroupoid of the product $\prod_{v\in Q_0} D(v)^{T(v)}$ that equalises post-composition with $D(h)$ and pre-composition with $T(h)$ for all edges $h: u\to v$ in $Q$. As $\Ob: \Grpd \to \Set$ preserves limits, objects of $\End_{\maq}\GRPd(T,D)$
are families $(A_v)_{v\in Q_0}$ of functors $A_v: T(v)\to D(v)$ such that the diagram
\begin{align}\label{eq:comm0}
\vcenter{\xymatrix{ T(u) \ar[d]_{A_u}\ar[r]^{T(h)} &T(v) \ar[d]^{A_v}\\
 D(u)\ar[r]_{D(h)} & D(v)
}}
\end{align}
commutes for all edges $h:u\to v$ in $Q$. This is  a  natural transformation $A: T\Rightarrow D$.

 Morphisms from $A^0: T\Rightarrow D$ to $A^1: T\Rightarrow D$ are families $(\gamma_v)_{v\in Q^X_0}$ of natural transformations $\gamma_v: A^0_v\Rightarrow A^1_v$ such that $D(h) \gamma_u=\gamma_v  T(h)$ for all edges $h: u\to v$. This is 
 a natural transformation $\gamma: T\times  I\Rightarrow D$  such that diagram \eqref{eq:gaugetrafodiag0} commutes.
 
2.~The morphisms $e_v: G(Q,T,D)\to D(v)^{T(v)}$ that characterise the end assign to a natural transformation $A: T\Rightarrow D$ the functor $A_v: T(v)\to D(v)$. 
The commuting rectangles on top and on the left in diagram \eqref{eq:commend} then state that $(A'F)_v\circ  \tau_v=A_v$ and $\gamma'_{f(v)}\circ (\tau_v\times\id_I)=\gamma_v$  for $A=G(f,\tau)A'$.
\end{proof}

\begin{remark} We will use  the concrete formulas from the proof of Lemma \ref{lem:gaugesimplify}. Note  that the objects of the groupoid can  be obtained more easily using that  the functor $\Ob: \Grpd \to \Set$ is a right adjoint:  \begin{align*}\Ob\int_{x \in \maq} \GRPd\big(T(x),D(x)\big)\cong\int_{x \in \maq} \Ob\, \GRPd\big(T(x),D(x)\big)=\int_{x \in \maq} \hom_\Grpd\big(T(x),D(x)\big)\cong{\rm Nat}(T,D). \end{align*}\end{remark}

\subsection{The groupoid of gauge configurations and gauge transformations}
\label{subsec:gaugegrpd}

A gauge configuration on a compact stratified $n$-manifold $X$ with defect data should relate the \emph{topological content} of the stratification to its \emph{algebraic} content. 
The stratification is encoded in the graph $Q^X$ from Definition \ref{def:embquiv}.
The topological content is encoded in the fundamental groupoids assigned to different strata and the functors between them. It is given by the functor $T^X: \mathcal Q^X\to\Grpd$ from Proposition \ref{prop:tfunctor}. The algebraic content is given by the defect data for the strata and encoded in the functor $D^X:\mathcal Q^X\to \Grpd$ from Lemma \ref{cor:deltabound}. 
The  groupoid  that relates this data was constructed in Proposition \ref{prop:gconfiggrp} and Lemma \ref{lem:gaugesimplify}.

\begin{definition}\label{def:gaugegrpd} Let $X$ be a compact stratified $n$-manifold  with  classical defect data. 
The \textbf{gauge groupoid}  on $X$, with  \textbf{gauge configurations} as objects and \textbf{gauge transformations} as morphisms,
 is the groupoid
$$G\big(Q^X,T^X,D^X\big):=\int_{s \in \maq^X} \GRPd\big(T^X(s),D^X(s)\big)=\int_{s \in \maq^X} \mathcal{D}_s^{\Pi_1 (\hat{s})}. $$
For each stratum $t$ of $X$  the end defines a functor
\begin{align*}
P_t: \int_{s \in \maq^X} \mathcal{D}_s^{\Pi_1 (\hat{s})} \to \mathcal{D}_t^{\Pi_1 (\hat{t})}.
\end{align*}
\end{definition}

It follows 
from  Definition \ref{def:gaugegrpd}  that
gauge configurations and  transformations on a compact stratified $n$-manifold $X$ restrict to gauge configurations  and  transformations on its boundary $\partial X$.
 By Proposition \ref{cor:deltabound} the  defect data on $X$ defines the defect data of $\partial X$ and  a  functor 
$D^{\partial X}=D^XI_\partial: \mathcal Q^{\partial X}\to \Grpd$, where   $I_\partial: \mathcal Q^{\partial X}\to \mathcal Q^X$ is the functor from Lemma \ref{lem:inclusions are insertive}.  
Proposition \ref{prop:tfunctor} defines a natural transformation $\tau^X: T^{\partial X}\Rightarrow T^XI_\partial$. 
 Proposition \ref{prop:gconfiggrp} and Lemma \ref{lem:gaugesimplify} then imply the following,  in the notation introduced before Lemma \ref{lem:gaugesimplify}.
 
 \begin{corollary}\label{prop:projfunc} Let $X$ be a stratified $n$-manifold with  classical defect data. There is a restriction functor 
  $P_\partial: G(Q^X,T^X,D^X)\to G(Q^{\partial X}, T^{\partial X}, D^{\partial X})$ 
 that assigns
\begin{compactitem}
\item to a gauge configuration $A^X: T^X\Rightarrow D^X$ on $X$ the boundary gauge configuration 
$$A^{\partial X}=(A^XI_\partial)\circ \tau^X: T^{\partial X}\Rightarrow D^{\partial X},$$
\item to a gauge transformation $\gamma^X: A^X\Rrightarrow A'^X$ on $ X$  the boundary gauge transformation
$$\gamma^{\partial X}=\gamma^XI_\partial\circ (\tau^X\times I) : A^{\partial X}\Rrightarrow A'^{\partial X}.$$
\end{compactitem}
 \end{corollary}

\begin{remark} \label{rem:boundproj} As in Remark \ref{rem:gfunc},
the functor $P_\partial$  
from Corollary \ref{prop:projfunc} can also be constructed  with ends as the following composite
\begin{align*} G(Q^X,T^X,D^X)
& =    \int_{t \in \maq^X} \!\!\!\! \GRPd\big(T^X(t),D^X(t)\big)
    \xrightarrow{\End_{\maq^X}\big(\GRPd(\hat{\tau}_X, D^X)\big)} \int_{t \in \maq^X}  \!\!\!\!  \GRPd\big(\Lan_{I_\partial} T^{\partial X} (t) ,D^X (t) \big)\\
    & \cong \int_{s \in \maq^{\partial X}}  \!\!\!\!  \GRPd\big(T^{\partial X}(s),D^{ X}I_\partial (s)\big)
     = \int_{s \in \maq^{\partial X}} \!\!\!\!  \GRPd\big(T^{\partial X}(s),D^{ \partial X} (s))=G(Q^{\partial X}, T^{\partial X}, D^{\partial X}\big),
    \end{align*}
where we  applied Lemma~\ref{lem:enriched Kan prop} for  $\V=\Grpd$ and $\hat{\tau}^X: \Lan_{I_\partial} T^{\partial X}\Rightarrow T^{X}$ is induced by  $\tau^X: T^{\partial X} \Rightarrow T^{X}  I_{\partial}$ via the universal property of the left Kan extension.
\end{remark}

We now apply Lemma \ref{lem:gaugesimplify} and the characterisation of the functors $D^X:\maq^X\to\Grpd$   from Lemma \ref{cor:deltabound} to 
 describe the gauge groupoid more concretely.

\begin{lemma}\label{lem:gexplicit} Let $X$ be a compact stratified $n$-manifold with classical defect data.
\begin{compactenum}
\item A gauge configuration on $X$  is an assignment
of  a map $h_u:\hat u\to M_u$ and a functor $\psi_u:\Pi_1(\hat u)\to \bullet\sslash G_u$ to each $k$-stratum $u$ such that for all paths $\delta:y\to y'$ in $\hat u$
\begin{align}
\label{eq:funcmapcomp}
 h_u(y')=\psi_u([\delta])\rhd h_u(y),
\end{align}
and the following diagrams commute for all local $(k+1)$-strata $q:u\to v$
\begin{align}\label{eq:commdiaghighstrata}\vcenter{
\xymatrix{
\hat u\ar[d]_{\iota_q} \ar[r]^{h_u} & M_u \ar[d]^{f_q}\\
\hat v \ar[r]_{h_v} & M_v
}}
\qquad \textrm{ and } \qquad\vcenter{
\xymatrix{
\Pi_1(\hat u)\ar[d]_{\Pi_1(\iota_q)} \ar[r]^{\psi_u} & \bullet\sslash G_u \ar[d]^{\phi_q}\\
\Pi_1(\hat v) \ar[r]_{\psi_v} & \bullet\sslash G_v.
}}
\end{align}

\item A gauge transformation $\gamma:A=(h,\psi)\Rrightarrow A'=(h',\psi')$ is an assignment of a map $\gamma_u:\hat u\to G_u$ to each $k$-stratum $u$ such that 
for all $y,y'\in \hat u$ and paths $\delta:y\to y'$ in $\hat u$
\begin{align}\label{eq:natconc}
h'_u(y)=\gamma_u(y)\rhd h_u(y)\qquad \psi'_u([\delta])=\gamma_u(y')\cdot \psi_u([\delta])\cdot \gamma_u(y)^\inv
\end{align}
and the following diagram commutes for all local $(k+1)$-strata $q:u\to v$
\begin{align}\label{eq:eq:condgtrafonat}
\vcenter{\xymatrix{ \hat u\ar[d]_{\iota_q} \ar[r]^{\gamma_u} & G_u \ar[d]^{\phi_q}\\
\hat v \ar[r]_{\gamma_v} & G_v.
}}
\end{align}
\end{compactenum}
In particular, the groupoid   $G(\maq^X,T^X,D^X)=\mathcal A^X\sslash\mathcal G^X$ is an action groupoid, where $\mathcal A^X$ is the set of gauge transformations, and $\mathcal G^X$ the \textbf{group of gauge transformations}, defined in the proof. 
\end{lemma}

\begin{proof}
1.~By Lemma \ref{lem:gaugesimplify} a gauge configuration $A:T^X\Rightarrow D^X$ is an assignment of a functor $A_u:\Pi_1(\hat u)\to \mathcal D_u$ to each stratum $u$ of $X$ such that the following diagram commutes for all local strata $q: u\to v$ 
\begin{align}\label{eq:natgfield}
\xymatrix{\Pi_1(\hat u) \ar[d]_{\Pi_1(\iota_q)} \ar[r]^{A_u} & \mathcal D_u\ar[d]^{D_q}\\
\Pi_1(\hat v) \ar[r]_{A_v} & \mathcal D_v.
}
\end{align}
Here, $\iota_q:\hat u\to\hat v$ is the map from \eqref{eq:commtop},  $\mathcal D_u=\Act(G_u,M_u)=M_u\sslash G_u$ is the action groupoid given by \eqref{eq:groupoiddef} and $D_q=\Act(\phi_q,f_q):\mathcal D_u\to\mathcal D_v$ is the  functor defined by \eqref{eq:locstrat} and Lemma \ref{cor:deltabound}.  

As  $\mathcal D_u=M_u\sslash G_u$ is an action groupoid,
 a functor $A_u:\Pi_1(\hat u)\to \mathcal D_u$ is equivalent to the choice of a map $h_u: \hat u\to M_u$ and a functor $\psi_u: \Pi_1(\hat u)\to \bullet\sslash G_u$ satisfying \eqref{eq:funcmapcomp}.  Condition \eqref{eq:natgfield} translates into  \eqref{eq:commdiaghighstrata}. 

2.~By the proof of Lemma \ref{lem:gaugesimplify}, a gauge transformation $\gamma:A=(h,\psi)\Rrightarrow A'=(h',\psi')$ is an assignment of a natural transformation $\gamma_u: A_u\Rightarrow A'_u$ to each $k$-stratum $u$ of $x$ such that
 for all local $(k+1)$-strata $q:u\to v$
\begin{align}\label{eq:condgtrafonat0}
D_q\gamma_u=\gamma_v \Pi_1(\iota_q).
\end{align}
A natural transformation $\gamma_u: A_u\Rightarrow A'_u$ between  functors $A_u,A'_u: \Pi_1(\hat u)\to M_u\sslash G_u$  is  equivalent to an assignment of a map $\gamma_u:\hat u\to G_u$ to each $k$-stratum $u$ that satisfies \eqref{eq:natconc}. Condition \eqref{eq:condgtrafonat0}  translates into  \eqref{eq:eq:condgtrafonat}.

3.~It is clear that the gauge configurations form a set $\mathcal A^X$. It  follows from 2.~that the gauge transformations form a group $\mathcal G^X$, with the group structure given by the pointwise multiplication of the maps $\gamma_u:\hat u\to G_u$. The action of $\mathcal G^X$ on $\mathcal A^X$ is given by \eqref{eq:natconc}, and this implies $G(Q^X,T^X,D^X)=\mathcal A^X\sslash \mathcal G^X$. 
More explicitly, the group of gauge transformations is $\mathcal{G}^X=\int_{t \in \maq^X} (G_t)^{\hat{t}}$, where $(G_t)^{\hat t}$ denotes the group of maps $\gamma_t:\hat t\to G_t$.
\end{proof}

We  now show that gauge configurations and gauge transformations for a compact stratified $n$-manifold $X$ with classical defect data  are determined uniquely by their components for strata of dimension $\geq n-1$.
For this we relate local $n$-strata $q: u\to t$ and  $(n-1)$-strata $x:u\to s$ at a $k$-stratum $u$ to paths in  $Q^X$. 

Recall from Lemma \ref{lem:repsequenceexists} that each local $l$-stratum $q: u\to v$ can be represented by a sequence of local strata, each of which  raises the dimension by one. The local strata in the sequence correspond to  edges in $Q^X$ and any such representation to a path in $Q^X$. 
Recall from Definition \ref{def:repoflocstrat} that such a representing sequence is characterised by the condition that the  maps $S^q_j: S^{loc}_j(v)\to S^{loc}_j(u)$
are given as the composites of the corresponding maps for the local strata in the sequence. 

By Lemma \ref{lem:gexplicit},  for each local stratum  $q:u\to v$ \emph{that raises degree by one},   the components of  gauge configurations and gauge transformations for $u$ and $v$   are related by the diagrams \eqref{eq:commdiaghighstrata} and \eqref{eq:eq:condgtrafonat}, which involve the maps $\iota_q:\hat u\to\hat v$ from \eqref{eq:commtop} and the maps $f_q: M_u\to M_v$ and the group homomorphisms $\phi_q: G_u\to G_v$ from \eqref{eq:locstrat}.  
We thus need to relate the maps $\iota_r$, $f_r$ and group homomorphisms $\phi_r$ for a local stratum $r$ to the  ones for local strata in its representing sequences.

\begin{lemma}\label{lem:iota-agrees} Let $s$ be a $j$-stratum of a compact stratified $n$-manifold $X$,  $p:s\to t$  a  local $k$-stratum  and  $q: t\to u$ and $r:s\to u$ local $l$-strata such that $s\xrightarrow{p} t \xrightarrow{q} u$ represents $r: s\to u$. 
Then the associated maps from  \eqref{eq:fpdef} and \eqref{eq:commtop}  satisfy  
$$
\iota_r=\iota_q\circ \iota_p: \hat s\to \hat u, \qquad f_r=f_q\circ f_p: M_s\to M_u, \qquad \phi_r=\phi_q\circ \phi_p: G_s\to G_u.
$$
\end{lemma}

\begin{proof} The second and third identity follow directly from the definition of the maps $f_x: M_v\to M_w$ and the group homomorphisms $\phi_x: G_v\to G_w$  for a local stratum $x: v\to w$ in \eqref{eq:fpdef}, from \eqref{eq:locstrat} and from the identity $S_j^r=S_j^p\circ S_j^q$, which holds by definition of a representing sequence of local strata.

To see that the first identity holds, consider a triangulation $T$ of $X$ that contains all skeleta of the stratification as subcomplexes. 
For a $j$-simplex $\hat \sigma$ of $\hat s$,  $k$-simplex $\hat \tau$ of $\hat t$ and  $l$-simplex $\hat \chi$ of $\hat u$, we denote by $\sigma=f_s(\hat\sigma)$, by $\tau=f_t(\hat\tau)$ and $\chi=f_u(\hat\chi)$ the associated simplexes in $\bar s$, $\bar t$ and $\bar u$.

Then
for each $j$-simplex $\hat\sigma$ of $\hat s$  and each  $k$-simplex $\hat\tau$  in $\hat t$ such that $\tau$ contains $\sigma$ as a $j$-face and intersects a representative of $p$ at a point  $x\in \mathring \sigma$, the image $\iota_p(\hat\sigma)$ is the corresponding $j$-face of $\hat \tau$. Recall that all those $j$-faces correspond to a single $j$-simplex in $\hat t$. 
Likewise,  for each $k$-simplex $\hat\tau$ of $\hat t$  and each  $l$-simplex $\hat\chi$  in $\hat u$ such that $\chi$ contains $\tau$ as a $k$-face and intersects a representative of $q$ at a point  $y\in \mathring \tau$,  the image $\iota_q(\hat\tau)$ is the corresponding $k$-face of $\hat \chi$.

For each $j$-simplex $\hat\sigma$ in $\hat s$ and each $k$-simplex $\hat \tau$ with $\iota_p(\hat\sigma)\subset \hat \tau$, we can choose  the point $y\in\mathring \tau$ such that it is  contained in a sufficiently small neighbourhood $U_x$ of some point $x\in \mathring \sigma$ as in
Definition \ref{def:homog} and in the proof of Lemma \ref{lem:stratinj}. 
The $j$-simplex $\iota_q\circ \iota_p(\hat\sigma)$ is then a $j$-face of an $l$-simplex $\hat \chi$  in $\hat u$ such that $\chi$ contains  $\sigma$ as  a $j$-face
and such that $\chi$ intersects a representative of $p$ at $y$. As $p$ and $q$ represent $r$  at  $y\in U_x$, this representative is also a representative of $r$ at $x$.
Hence, $\iota_q\circ \iota_p(\hat\sigma)=\iota_r(\hat\sigma)$.  
\end{proof}

\begin{proposition} \label{eq:gaugeconc}Let $X$ be a compact stratified $n$-manifold with classical defect data. 
\begin{compactenum}
\item A gauge configuration on $X$ is an assignment
 \begin{compactitem}
\item of a functor $A_t: \Pi_1(\hat t)\to \bullet\sslash G_t$ to each $n$-stratum $t$,
\item of a map  $h_s: \hat s\to M_s$ to each $(n-1)$-stratum $s$,
\end{compactitem}
such that for each path $\delta: y\to y'$ in $\hat s$ and the inclusions $\iota_{l(s)}$, $\iota_{r(s)}$  from Definition \ref{def:reggraph} and \eqref{eq:commtop}
\begin{align}\label{eq:rectcond}
h_s(y')=A_{L(s)}( [\iota_{l(s)}\circ \delta])\rhd h_s(y)\lhd A_{R(s)}( [\iota_{r(s)}\circ \delta])^\inv.
\end{align}

\item A gauge transformation on $X$ is an assignment of a map $\gamma_t:\hat t\to G_t$ to each $n$-stratum $t$ of $X$. It acts on a gauge configuration according to
\begin{align}\label{eq:gtrafocond}
A_t([\delta])\mapsto \gamma_t(x')\cdot A_t([\delta])\cdot \gamma_t(x)^\inv\qquad\qquad 
h_s(y)\mapsto \gamma_{L(s)}\circ \iota_{l(s)}(y)\rhd h_s(y)\lhd \gamma_{R(s)}\circ \iota_{r(s)}(y)^\inv
\end{align}
for each $n$-stratum $t$, $(n-1)$-stratum $s$, path $\delta: x\to x'$ in $\hat t$ and $y\in \hat s$.
\end{compactenum}
\end{proposition}

\begin{proof}
1.~We use  Lemma \ref{lem:gexplicit} to describe the functors $A_u: \Pi_1(\hat u)\to\mathcal D_u$ for strata $u$ of dimension $\geq n-1$. 

 For  each $n$-stratum $t$ Lemma \ref{lem:gexplicit} gives just a functor $A_t=\psi_t: \Pi_1(\hat t)\to \bullet\sslash G_t$. For each $(n-1)$-stratum $s$, the group $G_s$ is given by  $G_s=G_{L(s)}\times G_{R(s)}$, where $l(s): s\to L(s)$ and $r(s): s\to R(s)$ are the two local $n$-strata at $s$.
  The functors $D_{l(s)}: M_s\sslash G_{L(s)}\times G_{R(s)}\to \bullet\sslash G_{L(s)}$ and $D_{r(s)}: M_s\sslash G_{L(s)}\times G_{R(s)}\to \bullet\sslash G_{R(s)}$ send each element of $M_s$ to $\bullet$ and each morphism $(g_L,g_R): m\to g_L\rhd m\lhd g_R^\inv$ to $g_L$ and $g_R$, respectively. 
 
 The first diagram in \eqref{eq:commdiaghighstrata} commutes trivially. The second diagram in \eqref{eq:commdiaghighstrata} for $q=l(s)$ and $q=r(s)$ is  equivalent to  the condition
 $\psi_s([\delta])=(\psi_{L(s)}([\iota_{l(s)}\circ \delta]), \psi_{R(s)}([\iota_{r(s)}\circ \delta]))$ for all paths $\delta:y\to y'$ in $\hat s$. This shows that $\psi_s$ is determined uniquely by $\psi_{L(s)}$ and $\psi_{R(s)}$. Condition \eqref{eq:funcmapcomp}  becomes condition \eqref{eq:rectcond}.

2.~We use the description of  gauge transformations from Lemma \ref{lem:gexplicit} to describe the natural transformations $\gamma_u: A_u\Rightarrow A'_u$ for strata of dimension $\geq n-1$. 

The natural transformation $\gamma_t: A_t\Rightarrow A'_t$  for an $n$-stratum $t$ is simply a map $\gamma_t:\hat t\to G_t$ that satisfies $A'_t([\delta])=\gamma_t(x')\cdot A_t([\delta])\cdot \gamma_t(x)^\inv$ for all paths $\delta:x\to x'$ in $\hat t$. This gives the first formula in \eqref{eq:gtrafocond}.

For an $(n-1)$-stratum $s$ with local $n$-strata $l(s): s\to L(s)$ and $r(s): s\to R(s)$ the natural transformation $\gamma_s: A_s\Rightarrow A'_s$ is a map $\gamma_s:\hat s\to G_{L(s)}\times G_{R(s)}$ that satisfies  \eqref{eq:natconc}.  Diagram \eqref{eq:eq:condgtrafonat} for $q=l(s)$ and $q=r(s)$ implies
$\gamma_{L(s)}\circ \iota_{l(s)}=\pi_{L(s)}\circ \gamma_s$ and $\gamma_{R(s)}\circ \iota_{r(s)}=\pi_{R(s)}\circ \gamma_s$. With  \eqref{eq:natconc} this yields the second formula in \eqref{eq:gtrafocond}.

3.~We show that every gauge configuration $A: T^X\Rightarrow D^X$ is characterised uniquely by its components $A_r: \Pi_1(\hat r)\to\mathcal D_r$ for strata $r$ of dimension $\geq n-1$.

We choose for each local $(n-1)$-stratum $x:u\to s$ and each local $n$-stratum $q:u\to t$ representing sequences 
\begin{align}\label{eq:repseqs}
u=u_k\xrightarrow{x_{k+1}} u_{k+1}\to \ldots\to u_{n-2}\xrightarrow{x_{n-1}} u_{n-1}=s \qquad  u=v_k\xrightarrow{q_{k+1}} v_{k+1}\to \ldots\to v_{n-1}\xrightarrow{q_{n}} v_{n}=t,
\end{align}
in which each local stratum raises the degree by one.
By Lemma \ref{lem:iota-agrees} we then have 
\begin{align}\label{eq:comppose}
&\iota_x=\iota_{x_{n-1}}\circ \ldots\circ \iota_{x_{k+1}} & &f_x=f_{x_{n-1}}\circ \ldots\circ f_{x_{k+1}} & &\phi_x=\phi_{x_{n-1}}\circ \ldots\circ \phi_{x_{k+1}}\\ 
&\iota_q=\iota_{q_{n}}\circ \ldots\circ \iota_{q_{k+1}}
 & &f_q=f_{q_{n}}\circ \ldots\circ f_{q_{k+1}}
& &\phi_q=\phi_{q_{n}}\circ \ldots\circ \phi_{q_{k+1}}\nonumber
\end{align} 
for the associated maps $f_x: M_u\to M_s$ and the group homomorphism $\phi_q: G_u\to G_t$  from  \eqref{eq:fpdef}.
Composing the  commuting diagrams \eqref{eq:commdiaghighstrata} for the local strata $x_j$ and $q_j$ then yields commuting diagrams
\begin{equation}\label{eq:commdiaghighstrata2}
\begin{gathered}
\xymatrix{ \hat u \ar[d]_{\iota_x}\ar[r]^{h_u} & M_u \ar[d]^{f_x}\\
\hat s \ar[r]_{h_s} & M_s
}\end{gathered}
\qquad\textrm{ and }\qquad 
\begin{gathered}
\xymatrix{ \Pi_1(\hat u) \ar[d]_{\Pi_1(\iota_q)}\ar[r]^{\psi_u} & \bullet\sslash G_u \ar[d]^{\phi_q}\\
\Pi_1 (\hat{t})
\ar[r]_{\psi_t} & \bullet\sslash G_t}\end{gathered}
\end{equation}
Here, the maps $f_x: M_u\to M_s$ and $\phi_q: G_u\to G_t$ are simply the projection maps for the product.
By the universal property of the product, specifying the composites $f_x\circ h_u: \hat u\to M_s$ for all local $(n-1)$-strata $x:u\to s$ and the composites $\phi_q\circ \psi_u: \Pi_1(\hat u)\to G_t$ for local $n$-strata $q:u\to t$ determines the map $h_u:\hat u\to M_u$ and the functor $\psi_u: \Pi_1(\hat u)\to G_u$ from Lemma \ref{lem:gexplicit}. This shows that the functors $A_u:\Pi_1(\hat u)\to\mathcal D_u$ for any stratum $u$ are uniquely determined by the maps $h_s:\hat s\to M_s$ for $(n-1)$-strata $s$ and functors $\psi_t:\Pi_1(\hat t)\to \bullet\sslash G_t$ for $n$-strata $t$. 

Conversely, given maps $h_s:\hat s\to M_s$ for all $(n-1)$-strata $s$ and functors $\psi_t:\Pi_1(\hat t)\to \bullet\sslash G_t$ for all $n$-strata $t$ that satisfy \eqref{eq:rectcond}, we can define 
the maps $h_u:\hat u\to  M_u$ and functors $\psi_u:\Pi_1(\hat u)\to \bullet\sslash G_u$  via \eqref{eq:commdiaghighstrata2}. Identities \eqref{eq:grpdactdef}, \eqref{eq:rectcond} and \eqref{eq:comppose} then imply \eqref{eq:funcmapcomp}, which yields a functor $A_u:\Pi_1(\hat u)\to \mathcal D_u$ for each stratum $u$. 
Identities 
 \eqref{eq:comppose} and Lemma \ref{lem:iota-agrees} then ensure that they combine into a natural transformation  $A: T^X\Rightarrow D^X$.

4.~We show that a gauge  transformation is determined  by its components $\gamma_t: A_t\Rightarrow A'_t$ for $n$-strata $t$. 

Choose for each local $n$-stratum $q:u\to t$ a representing sequence as in \eqref{eq:repseqs}. 
Then the map $\iota_q:\hat u\to\hat t$ and group homomorphism $\phi_q: G_u\to G_t$   are given as in \eqref{eq:comppose}, and combining the commuting  diagrams \eqref{eq:eq:condgtrafonat}  for the local strata $q_j$ yields a corresponding diagram for $q$. This gives
 \begin{align}\label{eq:gammacond}
\phi_q \circ \gamma_u=\gamma_t \circ \iota_q
\end{align}
for all local $n$-strata $q:u\to t$. As $\phi_q: G_u\to G_t$ is the projection map for the product,  this determines $\gamma_u$.  

Conversely, given maps $\gamma_t: \hat t\to G_t$ for all $n$-strata $t$, we can define maps $\gamma_u:\hat u\to G_u$ for all strata $u$ by \eqref{eq:gammacond}.  Lemma \ref{lem:iota-agrees} and identities \eqref{eq:grpdactdef}, \eqref{eq:comppose} and \eqref{eq:gtrafocond} then imply \eqref{eq:natconc}. This shows that the maps $\gamma_u:\hat u\to G_u$
define a natural transformation $\gamma_u: A_u\Rightarrow A'_u$. 
Lemma \ref{lem:iota-agrees} and 
\eqref{eq:comppose} then imply \eqref{eq:eq:condgtrafonat} for all local strata $q:u\to v$ and yield a gauge transformation $\gamma:A\Rrightarrow A'$.
\end{proof}

Proposition \ref{eq:gaugeconc} shows again that the groupoid $G\big(Q^X,T^X,D^X\big)$ of gauge configurations and gauge transformations indeed an action groupoid,  $\mathcal A^X\sslash \mathcal G^X$. 
More specifically, 
the group of gauge transformations 
 is the product of the groups of maps  $\gamma_t: \hat t\to G_t$ over  all  $n$-strata $t$ of $X$  
$$\mathcal G^X=\prod_{t\in S_n^X} (G_t)^{\hat t}.$$
The functor $P_\partial: \mathcal A^X\sslash \mathcal G^X\to \mathcal A^{\partial X}\sslash \mathcal G^{\partial X}$ from Corollary \ref{prop:projfunc} defines a map $P_\partial: \mathcal A^X\to \mathcal A^{\partial X}$ and a group homomorphism  
$P_\partial: \mathcal G^X\to \mathcal G^{\partial X}$. The latter forgets the group elements for  $n$-strata that are not associated with a boundary $(n-1)$-stratum. For 
a boundary $(n-1)$-stratum $s$ with associated
$n$-stratum $t=\iota_{n-1}(s)$, the functor $P_\partial$
 restricts the map 
$\gamma_t:\hat t\to G_t$ to   $\gamma_s=\gamma_t\circ \iota_{s\partial}:\hat s\to G_s$, where  $\iota_{s\partial}: \hat s\to \hat t$  is the map from \eqref{eq:stratboundemb}.

Conversely, each boundary gauge transformation $\gamma^{\partial X}\in \mathcal G^{\partial X}$ lifts to a gauge transformation $\gamma^X\in \mathcal G^{ X}$ with $P_\partial(\gamma^X)=\gamma^{\partial X}$. The associated maps $\gamma_t: \hat t\to G_t$ for an $n$-stratum $t$ of $X$ assign the unit element of $G_t$ to each point of  $\hat t$ that is not in the image of a map $\iota_{s\partial}: \hat s\to \hat t$. For a point $x=\iota_{s\partial}(y)$, 
they are given  by $\gamma_t(x)=\gamma_s(y)$, as $\iota_{s\partial} : \hat{s} \to \hat{t}$ is injective. Hence, $P_\partial : \mathcal A^X\sslash \mathcal G^X\to  \mathcal A^{\partial X}\sslash \mathcal G^{\partial X}$ is a fibration of groupoids. 

If $M$ is a stratified 3-manifold with boundary $\partial M= \Sigma_0\amalg \Sigma_1$ and projection functors $P_j: \mathcal A^M\sslash \mathcal G^M\to \mathcal A^{\Sigma_j}\sslash \mathcal G^{\Sigma_j}$,  then we have $\mathcal A^{\partial M}\sslash \mathcal G^{\partial M}=\mathcal A^{\Sigma_0\amalg \Sigma_1}\sslash \mathcal G^{\Sigma_0\amalg \Sigma_1}\cong \mathcal A^{\Sigma_0}\sslash \mathcal G^{\Sigma_0}\times\mathcal A^{\Sigma_1}\sslash \mathcal G^{\Sigma_1}$, an isomorphism of categories and hence a fibration, and $P_\partial=\langle P_0,P_1\rangle:\mathcal A^M\sslash \mathcal G^M\to \mathcal A^{\Sigma_0}\sslash \mathcal G^{\Sigma_0} \times \mathcal A^{\Sigma_1}\sslash \mathcal G^{\Sigma_1}$. We obtain the following corollary.

\begin{corollary}\label{cor: span fib} Any compact stratified 3-manifold $M$  with classical defect data and boundary   
 $\partial M=\Sigma_0\amalg \Sigma_1$ defines a fibrant span of groupoids
$$
\mathcal A^{\Sigma_0}\sslash \mathcal G^{\Sigma_0} \xleftarrow{P_0} \mathcal A^M\sslash \mathcal G^M\xrightarrow{P_1}  \mathcal A^{\Sigma_1}\sslash \mathcal G^{\Sigma_1}.
$$
\end{corollary}

This corollary and its derivation from Proposition \ref{eq:gaugeconc} show that working with \emph{fibrant} spans of groupoids is not merely a choice of technique, but has a direct gauge theoretical interpretation.  The span of groupoids in Corollary \ref{cor: span fib} relates the gauge groupoid of 
a cobordism $M: \Sigma_0\to\Sigma_1$ to the gauge groupoids of its boundary surfaces $\Sigma_0$ and $\Sigma_1$. That this span is fibrant corresponds to the fact  that every boundary gauge transformation extends to a gauge transformation of the bulk.

\subsection{Reduced gauge groupoids}
\label{subsec:redgaugegrpd}

For computations and examples, it is convenient to work with fundamental groupoids with chosen basepoints instead of full fundamental groupoids.
This requires a coherent choice of basepoints for all strata and a counterpart of Proposition \ref{prop:tfunctor} for fundamental groupoids with basepoints.

We denote by   $\Top_B$ the category whose objects are pairs $(X,U)$ of a topological space $X$ and a subspace $U\subset X$ and whose morphisms $f: (X, U)\to (Y,V)$ are continuous maps $f: X\to Y$ with $f(U)\subset V$. The  functor $\Pi_1: \Top_B\to \Grpd$ sends a pair $(X,U)$ to the full subgroupoid $\Pi_1(X,U)\subset \Pi_1(X)$ with object set $U$ and  a morphism $f: (X,U)\to (Y,V)$ to the functor $\Pi_1(f): \Pi_1(X,U)\to \Pi_1(Y,V)$. 

To formulate a counterpart of Proposition \ref{prop:tfunctor} for fundamental groupoids with basepoints, one needs a choice of basepoints in $\hat s$ for each stratum $s$ 
 that is compatible with the maps $\iota_q:\hat s\to \hat t$  for local strata $q:s\to t$ from \eqref{eq:commtop} and with the maps $\iota_{s\partial}:\hat s\to \hat t$ for boundary strata $s$ from \eqref{eq:stratboundemb}.

\begin{definition} \label{def:basepointset}A  \textbf{coherent basepoint set} for a compact stratified $n$-manifold $X$ is a finite  set $B\subset X$ that satisfies  $V_s=f_s^\inv(B)\neq \emptyset$ for all strata and boundary strata $s$ of $X$.
\end{definition}

Note that the commuting diagrams \eqref{eq:commtop} and  \eqref{eq:stratboundemb} ensure that any coherent basepoint set $B\subset X$ satisfies
\begin{compactitem}
\item $\iota_q(V_s)\subset V_t$ for all $k$-strata  $s$ and local $(k+1)$-strata $q:s\to t$ at $s$,
\item $\iota_{s\partial}(V_s)\subset V_t$ for each boundary stratum $s$ and associated $(k+1)$-stratum $t$ from \eqref{eq:embedbound}.
\end{compactitem}

\begin{example} \label{ex:finestrat}For a compact $n$-manifold $X$ with a fine stratification the set $B=X^0\cup\partial X^0$ is a coherent basepoint set. The associated vertex sets
\begin{align}\label{eq:vertexset}
V_s=f_s^\inv(X^0\cup \partial X^0)\subset \hat s,
\end{align}
are (i) non-empty by assumption, satisfy (ii) due to diagram \eqref{eq:commtop} and (iii) due to  diagram \eqref{eq:stratboundemb}.
\end{example}

\begin{proposition}\label{prop:fincase} Let $X$ be a  compact stratified   $n$-manifold with a coherent basepoint set. 

There is  a functor $T^{X}_\bullet: \mathcal Q^X\to \Grpd$ that sends
\begin{compactitem}
\item  a $k$-stratum $s$ to the groupoid $\Pi_1(\hat s, V_s)$,
\item  a local $(k+1)$-stratum $q:s\to t$  to the functor $\Pi_1(\iota_q): \Pi_1(\hat s, V_s)\to \Pi_1(\hat t, V_t)$,
\end{compactitem}
and a natural transformation $\tau^{X\bullet}: T^{\partial X}_{\bullet}\Rightarrow T^{X}_{\bullet} I_\partial$ with components
\begin{align}\label{eq:tnatdefdisc}
\tau^{X\bullet}_s=\Pi_1(\iota_{s\partial}): \Pi_1(\hat s, V_s)\to \Pi_1(\hat t, V_t).
\end{align}  
The inclusion functors $\iota_s: \Pi_1(\hat s, V_s)\to \Pi_1(\hat s)$ are equivalences of categories and define a natural transformation $\iota^X: T^{X\bullet}\Rightarrow T^X$ such that the following diagram commutes
$$
\xymatrix{ T^{\partial X\bullet}  \ar@{=>}[d]_{\tau^{X\bullet}} \ar@{=>}[r]^{\iota^{\partial X}} & T^{\partial X} \ar@{=>}[d]^{\tau^X}\\
T^{X\bullet}I_\partial \ar@{=>}[r]_{\iota^XI_\partial} & T^XI_\partial.
}
$$
\end{proposition}

Given a stratification with a coherent basepoint set, we  obtain a reduced  gauge groupoid  by replacing the functor $T^X:\mathcal Q^X\to \Grpd$ from Proposition \ref{prop:tfunctor} by the functor
$T^X_\bullet:\mathcal Q^X\to \Grpd$   in Definition \ref{def:gaugegrpd}.

\begin{definition} \label{def:redgrpd}Let $X$ be a compact  stratified  $n$-manifold  with  a coherent basepoint set. 
The \textbf{reduced gauge groupoid} on $X$ is the groupoid
$\mathcal A^{X \bullet}\sslash \mathcal G^{ X \bullet}=G(Q^X,T^X_\bullet, D^X)$.
\end{definition}

As in Proposition \ref{eq:gaugeconc} one finds that  this groupoid is  an action groupoid and that reduced gauge configurations and reduced gauge transformations are characterised uniquely by their components for $n$- and $(n-1)$-strata.

\begin{corollary}\label{cor:redstrata} Let $X$ be a compact stratified $n$-manifold $X$ with defect data and a coherent basepoint set.
\begin{compactenum}
\item  A reduced gauge configuration on $X$  is an assignment of 
\begin{compactitem}
\item  a functor $A_t:\Pi_1(\hat t, V_t)\to \bullet\sslash G_t$ to each $n$-stratum $t$,
\item  a map $h_s: V_s\to M_s$ to each $(n-1)$-stratum $s$,
\end{compactitem}
 satisfying \eqref{eq:rectcond} for all $y,y'\in V_s$ and paths $\delta: y\to y'$ in $\hat s$.
 
 \item 
A reduced gauge transformation is an assignment of a map $\gamma_t: V_t\to G_t$ to each $n$-stratum $t$ that acts according to \eqref{eq:gtrafocond} for all paths $\delta: x\to x'$ in $\hat t$ with $x,x'\in V_t$ and all $y\in V_s$.
\end{compactenum}
\end{corollary}
We will now show that the reduced gauge groupoid of a compact stratified $n$-manifold $X$ is equivalent to its gauge groupoid. This will allow us to compute  defect TQFTs by choosing coherent basepoint sets and working with  reduced gauge groupoids. Moreover, it shows that all gauge groupoids are essentially finite.

\begin{proposition} \label{rem:simplegauge} Let $X$ be a compact   stratified  $n$-manifold  with  classical defect data and a  coherent basepoint set. 
Then the groupoids $\mathcal A^X\sslash \mathcal G^X$ and $\mathcal A^{X\bullet}\sslash \mathcal G^{X\bullet}$ are equivalent.
\end{proposition}

\begin{proof} 
There is a canonical restriction functor $R:\mathcal A^X\sslash \mathcal G^X\to \mathcal A^{X\bullet}\sslash \mathcal G^{X\bullet}$ that sends
\begin{compactitem}
\item a gauge configuration $A:T^X\Rightarrow D^X$ to  $R(A)=A\circ \iota^X: T^{X\bullet}\Rightarrow D^X$,
\item a gauge transformation $\gamma: T^X\times I\Rightarrow D^X$ to  $R(\gamma)=\gamma\circ (\iota^X\times I): T^{X\bullet}\times I\Rightarrow D^X$,
\end{compactitem}
where $\iota^X: T^{X\bullet}\Rightarrow T^X$ is the natural transformation from Proposition \ref{prop:fincase}. 
It restricts
\begin{compactitem}
\item the functors $A_t: \Pi_1(\hat t)\to \bullet\sslash G_t$ for each $n$-stratum $t$ to functors $A^\bullet_t: \Pi_1(\hat t, V_t)\to \bullet\sslash G_t$,
\item the maps $h_s:\hat s\to M_s$ for each $(n-1)$-stratum $s$ to   maps $h^\bullet_s: V_s\to M_s$,
\item the maps $\gamma_t: \hat t\to G_t$ for each $n$-stratum $t$ to maps $\gamma^\bullet_t: V_t\to G_t$. 
\end{compactitem}

To show that $R$ is essentially surjective and fully faithful
we choose 
\begin{compactitem}
\item for each $n$-stratum $t$ and $x\in \hat t$  a path $\alpha_x: v_x\to x$ in $\hat t$ with $v_x\in V_t$ and   $\alpha_x$ constant for   $x\in V_t$,
 \item for each $(n-1)$-stratum $s$ and $y\in \hat s$  a path
$\beta_y: v_y\to y$ in $\hat s$ with $v_y\in V_s$  and $\beta_y$ constant for  $y\in V_s$.
\end{compactitem}

1.~To see that $R$ is essentially surjective, we define for each $n$-stratum $t$ and  functor  $A^\bullet_t: \Pi_1(\hat t, V_t)\to \bullet\sslash G_t$ a functor $A_t: \Pi_1(\hat t)\to \bullet\sslash G_t$  by setting for each path $\alpha:x\to y$
$$A_t([\alpha])=A_t^\bullet([\alpha_y^\inv\circ\alpha\circ \alpha_x]).$$  Note that this implies $A_t([\alpha_x])=1$ for all $x\in \hat t$ and $A_t([\alpha])=A_t^\bullet([\alpha])$ for each path $\alpha:x\to y$ with $x,y\in V_t$.

If $h_s^\bullet: V_s\to M_s$  satisfies \eqref{eq:rectcond} for all paths  $\beta: y\to y'$ in $\hat s$ with $y,y'\in V_s$,
 we obtain a map $h_s:\hat s\to M_s$ that satisfies \eqref{eq:rectcond} for all paths  $\beta: y\to y'$ in $\hat s$  by setting
\begin{align*}
h_s(y)
&=A_{L(s)}([\iota_{l(s)}\circ \beta_y])\rhd h^\bullet_s(v_y)\lhd A_{R(s)}([ \iota_{r(s)}\circ \beta_y])^\inv,
 \end{align*}
 where $\iota_{l(s)}$ and $\iota_{r(s)}$ are the maps  from \eqref{eq:commtop} for the local $n$-strata $l(s): s\to L(s)$ and $r(s): s\to R(s)$ from Definition \ref{def:reggraph}.
 By construction, this yields $h_s(v)=h_s^\bullet(v)$ for all $v\in V_s$. 
  Note also that $h_s(y)$ does not depend on the choice of the path $\beta_y$, because $h_s^\bullet$ satisfies \eqref{eq:rectcond} for paths with endpoints in $V_s$. This shows that for every gauge configuration $A^\bullet:T^{X\bullet}\Rightarrow D^X$ there is a gauge configuration $A:T^{X}\Rightarrow D^X$ with $R(A)=A^\bullet$.

2. To see that $R$ is fully faithful, note that any gauge transformation $\gamma: A\Rrightarrow A'$ is given by maps $\gamma_t: \hat t\to G_t$ for each $n$-stratum $t$ that satisfy
$A'_t([\alpha])=\gamma_t(y)\cdot A_t([\alpha])\cdot \gamma_t(x)^\inv$ for all paths $\alpha:x\to y$ in $\hat t$. This implies 
\begin{align}
\gamma_t(x)=A'([\alpha_x])\cdot \gamma_t(v_x)\cdot A([\alpha_x])^\inv
\end{align}
for all $x\in \hat t$. Thus, the maps $\gamma_t$ are determined uniquely by the associated maps  $R(\gamma)_t$ and the functors $A_t,A'_t:\Pi_1(\hat t)\to \bullet\sslash G_t$. Hence, gauge transformations $\gamma: A\Rrightarrow A'$ are in bijection with  gauge transformations $R(\gamma): R(A)\Rrightarrow R(A')$. 
This shows that $R$ is fully faithful.
\end{proof}

\begin{corollary}\label{cor:essfinite} For each compact  stratified $n$-manifold $X$ the groupoid $\mathcal A^X\sslash \mathcal G^X$ of gauge configurations and gauge transformations on $X$ is essentially finite.
\end{corollary}

\begin{proof} By Proposition \ref{rem:simplegauge}  the groupoid $\mathcal A^X\sslash \mathcal G^X$ is equivalent to the  groupoid $\mathcal A^{X\bullet}\sslash \mathcal G^{X\bullet}$ for each coherent basepoint set.  It is thus sufficient to choose a coherent basepoint set and to show that $\mathcal A^{X\bullet}\sslash \mathcal G^{X\bullet}$ is  finite.

By definition of a coherent basepoint set, all sets $V_s$ for strata $s$ of $X$ are finite. Because all spaces  $\hat t$ for $n$-strata  $t$ of $X$  are compact PL manifolds with boundary, their fundamental groups are finitely presented. For this note that a compact PL manifold with boundary has a finite triangulation.  Contracting the 1-skeleton of such a triangulation gives a finite presentation of its fundamental group.  With the finiteness of the groups $G_t$, this implies that the number of functors $A_t:\Pi_1(\hat t, V_t)\to \bullet\sslash G_t$ is finite. 
Due to the finiteness of the sets $M_s$ assigned to the $(n-1)$-strata $s$, there is only a finite number of maps $h_s:V_s\to M_s$. This shows that $\mathcal A^{X\bullet}\sslash \mathcal G^{X\bullet}$ has only a finite number of objects.  Likewise, the number of maps $\gamma_t: V_t\to G_t$ is finite for each $n$-stratum $t$, which shows that all Hom-sets of $\mathcal A^{X\bullet}\sslash \mathcal G^{X\bullet}$ are finite.
\end{proof}

\subsection{Geometrical and combinatorial description}
\label{subsec:geomdescript}

If a compact   $n$-manifold $X$   is equipped with a fine stratification, it is possible to give a purely combinatorial description of the associated reduced gauge groupoid  in terms of graph gauge configurations and graph gauge transformations on the thickening $X_{th}$. Note  that in this case there are no isolated loops.

Recall from Table \ref{tab:strat}
that each (boundary) vertex  of $X_{th}$  is contained in a unique $n$-stratum $t$ and (boundary) vertices of $X_{th}$ in  $t$ are in bijection with elements of the set $V_t$ from \eqref{eq:vertexset}.

Recall also that the (boundary) edges  of $X_{th}$ are partitioned into two sets, black and red ones. Each black (boundary) edge of $X_{th}$ is contained in a unique $n$-stratum $t$ and corresponds to a unique  (boundary) edge $e$ of $X$ with $e\subset \bar t$.  
Red (boundary) edges  of $X_{th}$ are edges between a thickened vertex $v$ and a thickened $(n-1)$-stratum $s$ of $X$. They correspond to pairs $(v, q)$ with a local $(n-1)$-stratum  $q:v\to s$. Thus, the red edges in the  thickening of an $(n-1)$-stratum $s$ are in bijection with elements of the set $V_s$ from \eqref{eq:vertexset}.

Recall also that each (boundary) edge  $e$ of $X$ with $e\subset \bar s$ for an $(n-1)$-stratum $s$ is thickened to a green rectangle 
with two red edges $s(e), t(e)$ at the start and target end and two black edges  $l(e), r(e)$ parallel to $e$ that are contained in the $n$-strata $L(s)$, $R(s)$. 
Each black and red edge is contained in a unique rectangle of this type.
For edges $e$ of a 3-manifold $X$, the orientation of the black edges is fixed by the orientation of $e$. For boundary edges and edges of a surface $\Sigma$ it is  arbitrary. All red edges are oriented by the normals.
\begin{align}\label{eq:defectrect}
\begin{tikzpicture}[scale=.6, baseline=(current  bounding  box.center)]
\draw[color=green, fill=green, fill opacity=.15, line width=0pt] (0,0)--(0,2)--(4,2)--(4,0)--(0,0);
\draw[color=red, line width=1.5pt] (0,0)--(0,2) node[sloped, pos=0.5, allow upside down]{\arrowOut}  node[pos=.5, anchor=east]{$s(e)$};
\draw[color=red, line width=1.5pt] (4,0)--(4,2) node[sloped, pos=0.5, allow upside down]{\arrowOut}  node[pos=.5, anchor=west]{$t(e)$};
\draw[color=black, line width=2pt, ->,>=stealth] (.5,1)--(3.5,1)  node[pos=.5, anchor=north]{$e$};
\draw[color=red, line width=1pt,->,>=stealth] (2,1)--(2,1.5);
\draw[color=black, line width=1.5pt] (0,2)--(4,2) node[sloped, pos=0.5, allow upside down]{\arrowOut}  node[pos=.5, anchor=south]{$l(e)\subset L(s)$};
\draw[color=black, line width=1.5pt] (0,0)--(4,0) node[sloped, pos=0.5, allow upside down]{\arrowOut}  node[pos=.5, anchor=north]{$r(e)\subset R(s)$};
\draw[fill=black](0,2) circle(.08) ;
\draw[fill=black](4,2) circle(.08) ;
\draw[fill=gray, color=black](0,0) circle(.08) ;
\draw[fill=gray, color=black](4,0) circle(.08) ;
\node at (0,2)[anchor=south east]{$(s,l)$};
\node at (4,2)[anchor=south west]{$(t,l)$};
\node at (0,0)[anchor=north east]{$(s,r)$};
\node at (4,0)[anchor=north west]{$(t,r)$};
\end{tikzpicture}
\end{align}

\begin{definition} \label{def:graphgt} Let $X$ be a compact   $n$-manifold  with  a fine stratification and classical defect data.
\begin{enumerate}
\item A \textbf{graph gauge configuration} on $X_{th}$ is an assignment of the following data to the edges of $X_{th}$
\begin{compactitem}
\item to each black edge $f$  contained in an $n$-stratum $t$ a group element $g_f\in G_t$,
\item 
to each  red edge $f$  in  the boundary of a thickened $(n-1)$-stratum $s$ an element $m_f\in M_s$, 
\end{compactitem}
such that 
\begin{compactenum}[(i)]
\item the oriented ordered product of the group elements on each blue and white  plane is trivial.
\item 
for each edge  $e\subset \bar s$ one has
$m_{t(e)}=g_{l(e)}\rhd m_{s(e)}\lhd g_{r(e)}^\inv$. 
\end{compactenum}

\item  A \textbf{graph gauge transformation} on $X_{th}$ is an assignment of an element $g_v\in G_t$ to each vertex $v$ of $X_{th}$ with $v\in t$. It acts
on a gauge configuration by 
\begin{align}
\label{eq:gaugeactconc}
&g_{l(e)}\mapsto g_{(t,l)}\cdot g_{l(e)}\cdot g_{(s,l)}^\inv & &g_{r(e)}\mapsto g_{(t,r)}\cdot g_{r(e)}\cdot g_{(s,r)}^\inv\\
&m_{s(e)}\mapsto g_{(s,l)}\rhd m_{s(e)}\lhd g_{(s,r)}^\inv & &m_{t(e)}\mapsto g_{(t,l)}\rhd m_{t(e)}\lhd g_{(t,r)}^\inv.  \nonumber
\end{align}
\end{enumerate}
\end{definition}

\begin{figure}
\begin{center}
\def\svgwidth{.6\columnwidth}
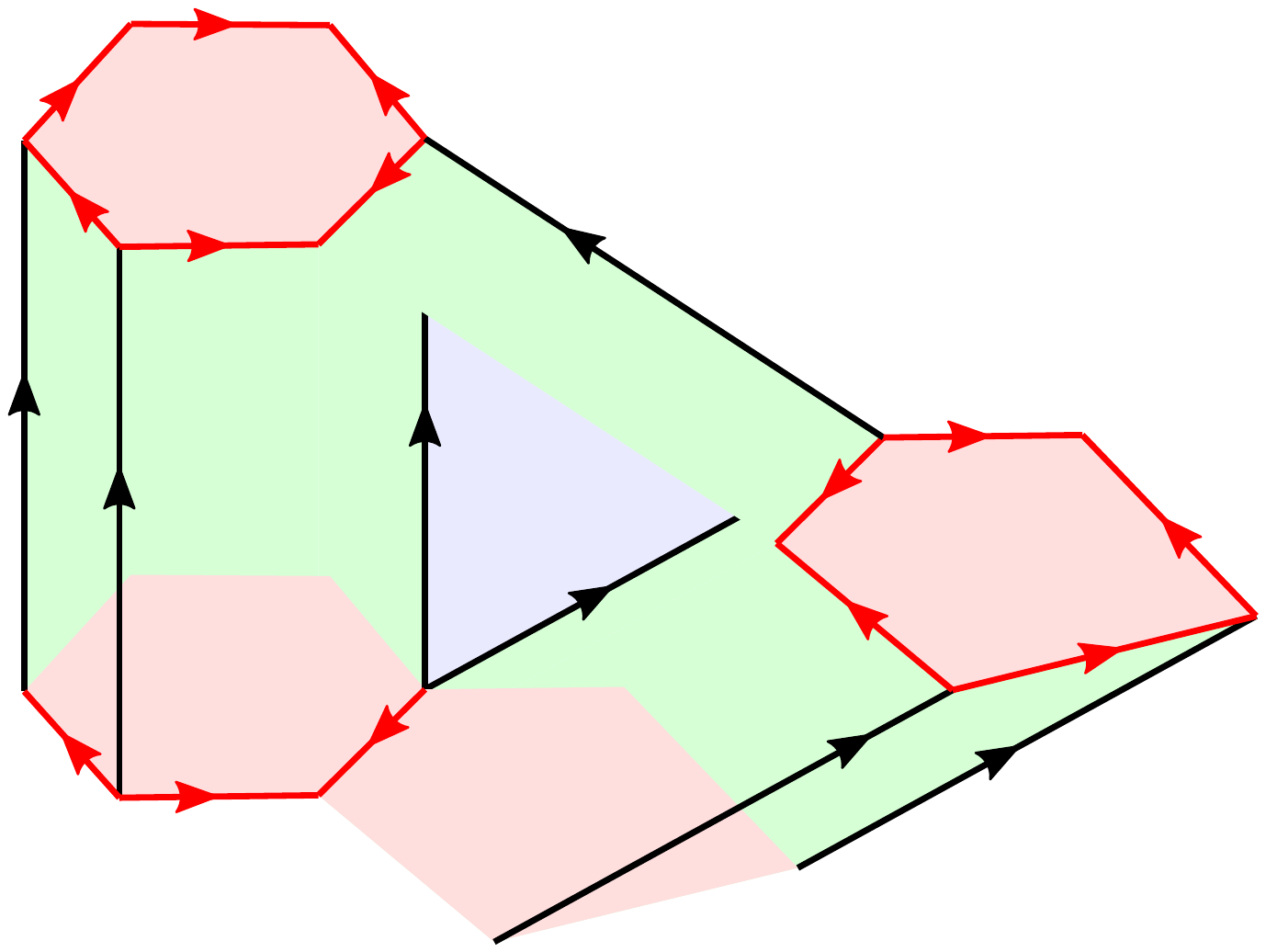
\end{center}
\vspace{-1cm}
\caption{Labelling the edges of the thickening $X_{th}$ with elements of the groups $G_{L(p)}$, $G_{R(p)}$ and the\newline  $G_{L(p)}\times G_{R(p)}^{op}$-set $M_p$. The labels are related by the  conditions:\newline
 $m_v=g_d\rhd m_w\lhd h_d^\inv$, \;
 $m_w=g_f\rhd m_u\lhd h_f^\inv$,  \;
 $m_v=g_e\rhd m_u\lhd h_{e}^\inv$,  \;
 $g_e=g_d\cdot g_f$, \;
 $h_e=h_d\cdot h_f$.
}
\label{fig:defect_label}
\end{figure}

A graph gauge configuration on the thickening  $X_{th}$ of a stratification is shown in Figure \ref{fig:defect_label}. The action of a graph gauge transformation is illustrated in Figure \ref{fig:defect_gauge}. Note that graph gauge configurations and graph gauge transformations in the sense of Definition \ref{def:graphgt} are a straightforward generalisation of the usual notion of a flat graph connection and graph gauge transformation for a finite group. 

If $X$ is equipped with a fine stratification, then we can consider its coherent basepoint set from Example \ref{ex:finestrat} and  the associated groupoid $\mathcal A^{X\bullet}\sslash \mathcal G^{X\bullet}$ of reduced gauge configurations and gauge transformations from Definition \ref{def:redgrpd}, which is equivalent to the groupoid $\mathcal A^X\sslash \mathcal G^X$  by Proposition \ref{rem:simplegauge}. We then have

\begin{corollary} \label{cor:redpoly} Let $X$ be a compact   $n$-manifold  with  a fine stratification and classical defect data. 
The groupoid $\mathcal A^{X\bullet}\sslash \mathcal G^{X\bullet}$  is isomorphic to the groupoid of graph gauge configurations and  transformations on $X_{th}$.
\end{corollary}

\begin{proof} By Proposition \ref{rem:simplegauge} a reduced gauge configuration is given by 
\begin{compactitem}
\item  a functor $A_t:\Pi_1(\hat t, V_t)\to \bullet\sslash G_t$ for each $n$-stratum $t$,
\item  a map $h_s: V_s\to M_s$ for each $(n-1)$-stratum $s$,
\end{compactitem}
 satisfying \eqref{eq:rectcond} for all $y,y'\in V_s$ and paths $\delta: y\to y'$ in $\hat s$. 
 
 A reduced gauge transformation is given by a map $\gamma_t: V_t\to G_t$ for each $n$-stratum $t$ and acts on a reduced gauge configuration by \eqref{eq:gtrafocond} for each $y\in V_s$ and path $\delta:x\to x'$ with $x,x'\in V_t$.
 
 As the stratification is fine,  the thickening of an $n$-stratum $t$ is an $n$-ball, whose boundary   is a polyhedral $(n-1)$-sphere with vertex set $V_t$ and the black edges in $t$ as edges. 
 Each thickened $(n-1)$-stratum $s$ is a polygonal cylinder with vertex set $V_s\times\{0,1\}$. Its red edges connect vertices $(v,0)$ and $(v,1)$ for $v\in V_s$. Its remaining edges are black edges shared with the thickenings of  $L(s)$ and $R(s)$. 
 This implies that reduced gauge configurations on $X$ are in bijection with assignments
 \begin{compactitem}
 \item of group elements $g_f=A_t([f])$ to each black edge $f$ of $X_{th}$ with $f\subset t$,
 \item elements $m_f=h_s(v)$ to each red edge $f$ of $X_{th}$  that connects $(v,0)$ and $(v,1)$ for $v\in V_s$.
 \end{compactitem}
 Condition \eqref{eq:rectcond} then becomes condition (ii) in Definition \ref{def:graphgt}. The functoriality of $A_t$ is equivalent to condition (i) in Definition \ref{def:graphgt}
 As vertices of $X_{th}$ in an $n$-stratum $t$ are in bijection with elements of $V_t$, a gauge transformation is simply an assignment of an element $g_v\in G_t$ to each vertex of $X_{th}$ contained in $t$.  Formula \eqref{eq:gtrafocond} for the action of gauge transformation then becomes \eqref{eq:gaugeactconc}.
\end{proof}

\begin{figure}
\begin{center}
\begin{tikzpicture}[scale=.6]
\draw[color=green, fill=green, fill opacity=.15, line width=0pt] (0,0)--(0,2)--(4,2)--(4,0)--(0,0);
\draw[color=red, line width=1.5pt] (0,0)--(0,2) node[sloped, pos=0.5, allow upside down]{\arrowOut}  node[pos=.5, anchor=east]{$m_{s(e)}$};
\draw[color=red, line width=1.5pt] (4,0)--(4,2) node[sloped, pos=0.5, allow upside down]{\arrowOut}  node[pos=.5, anchor=west]{$m_{t(e)}$};
\draw[color=black, line width=2pt,->,>=stealth] (.5,1)--(3.5,1)  node[pos=.5, anchor=north]{$e$};
\draw[color=red, line width=1pt,->,>=stealth] (2,1)--(2,1.5);
\draw[color=black, line width=1.5pt] (0,2)--(4,2) node[sloped, pos=0.5, allow upside down]{\arrowOut}  node[pos=.5, anchor=south]{$g_{l(e)}$};
\draw[color=black, line width=1.5pt] (0,0)--(4,0) node[sloped, pos=0.5, allow upside down]{\arrowOut}  node[pos=.5, anchor=north]{$g_{r(e)}$};
\draw[fill=black](0,2) circle(.08) node[anchor=south east]{$g_{(s,l)}$};
\draw[fill=black](4,2) circle(.08) node[anchor=south west]{$g_{(t,l)}$};
\draw[fill=gray, color=black](0,0) circle(.08) node[anchor=north east]{$g_{(s,r)}$};
\draw[fill=gray, color=black](4,0) circle(.08) node[anchor=north west]{$g_{(t,r)}$};
\end{tikzpicture}
\qquad 
\begin{tikzpicture}[scale=.6]
\draw[color=green, fill=green, fill opacity=.15, line width=0pt] (10,0)--(10,2)--(14,2)--(14,0)--(10,0);
\draw[color=red, line width=1.5pt] (10,0)--(10,2) node[sloped, pos=0.5, allow upside down]{\arrowOut}  node[pos=.5, anchor=east]{$g_{(s,l)}\rhd m_{s(e)}\lhd g_{(t,r)}^\inv\;$};
\draw[color=red, line width=1.5pt] (14,0)--(14,2) node[sloped, pos=0.5, allow upside down]{\arrowOut}  node[pos=.5, anchor=west]{$\;g_{(t,l)}\rhd m_{t(e)}\lhd g_{(t,r)}^\inv$};
\draw[color=black, line width=2pt,->,>=stealth] (10.5,1)--(13.5,1)  node[pos=.5, anchor=north]{$e$};
\draw[color=red, line width=1pt,->,>=stealth] (12,1)--(12,1.5);
\draw[color=black, line width=1.5pt] (10,2)--(14,2) node[sloped, pos=0.5, allow upside down]{\arrowOut}  node[pos=.5, anchor=south]{$g_{(t,l)}\cdot g_{l(e)}\cdot g_{(s,l)}^\inv$};
\draw[color=black, line width=1.5pt] (10,0)--(14,0) node[sloped, pos=0.5, allow upside down]{\arrowOut}  node[pos=.5, anchor=north]{$g_{(t,r)}\cdot g_{r(e)}\cdot g_{(s,r)}^\inv$};
\draw[fill=black](10,2) circle(.08) ;
\draw[fill=black](14,2) circle(.08) ;
\draw[fill=gray, color=black](10,0) circle(.08);
\draw[fill=gray, color=black](14,0) circle(.08) ;
\end{tikzpicture}
\end{center}$\quad$\\[-10ex]
\caption{Action of a gauge transformation  for green planes in the thickening of an edge $e$.}
\label{fig:defect_gauge}
\end{figure}
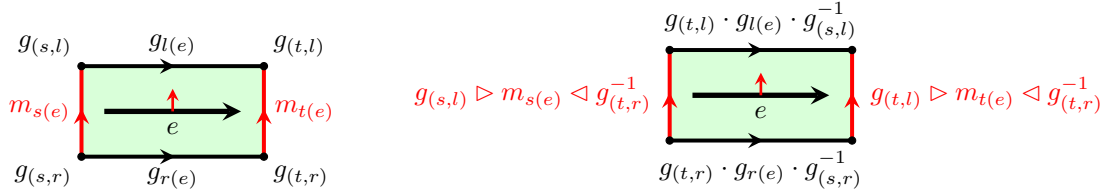

The reduced counterparts of the restriction functors  $P_t:\mathcal A^X\sslash \mathcal G^X\to \mathcal D_t^{\Pi_1(\hat t)}$ and   $P_\partial: \mathcal A^X\sslash \mathcal G^X\to \mathcal A^{\partial X}\sslash \mathcal G^{\partial X}$  from Definition \ref{def:gaugegrpd}  and  Corollary \ref{prop:projfunc} also have a  geometric interpretation.
 They restrict the labels on  edges and vertices of $X_{th}$  to  those edges and vertices  that are contained   in the thickening of a stratum $t$ and in the boundary of $X_{th}$, respectively. 
For a  vertex $v$ of $X$  these are the red edges and vertices in its thickening. For an edge $e$  these are the  red edges at its starting and target vertex  and the black edges that connect them.

It is possible to reduce the action groupoid for a compact $n$-manifold $X$ with a fine stratification and classical defect data even further. 
As all strata of $X$ are balls, one can apply  reduced gauge transformations to  achieve that 
$A^\bullet_t([\alpha])=1$ for all paths $\alpha: x\to y$ in $\hat t$ with $x,y\in V_t$ and for all $n$-strata $t$. Condition \eqref{eq:rectcond} then states that the maps $h_s^\bullet:V_s\to M_s$ are constant. This amounts to  labelling all  black edges of $X_{th}$ by units of the associated groups and  all red edges in the closure of an $(n-1)$-stratum $s$ by the same element $m_s$. The reduced gauge transformations that preserve these restrictions are the ones given by constant maps $\gamma_t: V_t\to G_t$ for each $n$-stratum $t$. 
Applying Proposition \ref{rem:simplegauge} then yields

\begin{corollary}\label{cor:actred} Let $X$ be a compact  $n$-manifold with a fine stratification and classical defect data. 
The groupoid $\mathcal A^X\sslash \mathcal G^X$ of gauge configurations and gauge configurations on $X$ is equivalent to the action groupoid
\begin{align}\label{eq:fullred}
\mathcal A^0\sslash \mathcal G^0=\big(\Pi_{s\in S^X_{n-1}} M_s\big)\sslash \big(\Pi_{t\in S^X_n} G_t\big).
\end{align}
\end{corollary}

This reduction procedure can also be understood geometrically. Considering only gauge configurations on $X_{th}$ that assign trivial group elements to all black edges amounts to contracting all black edges in  $X_{th}$.  This contracts each thickened $n$-stratum  to a point, each thickened $(n-1)$-stratum to an edge, each thickened $(n-2)$-stratum to a polygon and leaves the thickenings of $(n-3)$-strata unchanged. As a consequence, we obtain a graph with vertices corresponding to contracted $n$-strata and red edges dual to the $(n-1)$-strata of $X$. The normals of the $(n-1)$-strata equip this graph with an edge orientation. 

\begin{definition} Let $X$ be a compact  $n$-manifold with a fine stratification and classical defect data and $\Gamma$   the  1-skeleton of its dual  with the induced edge orientation.
\begin{compactenum}
\item A \textbf{fully reduced graph gauge configuration} on $X$ is an assignment of an element $m_s\in M_s$ to the  edge of $\Gamma$ dual to an $(n-1)$-stratum $s$.
\item A \textbf{fully reduced graph gauge transformation} on $X$ is an assignment of an element $g_t\in G_t$
 to the vertex of $\Gamma$ dual to an $n$-stratum $t$ of $X$. It acts on a gauge configuration by
 $$
 m_s\mapsto g_{L(s)}\rhd m_s\lhd g_{R(s)}^\inv.
 $$ 
 \end{compactenum}
\end{definition}

Clearly, fully reduced gauge configurations and gauge transformations on $X$ form a groupoid, namely the action groupoid $\mathcal A^0\sslash \mathcal G^0$ from Corollary \ref{cor:actred}. We can then rephrase  Corollary \ref{cor:actred} as follows.

\begin{corollary}\label{cor:redggrpd} Let $X$ be a compact  $n$-manifold with a fine stratification and classical defect data. Then the groupoid $\mathcal A^X\sslash \mathcal G^X$
of gauge configurations and gauge transformations on $X$ is equivalent to the groupoid of fully reduced graph gauge configurations and  transformations on $X$.
\end{corollary}

\section{Construction of the defect TQFT}
\label{sec:defect}

In this section we use the gauge groupoids of compact  stratified $n$-manifolds with defect data to construct a defect TQFT. 
We first show in Section \ref{subsec:classdeftqft} that the assignments of gauge groupoids to defect surfaces and defect cobordisms 
 define a symmetric monoidal functor $C:\mathrm{Cob}_3^{\mathrm{def}}\to \mathrm{span}^{\mathrm{fib}}(\Grpd)$.

We then consider for each defect cobordism $\Sigma_0\xrightarrow{\iota_0}M\xleftarrow{\iota_1} \Sigma_1$  the graph defined by its defect vertices and defect edges, with endpoints on the stratified surfaces $\Sigma_0$ and $\Sigma_1$. We show in Section \ref{subsec:edgeproj} that  each defect edge $e$ of $M$ defines projection functors from the gauge groupoid $\mathcal A^M\sslash \G^M$  to the action groupoid $\mathcal D_e=M_e\sslash G_e$ of the edge and natural transformations between them. 
In Subsection \ref{subsec:quantumtqft} we then combine these natural transformations  with the natural transformations at the defect vertices. 
This defines a graph coloured by representations of the gauge groupoid $\mathcal A^M\sslash\G^M$ and intertwiners between them. 

Its evaluation  assigns to each surface $\Sigma_j$ a functor $F_{\Sigma_j}: \mathcal A^{\Sigma_j}\sslash \mathcal G^{\Sigma_j}\to \Vect_\C$ and to the cobordism $M$ a natural transformation $\mu^M:F_{\Sigma_0}P_0\Rightarrow F_{\Sigma_1}P_1$, where $P_i:\A^M\sslash \G^M\to \A^{\Sigma_i}\sslash \G^{\Sigma_i}$ are the projection functors from 
Corollary \ref{cor: span fib}.
We then show that this defines a symmetric monoidal functor $G:\mathrm{Cob}_3^{\mathrm{def}}\to \mathrm{span}^{\mathrm{fib}}(\repGrpd)$. Composing it with the symmetric monoidal functor $L:\mathrm{span}^{\mathrm{fib}}(\repGrpd)\to \Vect_\C$ from Proposition \ref{prop:lfunctor}  defines the defect TQFT.

\subsection{Fibrant spans of groupoids from defect cobordisms}
\label{subsec:classdeftqft}

We consider for each stratified surface or cobordism $X$ with classical defect data the associated gauge groupoid $\mathcal A^X\sslash \G^X$ from Definition \ref{def:gaugegrpd}. We show that this assignment of gauge groupoids to stratified surfaces and cobordisms defines a symmetric monoidal functor from the defect cobordism category $\mathrm{Cob}_3^{\mathrm{def}}$ from Definition \ref{def:coddef} into the category $\mathrm{span}^{\mathrm{fib}}(\Grpd)$  from Lemma \ref{def:spangrpd}. Recall from Corollary \ref{cor:essfinite} that the gauge groupoid of any surface or cobordism  is essentially finite and from Corollary \ref{cor: span fib}  that a stratified compact 3-manifold $M$ with classical  defect data and boundary  $\partial M\cong\Sigma_0\amalg\Sigma_1$ defines a fibrant span 
of essentially finite groupoids.

\begin{theorem} \label{th:classcobfunc} There is a symmetric monoidal functor   $C: \mathrm{Cob}_3^{\mathrm{def}}\to \mathrm{span}^{\mathrm{fib}}(\Grpd)$, that assigns
\begin{compactitem}
\item  to a stratified surface $\Sigma$ with  defect data its gauge groupoid $\mathcal A^\Sigma\sslash\mathcal G^\Sigma$, 
\item  to  a defect cobordism $\Sigma_0\xrightarrow{\iota_0} M\xleftarrow{\iota_1} \Sigma_1$ the equivalence class of the fibrant span of groupoids
\begin{align*}
\mathcal A^{\Sigma_0}\sslash\G^{\Sigma_0} \xleftarrow{P_0} \mathcal A^M\sslash\G^M\xrightarrow{P_1} \mathcal A^{\Sigma_1}\sslash\G^{\Sigma_1}.
\end{align*}
\end{compactitem}
\end{theorem}

\begin{proof}
1.~It is clear that $C$ sends equivalent defect cobordisms to isomorphic and hence equivalent spans of essentially finite groupoids.  
Although the functors $T^{\Sigma_0}$, $T^{\Sigma_1}$ and $T^M$ are defined by choosing a triangulation, they are triangulation-independent up to natural isomorphisms, as discussed in Subsection \ref{sec:fundgrpd}.

2.~We show that $C$ respects identity morphisms:

We consider the cylinder cobordism $\Sigma \xrightarrow{\iota_0} \Sigma\times[0,1]\xleftarrow{\iota_1} \Sigma$.  As $M:=\Sigma\times[0,1]$ is equipped with the induced stratification, where each stratum $s$ of $\Sigma$ corresponds to a unique stratum $s\times[0,1]$ of $M$,
we can identify $Q^M=Q^\Sigma$ and take the insertions  for both boundary components as $\id:Q^\Sigma\to Q^M$. We then have $D^M=D^\Sigma$. 

By the construction  in Section \ref{sec:fundgrpd}, for each stratum $s$ of $\Sigma$ with associated space $\hat s$,  the corresponding space for the stratum $s\times[0,1]$ in $M$ is given by $\hat s\times[0,1]$. As  $\Pi_1:\Top\to \Grpd$ preserves products, 
we have $T^M=T^\Sigma\times\Pi_1[0,1]:\maq^M\to \Grpd$. The natural transformations $i^j: T^\Sigma\Rightarrow T^\Sigma\times\Pi_1[0,1]$ have components $i^j_s=\id_{\Pi_1(\hat s)}\times \{j\}:\Pi_1(\hat s)\to \Pi_1(\hat s)\times\Pi_1[0,1]$ for $j=0,1$.  Using the fact that $\Grpd$ is cartesian closed and that the functor $\GRPd(\G,-): \Grpd \to \Grpd$ is a right adjoint for any groupoid $\G$, we then obtain
\begin{align*}
\mathcal A^M\sslash \G^M&=\int_{x \in Q^\Sigma} \GRPd\big(T^{\Sigma}(x)\times \Pi_1[0,1], D^{\Sigma}(x)\big )\cong 
 \int_{x \in Q^\Sigma}  \GRPd\big(  \Pi_1[0,1], \GRPd(T^{\Sigma}(x), D^{\Sigma}(x))\big)\\
&\cong\GRPd\left(\Pi_1[0,1], \int_{x \in Q^\Sigma} \!\!\!\! \GRPd(T^{\Sigma}(x), D^{\Sigma}(x))\right)=(\mathcal A^\Sigma\sslash \G^\Sigma)^{\Pi_1[0,1]}.
\end{align*}
As the projections $P_j:\mathcal A^M\sslash \G^M\to\mathcal A^\Sigma\sslash \G^\Sigma$ are induced by the natural transformations $\iota^j:T^\Sigma\Rightarrow T^\Sigma\times\Pi_1[0,1]$, they are identified with the functors 
 $\Pi_1(\iota_0)^*,\Pi_1(\iota_1)^*: (\mathcal A^\Sigma\sslash \G^\Sigma)^{\Pi_1[0,1]}\to (\mathcal A^\Sigma\sslash \G^\Sigma)^\bullet\cong \mathcal A^\Sigma\sslash \G^\Sigma$ in \eqref{eq:def_ell}. This shows that  $C$ sends the equivalence class of the cylinder cobordism to the equivalence class of the span from \eqref{eq:def_ell} in $\mathrm{span}^{\mathrm{fib}}(\Grpd)$ for $\mathcal G=\mathcal A^\Sigma\sslash \mathcal G^\Sigma$. 
By Lemma \ref{lem:pi1interval}  this is the identity morphism in $\mathrm{span}^{\mathrm{fib}}(\Grpd)$.

\medskip
3.~We show that $C$ respects the composition of morphisms:

We glue cobordisms $\Sigma_0\xrightarrow{f_0} M_{01}\xleftarrow{f_1}\Sigma_1$ and $\Sigma_1\xrightarrow{f'_1} M_{12}\xleftarrow{f_2} \Sigma_2$  to a cobordism $\Sigma_0 \xrightarrow{j_0} M_{012}\xleftarrow{j_2} \Sigma_2$ along the stratified surface $\Sigma_1$. 
This amounts to gluing for each $k$-stratum $s_1$ of $\Sigma_1$ the associated $(k+1)$-strata $s_{01}$ of $M_{01}$ and $s_{12}$ of $M_{12}$ along $s_1$.

We denote by $Q_j:=Q^{\Sigma_j}$ for $j=0,1,2$ and $Q_j:=Q^{M_j}$ for $j=01,12, 012$ the associated graded graphs and by $\maq_j$ the associated free categories. By Lemma \ref{lem:inclusions are insertive}, the maps $f_0$,
$f_1$, $f'_1$  and $f_2$ define insertions, which we denote by the same symbols.

\begin{lemma}\label{lem:graphglue} Gluing $M_{01}$ and $M_{12}$ along $\Sigma_1$  yields the following commuting diagrams, 
where  the middle diamonds are pushouts:
\begin{align}\label{eq:pushq1}
\begin{gathered}
\xymatrix@C=18pt{ Q_0\ar[dr]^{f_0} \ar@/_2pc/[ddrr]_{j_0} & &  Q_1\ar[dl]_{f_1} \ar[dr]^{f'_1} \ar[dd]^{j_1} &&  Q_2 \ar[dl]_{f_2}  \ar@/^2pc/[ddll]^{j_2} \\ &  Q_{01}\ar[dr]_{j_{01}} && \ Q_{12} \ar[dl]^{j_{12}} \\
&& Q_{012}}
\end{gathered}
\quad\textrm{ and }\quad
\begin{gathered}
\xymatrix@C=18pt{ \mathcal Q_0\ar[dr]^{F_0} \ar@/_2pc/[ddrr]_{J_0} & & \mathcal Q_1\ar[dl]_{F_1} \ar[dr]^{F'_1} \ar[dd]^{J_1} && \mathcal Q_2 \ar[dl]_{F_2}  \ar@/^2pc/[ddll]^{J_2} \\ & \mathcal Q_{01}\ar[dr]_{J_{01}} && \mathcal Q_{12} \ar[dl]^{J_{12}} \\
&& \mathcal Q_{012}.}
\end{gathered}
\end{align}
 {All} arrows in the first diagram are insertions, and in the second discrete opfibrations.
\end{lemma}

\begin{proof}
{The second diagram is obtained by applying the free category functor   $L: \Graph\to \mathrm{Cat}$ to the first diagram}. {As}
a left adjoint, the  functor  $L: \Graph\to \mathrm{Cat}$ preserves pushouts, and by Lemma \ref{lem:I-faithfulandInsertive} it sends insertions to discrete opfibrations. Hence, it is sufficient to prove that the middle diamond in the first diagram 
is a pushout and all arrows are insertions.

The first statement follows, because the strata of the manifolds $M_{01}$ and $M_{12}$ are glued
 along points $x\in \Sigma_1$ to form the strata of $M_{012}$.
 This shows 
 that the vertex set of the graph $Q_{012}$ associated to $M_{012}$ is indeed the pushout of the vertex sets of the graphs $Q_{01}$ and $Q_{12}$ along $Q_1$.

 To show the corresponding statement for their edge sets, consider a stratum $s$ of $\Sigma_1$.
 Because all bundles of local strata are trivial, the sets of edges of  $Q_{01}$, $Q_{12}$, $Q_1$ and $Q_{012}$ starting at $f_1(s)$, $f_1'(s)$,  $s$, and $j_1(s)$ are all canonically identified with the sets of strata in any  special neighbourhood of any point on these strata. 
 For any point $x\in s$, forming the pushout of special neighbourhoods $U_{01}$ of $f_1(x)$ in $M_{01}$ and $U_{12}$ of $f'_{1}(x)$ in $M_{12}$ 
 with $U_{01}\cap \partial M_{01}=f_1(U)$ and $U_{12}\cap \partial M_{12}=f'_1(U)$ for a special neighbourhood $U$ of $x$ in $\Sigma_1$ gives a special neighbourhood of $j_1(x)$ in $M_{012}$. This shows that  the diamond at the left of \eqref{eq:pushq1} is a pushout.

Consider  a local $(k+1)$-stratum $q:s\to t$ at a {$k$-stratum $s$} of $M_{012}$. If $s=j_1(s_1)$ for a $(k-1)$-stratum $s_1$  of $\Sigma_1$, then, by construction,  $q:s\to t$   corresponds to a unique local $(k+1)$-stratum $q_{01}:f_1(s_{1})\to t_{01}$, a unique local $(k+1)$-stratum  $q_{12}:f_1'(s_{1})\to t_{12}$ and a unique local $k$-stratum $q_1:s_1\to t_1$. {If $j_1^\inv(s)=\emptyset$, then $q$ corresponds either to a unique local $(k+1)$-stratum $q_{01}: f_1(s_1)\to t_{01}$ or to a unique local $(k+1)$-stratum $q_{12}: f'_1(s_1)\to t_{12}$.}
This  shows that maps 
$j_{01}$, $j_1$ and $j_{12}$ are  insertions.
As $f_0,f_1,f'_1, f_2$ are insertions by assumption and composites of insertions are insertions, this also holds for the maps $j_0$ and $j_2$. 
 \end{proof}

From \eqref{eq:pushq1} we  then have  the following diagrams of categories, functors and natural transformations 
\begin{equation}\label{diags:TandD-0}
\begin{gathered}\hskip-1cm
\xymatrix@C=20pt{&\mathcal{Q}_0\ar@/_2pc/[ddrr]_{T_0}
\xtwocell[ddrr]{}<>{  ^<3> \tau_0 }
\ar[dr]^{F_0}&& \mathcal{Q}_1\ar[dl]_{F_1}
\xtwocell[dd]{}<>{  <3> \tau_1 }
\xtwocell[dd]{}<>{  ^<-3> {\tau_1'} }
\ar[dr]^{F_1'} \ar[dd]|{T_1} &&\mathcal{Q}_2\ar[dl]_{F_2}\ar@/^2pc/[ddll]^{T_2}
\xtwocell[ddll]{}<>{  <-3> \tau_2 } \\
&&\ar[dr]_{T_{01}} \mathcal{Q}_{01} && \mathcal{Q}_{12}\ar[dl]^{{T_{12}}}\\
&&&\Grpd
} \xymatrix@C=20pt{&\mathcal{Q}_0\ar[dr]^{F_0}  \ar@/_2pc/[ddrr]_{D_0}
&& \mathcal{Q}_1\ar[dl]_{F_1}
\ar[dr]^{F_1'} \ar[dd]|{D_1} &&\mathcal{Q}_2\ar[dl]_{F_2}   \ar@/^2pc/[ddll]^{D_2}\\
&&\ar[dr]_{D_{01}} \mathcal{Q}_{01} && \mathcal{Q}_{12}\ar[dl]^{{D_{12}}}\\
&&&\Grpd
},\end{gathered}
\end{equation}
where  
$D_j:=D^{\Sigma_j}, T_j:=T^{\Sigma_j}$ for $j=0,1,2$ and
$D_j:=D^{M_j},T_j:=T^{M_j}$ for $j=01,12$. The natural transformations $\tau_j$ and $\tau'_j$ are the ones from Proposition \ref{prop:tfunctor}. 
The diagram on the right commutes. The arrows in the one on the left are related by natural transformations, as indicated in the diagram.
Applying Theorem \ref{thm:main-end}  for $\V=\Grpd$ yields the following commuting diagram, in which the middle diamond is a pullback
\begin{equation}\label{eq:spandiagcom}
\vcenter{\xymatrix{&\End_{\maq_0}[T_0,D_0]  & \End_{\maq_1}[T_1,D_1] &  \End_{\maq_2}[T_2,D_2]\\
& \End_{\maq_{01}}[T_{01}, D_{01}] \ar@{->}[ru]_{\hat\T_1}  \ar@{->}[u]^{ \hat \T_0} & & \End_{\maq_{12}}[T_{12}, D_{12}] \ar@{->}[u]_{\hat \T_2}  \ar@{->}[lu]^{\hat \T'_1}\\
& &\End_{\maq_{012}}[T_{012}, D_{012}] \ar@{->}[ru]_{\hat I_{12}}  \ar@{->}[lu]^{ \hat I_{01}}
}}
\end{equation}
where we abbreviate
$$
\End_\maq [T,D]:=\int_{s\in \maq} \GRPd\big(T(s), D(s)\big).
$$
The  functor $D_{012}:\maq_{012}\to \Grpd$ in \eqref{eq:spandiagcom} is induced by the  pushout from the  commuting diagram  \eqref{eq:pushq1} and the diagram  on the right in \eqref{diags:TandD-0}
and  hence satisfies $D_{012} J_{01}=D_{01}$ and $D_{012}J_{12}=D_{12}$.  The functor $T_{012}:\maq_{012}\to \Grpd$ in \eqref{eq:spandiagcom} is given by
formula \eqref{eq:deft012} in the proof of Theorem \ref{thm:main-end} as
 \begin{align}\label{eq:deft0122}
T_{012}=\Lan_{J_{01}} T_{01}\!\!\!\!\coprod_{\Lan_{J_1} T_1} \!\!\!\!\Lan_{J_{12}} T_{12}: \mathcal Q_{012}\to \Grpd.
\end{align}
We now use  Definition \ref{def:gaugegrpd}  of the gauge groupoid and compare the definitions of $\hat\T_0$, $\hat\T_1$, $\hat \T'_1$, $\hat \T_2$ in   the proof of Theorem \ref{thm:main-end} via  \eqref{eq:t0def} and \eqref{eq:totherdef} with the  boundary projections in Remark \ref{rem:boundproj} and Corollary \ref{cor: span fib}. This allows us to rewrite diagram \eqref{eq:spandiagcom} as
\begin{equation}\label{eq:spandiagcom2}\vcenter{
\xymatrix{& \A^{\Sigma_0}\sslash \G^{\Sigma_0}  & \A^{\Sigma_1}\sslash \G^{\Sigma_1}  &  \A^{\Sigma_2}\sslash \G^{\Sigma_2} \\
& \A^{M_{01}}\sslash \G^{M_{01}}  \ar@{->}[ru]_{P_1}  \ar@{->}[u]^{P_0} & & \A^{M_{12}}\sslash \G^{M_{12}} \ar@{->}[u]_{P_2}  \ar@{->}[lu]^{P'_1}\\
& &\End_{\maq_{012}}[T_{012}, D_{012}] \ar@{->}[ru]_{\hat I_{12}}  \ar@{->}[lu]^{ \hat I_{01}}
}}
\end{equation}
The functor $D_{012}:\maq_{012}\to \Grpd$  in diagrams \eqref{eq:spandiagcom} and \eqref{eq:spandiagcom2} coincides with the functor $D^{M_{012}}:\maq_{012}\to\Grpd$. This follows, because the defect data on the strata of $M_{01}$ and $M_{12}$ that are glued along strata of $\Sigma_1$ matches, and by construction of the functor $D^X:\maq^X\to \Grpd$ in Lemma \ref{cor:deltabound}. 

The functors $\hat I_{01}$ and $\hat I_{12}$ are induced by the ends from the natural transformations $I_{01}: \Lan_{J_{01}}T_{01}\Rightarrow T_{012}$ and $I_{12}:\Lan_{J_{12}}T_{12}\Rightarrow T_{012}$ defined by the pushout \eqref{eq:deft0122}. 

It  remains to show that the functor $T_{012}:\maq_{012}\to\Grpd$ in diagrams \eqref{eq:spandiagcom} and \eqref{eq:spandiagcom2} coincides with  the functor $T^{M_{012}}:\maq_{012}\to \Grpd$ of the defect cobordism $M_{012}$.  

{For this, we consider  the   functors $B^{X}:\maq^X\to \Top$ from Proposition \ref{prop:tfunctor} with $T^X=\Pi_1B^X$ and  abbreviate $B_j=B^{\Sigma_j}$ for $j=0,1,2$ and $B_{j}:=B^{M_j}$ for $j=01,12$. The  natural transformations from Proposition \ref{prop:tfunctor}  are denoted $\beta_1: B_{1}\Rightarrow B_{01} F_1$ and $\beta'_1: B_1\Rightarrow B_{12}F'_1$. The  counterparts of the 
natural transformations $\T_1$ and $\T'_1$  are   $\B_1:\Lan_{J_1} B_1\Rightarrow \Lan_{J_{01}}B_{01}$ and $\B'_1:\Lan_{J_1} B_1\Rightarrow \Lan_{J_{12}}B_{12}$, given by  \eqref{eq:t0def} and \eqref{eq:totherdef}.}

\begin{lemma} \label{lem:bfunctor} The functor $B^{M_{012}}: \maq_{012}\to \Top$ is the pushout in the functor category $\Grpd^{\maq_{012}}$ of $\Lan_{J_{01}} B_{01}$ and $\Lan_{J_{12}}B_{12}$ along the natural transformations $\B_1: \Lan_{J_{1}}B_1\Rightarrow \Lan_{J_{01}} B_{01}$ and $\B'_1:\Lan_{J_1}B_1\Rightarrow \Lan_{J_{12}} B_{12}$, 
$$
B^{M_{012}}\cong\Lan_{J_{01}} B_{01}\coprod_{\Lan_{J_1}B_1} \Lan_{J_{12}}B_{12}.
$$
\end{lemma}

\begin{proof}
Choose a triangulation of $M_{012}$  that induces triangulations of $M_{01}$, $M_{12}$ and $\Sigma_0$, $\Sigma_1$, $\Sigma_2$,  such that all skeleta of the stratifications are subcomplexes.
By the construction in Section \ref{sec:fundgrpd}, the space $\hat s$ for each stratum $s$ of $M_{012}$ is  
the pushout
\begin{align}\label{eq:pushouLem1}
\vcenter{\xymatrix@C=40pt{
\hat s & & \coprod_{s_{01}\in j_{01}^\inv (s)} \hat s_{01} \ar[ll]\\
\coprod_{s_{12}\in j_{12}^\inv(s)} \hat s_{12} \ar[u]& &\coprod_{s_1\in j_1^\inv(s)}\hat s_1, \ar[u]_-{(\B_1)_{s}=\langle \iota_{s_1\partial}\rangle} \ar[ll]^{(\B'_1)_{s}=\langle \iota'_{s_1\partial}\rangle}
}}
\end{align}
where $\iota_{s_1\partial}=(\beta_1)_{s_1}: \hat s_1\to \hat s_{01}$ and $\iota'_{s_1\partial}=(\beta'_1)_{s_1}:\hat s_1\to \hat s_{12}$ are the components of the natural transformations $\beta_1:B_1\Rightarrow B_{01}F_1$ and $\beta'_1: B_1\Rightarrow B_{12}F'_1$  from Proposition \ref{prop:tfunctor} and $\langle \iota_{s_1\partial}\rangle$ and $\langle \iota_{s_2\partial}\rangle$ are induced by them via   the universal properties of the coproducts. 
This shows in particular that the maps $(\B_1)_s$ and $(\B'_1)_s$ are embeddings of subcomplexes for finite simplicial complexes.

By Lemma \ref{lem:graphglue}, the functors $J_1$, $J_{01}$ and $J_{12}$ are discrete opfibrations. By comparing with Lemma \ref{lem:kanweakinsertion} we can rewrite diagram \eqref{eq:pushouLem1} in terms of the functors $B^X:\maq^X\to\Top$ from Proposition \ref{prop:tfunctor} as 
\begin{align}\label{eq:Kanexpression}
\vcenter{\xymatrix{ B^{M_{012}}(s) &   \Lan_{J_{01}} B_{01}(s) \ar[l]_{i^{01}_s}\\
\Lan_{J_{12}}B_{12}(s) \ar[u]^{i^{12}_s}& \Lan_{J_1}B_1(s). \ar[u]_{(\B_1)_s} \ar[l]^{(\B'_1)_s}}
}
\end{align}
The map  between the spaces $\Lan_{J_x}B_x(s)$ and $\Lan_{J_x}B_x(t)$ for $x=1,01,12$  associated to a local stratum  $q:s\to t$ of $M_{012}$
coincides with the ones in the commuting diagram 
in  Lemma \ref{lem:kanweakinsertion}  that describes the Kan extension on the morphisms. Combining the commuting  diagrams \eqref{eq:pushouLem1} for a local stratum $q:s\to t$ of $M_{012}$ then yields
that the map $B^{M_{012}}(q)=\iota_q:\hat s\to\hat t$  is induced by the universal property of this pushout.
\end{proof}

\begin{lemma} The functor $T^{M_{012}}:\maq_{012}\to \Grpd$ is given by the pushout
$$
T^{M_{012}}=T_{012}\cong\Lan_{J_{01}} T_{01}\!\!\!\!\coprod_{\Lan_{J_1} T_1} \!\!\!\!\Lan_{J_{12}} T_{12}
$$
\end{lemma}

\begin{proof} By the proof of Lemma \ref{lem:bfunctor} 
the components of the natural transformations $\B_1: \Lan_{J_1}B_1\Rightarrow \Lan_{J_{01}}B_{01}$ and $\B'_1:\Lan_{J_1}B_1\Rightarrow \Lan_{J_{12}}B_{12}$ in \eqref{eq:Kanexpression} are embeddings of subcomplexes for a finite simplicial complex. Hence, 
  they are cofibrations. The Seifert van Kampen theorem for fundamental groupoids, as stated in \cite[9.1.2]{Br} {implies for each stratum $s$ of $M_{012}$ we have}  a pushout in $\Grpd$
  \begin{align*}\vcenter{\xymatrix@C=40pt{ \Pi_1 B^{M_{012}}(s) &   \Pi_1  \Lan_{J_{01}} B_{01}(s) \ar[l]_{i^{01}_s}\\
\Pi_1  \Lan_{J_{12}}B_{12}(s) \ar[u]^{i^{12}_s}& \Pi_1 \Lan_{J_1}B_1(s). \ar[u]_{\Pi_1 (\B_1)_s} \ar[l]^{\Pi_1 (\B'_1)_s}}
}
\end{align*}
As colimits in functor categories are objectwise, this yields
$$
T^{M_{012}}=\Pi_1B^{M_{012}}=(\Pi_1\Lan_{J_{01}}B_{01})\amalg_{\Pi_1\Lan_{J_1}B_1} (\Pi_1\Lan_{J_{12}}B_{12}).
$$
It remains to show that $\Pi_1:\Top\to\Grpd$ preserves the Kan extensions $\Lan_{J_x} B_x$ for $x=1,01,12$, that is,
 $\Pi_1 \Lan_{J_x} B_x \cong  \Lan_{J_x}\Pi_1 B_x$ for $x=1,01,12$. As the maps $j_x:Q_x\to Q_{012}$ are insertions by Lemma \ref{lem:graphglue},  Lemma \ref{lem:kanweakinsertion} implies  $\Lan_{J_x} F_x(v')=\amalg_{v\in j_x^\inv(v')} F_x(v)$ for any functor $F_x: \maq_x\to\mac$ into a cocomplete category $\mac$ and any vertex $v'$ of $Q_{012}$. As $\Pi_1:\Top\to\Grpd$ preserves coproducts, the claim follows.
 \end{proof}

Thus, we have  $T^{M_{012}}=T_{012}$. This allows us
 to replace  $\End_{Q_{012}}[T_{012}, D_{012}]=\mathcal A^{M_{012}}\sslash \G^{M_{012}}$ in diagram \eqref{eq:spandiagcom2} and shows that $C$ sends composites of cobordisms to the pullbacks of the associated spans.
Hence, we have a functor $C: \mathrm{Cob}_3^{\mathrm{def}}\to \mathrm{span}^{\mathrm{fib}}(\Grpd)$. It is clear  that $C$ is symmetric monoidal.
\end{proof}

\subsection{Projection functors for edges and natural transformations between them}
\label{subsec:edgeproj}

As the second step in the construction of the defect TQFT, we  investigate the  groupoids associated with the defect edges of a compact stratified $n$-manifold $X$. For each defect edge $e$ we  construct projection functors to the groupoids $\mathcal D_e$ and  natural transformations between them. 

Recall from Definition \mbox{\ref{def:gaugegrpd}} that each stratum $s$ of $X$ gives rise to a functor, $P_s:\mathcal A^X\sslash \G^X\to \mathcal D_e^{\Pi_1(\hat e)}$.

\begin{lemma} \label{cor:indstrata-tweaked} 
Let $X$ be a compact  stratified $n$-manifold  with classical defect data. 
\begin{compactitem}
\item For each   edge $e$ and $x \in \hat{e}$ we have a functor $R_e^x:\mathcal A^X\sslash \mathcal G^X\to \mathcal D_e$.
\item  
For each edge $e$ 
that is not an isolated loop, we have 
functors $P_e^s, P_e^t:\mathcal A^X\sslash \mathcal G^X\to \mathcal D_e$ 
such that the following diagrams commute
\begin{align}\label{eq:projddiag-in}
&(i)\begin{gathered}
\xymatrix{ & \mathcal A^X\sslash \mathcal G^X \ar[ld]_{P^s_e} \ar[d]^{P_v} \\
\mathcal D_e & \ar[l]^{D_{s(e)}}\mathcal D_v.
}
\end{gathered} &  
&(ii) \begin{gathered}
\xymatrix{  \mathcal A^X\sslash \mathcal G^X \ar[d]_{P_v} \ar[rd]^{P_e^t}\\
\mathcal D_v \ar[r]_{D_{t(e)}} & \mathcal D_e.
}
\end{gathered}\\
&(iii) 
\begin{gathered}
\xymatrix{ \mathcal A^X\sslash \mathcal G^X \ar[d]_{P_e^s} \ar[r]^{P_\partial}& \mathcal A^{\partial X}\sslash \mathcal G^{\partial X} \ar[ld]^{P_v^{\partial X}}\\
\mathcal D_e^{X}=\mathcal D_v^{\partial X}
}
\end{gathered}
& &(iv)
\begin{gathered}
\xymatrix{ \mathcal A^X\sslash \mathcal G^X \ar[d]_{P_e^t} \ar[r]^{P_\partial}& \mathcal A^{\partial X}\sslash \mathcal G^{\partial X} \ar[ld]^{P_v^{\partial X}}\\
\mathcal D_e^{ X}=\mathcal D_v^{\partial X},}
\end{gathered}\nonumber
\end{align}
\begin{compactenum}[(i)]
\item if $e$ is outgoing at a vertex $v$ of $X$ with associated  local stratum $s(e): v \to e$,
\item if $e$ is incoming at a vertex $v$ of $X$ with  associated local stratum $t(e): v \to e$,
\item if $e$ is outgoing at a boundary vertex $v$ of $\partial X$,
\item if $e$ is incoming at a boundary vertex $v$ of $\partial X$.
\end{compactenum}
\end{compactitem}
\end{lemma}

\begin{proof} 
For a stratum $t$  of $X$ or of $\partial X$ we denote by $P_t: \mathcal A^X\sslash\mathcal G^X\to \mathcal D_t^{\Pi_1(\hat t)}$
or $P^{\partial X}_t: \mathcal A^{\partial X}\sslash\mathcal G^{\partial X}\to \mathcal D_t^{\Pi_1(\hat t)}$
the functor from Definition \ref{def:gaugegrpd}. 
The functor $R^x_e$ for an edge $e$ is then given as the composite
\begin{align}\label{eq:defRex}
R^x_e:\mathcal A^X\sslash \mathcal G^X\xrightarrow{P_e} \mathcal D_e^{\Pi_1(\hat e)} \xrightarrow{\Pi_1(\iota_x)^*} \mathcal D_e^\bullet\cong \mathcal D_e,
\end{align}
where $\iota_x:\bullet\to \hat e$, $\bullet \mapsto x$. 
If $e$ is not an isolated loop, we identify $\hat e\cong [0,1]$ and define
\begin{align}\label{eq:pestdef}
P^s_e=R^0_e:\mathcal A^X\sslash \mathcal G^X\xrightarrow{P_e} \mathcal D_e^{\Pi_1(\hat e)} \xrightarrow{\Pi_1(\iota_0)^*}  \mathcal D_e\qquad P^t_e=R^1_e:\mathcal A^X\sslash \mathcal G^X\xrightarrow{P_e} \mathcal D_e^{\Pi_1(\hat e)} \xrightarrow{\Pi_1(\iota_1)^*}  \mathcal D_e.
\end{align}
We also have $\Pi_1(\hat v)=\Pi_1(\bullet)\cong \bullet$ for a vertex $v$ and  canonical isomorphisms $\G^{\Pi_1(\hat v)}\cong  \G$ for each groupoid $\G$.

Cases (i) and (ii): If $v=s(e)$ is the starting vertex of the edge $e$, the associated 1-stratum $q: v\to e$ corresponds to the inclusion $\iota_0: \bullet \to [0,1]$, and if $v=t(e)$ is the target vertex, it corresponds to the inclusion $\iota_1: \bullet \to [0,1]$. We then have $T^X(q)=\Pi_1(\iota_0): \bullet\to \Pi_1[0,1]$ for $v=s(e)$ and and $T^X(q)=\Pi_1(\iota_1): \bullet\to \Pi_1[0,1]$ for $v=t(e)$.
The outer paths in the following diagram  commute by definition of the end, and the inner vertical arrow is $P_e^s$ for $v=s(e)$ or $P_e^t$ for $v=t(e)$
\begin{equation}\label{eq: commun-Ps}
\xymatrix
{  & \mathcal{A}^X\sslash \mathcal{G}^X\ar[dl]_{P_v} \ar[dr]^{P_e} \ar[d]^{P_e^{v}}\\
            \mathcal{D}_v^{\Pi_1 (\hat{v})}\cong \mathcal D_v\ar[r]_{D_q}    &\mathcal D_e\cong \mathcal{D}_e^{\Pi_1 ({\hat{v}})}   &\mathcal{D}_e^{\Pi_1 (\hat{e})}.   \ar[l]^{\Pi_1(q)^*}}
\end{equation}
Cases (iii) and (iv): If $e$ is incident at a boundary vertex $v$,  
the associated map $\iota_{v\partial}: \bullet\to [0,1]$ from \eqref{eq:stratboundemb} is given by $\iota_{v\partial}=\iota_0: \bullet\to [0,1]$, if $v$ is the starting vertex of $e$, and by $\iota_{v\partial}=\iota_1: \bullet\to [0,1]$, if $v$ is the target vertex of $e$. Hence, the associated functor $\tau_v: \Pi_1(\hat v)\to \Pi_1(\hat e)$ is given by $\tau_v=\Pi_1(\iota_0)$ and $\tau_v=\Pi_1(\iota_1)$, respectively.  Corollary \ref{prop:projfunc}  and Remark \ref{rem:boundproj} then yield commuting diagrams 
\begin{align*}
 \xymatrix{ \mathcal A^X\sslash \mathcal G^X \ar[r]^{P_e} \ar[rd]^{P_e^t}\ar[d]_{P_\partial} & \mathcal D_e^{\Pi_1[0,1]} \ar[d]^{\Pi_1(\iota_0)^*}\\
 \mathcal{A}^{\partial X}\sslash\mathcal{G}^{\partial X} \ar[r]_{P_v^{\partial X}}& \mathcal D_e=\mathcal D_v
 }   \qquad  \xymatrix{ \mathcal A^X\sslash \mathcal G^X \ar[r]^{P_e} \ar[rd]^{P_e^s}\ar[d]_{P_\partial} & \mathcal D_e^{\Pi_1[0,1]} \ar[d]^{\Pi_1(\iota_1)^*}\\
 \mathcal{A}^{\partial X}\sslash\mathcal{G}^{\partial X} \ar[r]_{P_v^{\partial X}}& \mathcal D_e=\mathcal D_v.
 }   
\end{align*}
\end{proof}

\begin{lemma} \label{cor:indstrata-tweaked-nat} Let $X$ be a compact stratified $n$-manifold  with classical defect data and  $e$ an edge of $X$. 
\begin{compactitem}
\item Each path $\gamma:x\to y$ in $\hat e$ {defines} a natural transformation $H_e^{[\gamma]}: R_e^x\Rightarrow R_e^y$ that depends only on the homotopy class of $\gamma$ and satisfies for all paths $\delta:y\to z$ in $\hat e$
$$H^{[\delta]\cdot[\gamma]}_e=H_e^{[\delta]}\circ H^{[\gamma]}_e.$$ 
\item  Each edge $e$ that is incident to at least one vertex or boundary vertex {defines}
\begin{compactenum}[(i)]
    
\item   a unique natural transformation $P_e^d: P_e^s\Rightarrow P_e^t$,
    \item   for any  $x\in \hat e$  unique  natural transformations $H^0_{x}: P_e^s \Rightarrow R_e^x$ and $H^1_x: R_e^x \Rightarrow  P_e^t$ with $H^1_x \circ H^0_x=P_e^d$. 
    \end{compactenum}
\end{compactitem}
\end{lemma}

\begin{proof} The functors $R_e^x, R_e^y:\mathcal A^X\sslash\mathcal G^X\to \mathcal D_e$ from Lemma \ref{cor:indstrata-tweaked} are given by \eqref{eq:defRex}.
We define 
$$
H_e^{[\gamma]}=\Pi_1([\gamma])^*P_e: R^x_e\Rightarrow R_e^y.
$$
If $e$ is incident at at least one vertex or boundary vertex $v$, we can identify $\hat e\cong [0,1]$. The functors $P_e^s,P_e^t:\mathcal A^X\sslash\mathcal G^X\to \mathcal D_e$ from Lemma \ref{cor:indstrata-tweaked} are then given by \eqref{eq:pestdef}. We choose paths   $\delta:0\to 1$, $\delta_0:0\to x$ and $\delta_1: x\to 1$ in $[0,1]$ and set
$P_e^d=H^{[\delta]}_e$,  $H_x^0=H_e^{[\delta_0]}$ and $H_x^1=H_e^{[\delta_1]}$.
\end{proof}

\subsection{Quantum defect data and the defect TQFT}
\label{subsec:quantumtqft}

We now use the projection functors and natural transformations  from Section \ref{subsec:edgeproj} and the assignment of defect data in  Definition \ref{def:defectdata} to associate to each defect cobordism $M$ a graph that is labelled with representations of the gauge groupoid $\mathcal A^M\sslash \mathcal G^M$ and intertwiners between them.

Let $X$ be a compact stratified $n$-manifold with an assignment of defect data. Definition \ref{def:gaugegrpd} assigns   
\begin{compactitem}
\item to each edge $e$  of $X$ a functor $P_e:\mathcal A^X\sslash\mathcal G^X\to \mathcal D_e^{\Pi_1(\hat e)}$,
\item to each vertex $v$ of $X$ a  functor $P_v:\mathcal A^X\sslash \mathcal G^X\to \mathcal D_v$.
\end{compactitem}
Recall from Definition \ref{def:defectdata} that an assignment of quantum defect data assigns
\begin{compactitem}
\item  to each codim 2-stratum $u$ of $X$ a representation $\rho_u: \mathcal D_u\to \Vect_\C$,
\item  to each codim 3-stratum $v$ an  intertwiner $\sigma_v: \rho_v^{t}\Rightarrow \rho_v^{s}$, where $\rho_v^s,\rho_v^t: \mathcal D_v\to \Vect_\C$ are given by \eqref{eq:rhovdef} and 
the groupoids $\mathcal D_u=D^X(u)$ by Proposition \ref{cor:deltabound}.
\end{compactitem}
Lemma \ref{cor:indstrata-tweaked} and Lemma \ref{cor:indstrata-tweaked-nat} assign
\begin{compactitem}
\item to each edge $e$ of $X$ that is not an isolated loop functors $P_e^s,P_e^t:\mathcal A^X\sslash \mathcal G^X\to \mathcal D_e$ and a natural transformation $P_e^d:P_e^s\Rightarrow P_e^t$,
\item to each isolated loop $e$ of $X$ and each $x\in \hat e$ a functor $R_e^x:\mathcal A^X\sslash \mathcal G^X\to \mathcal D_e$ and a natural transformation $H_e^{[\delta]}: R^x_e\Rightarrow R^x_e$, where $\delta:x\to x$ is a path that goes once around $\hat e$ in the direction of its orientation. 
\end{compactitem}

We assign to each stratified surface $\Sigma$ with defect data  a functor $F_\Sigma: \mathcal A^\Sigma\sslash \mathcal G^\Sigma\to \Vect_\C$  as follows. For a vertex $v$ of $\Sigma$ we set  $F_v^{\Sigma}=\rho^*_vP_v: \mathcal A^\Sigma\sslash \mathcal G^\Sigma\to \Vect_\C$, if $v$ has negative orientation, and $F_v^\Sigma=\rho_vP_v: \mathcal A^\Sigma\sslash \mathcal G^\Sigma\to \Vect_\C$, if $v$ has positive orientation. We then  take the tensor product of these functors over all vertices of $\Sigma$
\begin{align}\label{eq:fsigmadef}
F_\Sigma=\bigotimes_{v\in S_0^\Sigma} F^{\Sigma}_v:\mathcal A^\Sigma\sslash \mathcal G^\Sigma\to \Vect_\C.
\end{align}
To each stratified cobordism $\Sigma_0\xrightarrow{i_0} M\xleftarrow{i_1} \Sigma_1$  with defect data we will assign a natural transformation $\mu^M: F_{\Sigma_0}P_0\Rightarrow F_{\Sigma_1}P_1$, where $P_j:\mathcal A^M\sslash \G^M\to \mathcal A^{\Sigma_j}\sslash \G^{\Sigma_j}$ is the projection functor from Corollary \ref{cor: span fib}. This natural transformation is constructed from a directed graph $\Gamma^M$, whose edges and vertices are coloured with representations of the groupoid $\mathcal A^M\sslash \mathcal G^M$ and intertwiners between them. 

 As the functor category $\Vect_\C^{\mathcal A^M\sslash \mathcal G^M}$ is  spherical symmetric monoidal, cf.~Section \ref{sec:groupoidbasic}, this defines a diagram for a spherical symmetric monoidal category, whose evaluation is an intertwiner between the representations $F_{\Sigma_0}P_0
 :\mathcal A^M\sslash \mathcal G^M\to \Vect_\C$ and $F_{\Sigma_1}P_1:\mathcal A^M\sslash \mathcal G^M\to \Vect_\C$.

The graph $\Gamma^M$ is constructed as follows. We consider the 1-skeleton of the stratification and   add to each edge $e$ of $M$ a bivalent vertex $v_e$. For an isolated loop $e$, we also choose a vertex $\hat v_e\in \hat e$ with $f_e(\hat v_e)=v_e$.
This replaces each edge of $M$ that is not an isolated loop   by a pair of edges corresponding to its starting and target end. All edges created in this way are equipped with the induced orientations. The open edge ends of  $\Gamma^M$  correspond to the boundary vertices of $M$.

We assign to the directed graph  $\Gamma^M$ the following data:
\begin{compactitem}
\item to each edge $e$ of $\Gamma^M$  that is outgoing at a  vertex or boundary vertex $v$ of $M$ the functor 
\begin{align}\label{eq:fes}
F_e^s:=\rho_eP_e^s: \mathcal A^M\sslash \mathcal G^M\to \Vect_\C,
\end{align}
\item to each edge $e$ of $\Gamma^M$ that is incoming at a vertex or boundary vertex $v$ of $M$ the functor 
\begin{align}\label{eq:fet}
F_e^t:=\rho_eP_e^t: \mathcal A^M\sslash \mathcal G^M\to \Vect_\C,
\end{align}
\item to each edge of $\Gamma^M$ that arises from an isolated loop,  the functor defined by \eqref{eq:defRex}
\begin{align}
F^{\hat v_e}_e:=\rho_eR_e^{\hat v_e}: \mathcal A^M\sslash \mathcal G^M\to \Vect_\C,
\end{align} 

\item to a vertex of $\Gamma^M$ that arises from a vertex $v$ of $M$ the natural transformation 
\begin{align}\label{eq:vertexornot}
\mu^v:=\sigma_v P_v: \rho_v^{t}P_v\Rightarrow\rho_v^{s}P_v,
\end{align}

\item to each  new bivalent vertex $v_e$ of $\Gamma^M$ that is not  on an isolated loop the 
natural transformation 
\begin{align}\label{eq:nonloopnat}
\mu^e:=\rho_eP_e^d:  F_e^s\Rightarrow F_e^t,
\end{align}

\item to each bivalent vertex $v_e$  on an isolated loop 
the natural transformation 
\begin{align}\label{eq:loopnat}
\rho_e H_e^{[\delta]}: F_e^{\hat v_e}\Rightarrow F_e^{\hat v_e}
\end{align} for a path $\delta: \hat v_e\to \hat v_e$ that circles $\hat e$ once in the direction of its orientation. \end{compactitem} 

This defines a graph $\Gamma^M_{col}$ coloured by objects and morphisms  in  $\Vect_\C^{\mathcal A^M\sslash \mathcal G^M}$.
Each directed graph with such a labelling
defines an equivalence class of diagrams for this  symmetric monoidal category. 

They are obtained by embedding the  coloured graph in $\mathbb R^3$ and choosing a generic projection $\pi:\mathbb R^3\to[0,1]\times[0,1]$, such that any singular point is a transversal crossing of two edges and the open edge ends end on $[0,1]\times\{1\}$ or $[0,1]\times\{0\}$. We call the former incoming and the latter outgoing edge ends.  As a consequence of the coherence theorem for spherical categories and the symmetric monoidality,  two diagrams obtained in this way evaluate to the same morphism, whenever their sets of incoming and of outgoing edges coincide and have the same linear ordering. 
Different choices of linear orderings for the edge ends lead to morphisms related by permutations of tensor factors in $\mathrm{Vect}_\C$, and different choices of incoming and outgoind edge ends are related by taking duals in $\mathrm{Vect}_\C$. The relevant information is the structure of a directed graph.

 For our defect cobordism $\Sigma_0\xrightarrow{i_0} M\xleftarrow{i_1} \Sigma_1$ 
  we choose the open ends corresponding to vertices in $\Sigma_0$ as the incoming ends in the diagram and the open ends for vertices of $\Sigma_1$ as the outgoing ends. Each vertex $u$ of the boundary $\partial M$ is incident to a unique 1-stratum $e_u$ of $M$. We put $x_u=s$ if  $e_u$ is outgoing at $u$ and
  $x_u=t$ if $e_u$ is incoming at $u$. For a vector space $V$ we put $V^u=V$, if  $u\in \Sigma_0$ with negative orientation or if $u\in \Sigma_1$ with positive orientation, and set $V^u=V^*$ otherwise. Note that negative orientation corresponds to an outgoing and positive orientation to an incoming stratum $e_u$ at $u$. 

The coloured graph $\Gamma^M_{col}$ then evaluates to the  natural transformation 
 $$\mu^M:=\langle  \Gamma^M_{col}\rangle : \bigotimes_{v\in S_0^{\Sigma_0}} \left (F_{e_v}^{x_v} \right)^v\Rightarrow \bigotimes_{u\in S_0^{\Sigma_1}}\left ( F_{e_u}^{x_u} \right )^{u}. $$
With formula \eqref{eq:fsigmadef} and
 Corollary  \ref{prop:projfunc} we  obtain 
\begin{align*}
    \bigotimes_{v\in S_0^{\Sigma_0}} \left (F_{e_v}^{x_v} \right)^v
    =\bigotimes_{v\in S_0^{\Sigma_0}}  F_v^{\Sigma_0}P_0=F_{\Sigma_0}P_0,\qquad\qquad
   \bigotimes_{u\in S_0^{\Sigma_1}}\left ( F_{e_u}^{x_u} \right )^{u}
    =\bigotimes_{u\in S_0^{\Sigma_1}}  F_u^{\Sigma_1}P_1=F_{\Sigma_1}P_1,
\end{align*}
which shows that  the evaluation of  $\Gamma^M_{col}$ is a natural transformation 
 \begin{align}\label{eq:natdef}
\mu^M=\langle \Gamma^M_{col}\rangle: F_{\Sigma_0}P_0\Rightarrow F_{\Sigma_1}P_1.
\end{align}

Although we selected a point  $\hat v_e \in \hat e$ for each isolated loop $e$ to construct the labels of $\Gamma^M_{col}$, its evaluation  does not depend on this choice. If $\hat v_e,\hat v'_e \in \hat e$ are two choices of vertices with associated loops $\delta: \hat v_e\to \hat v_e$ and $\delta': \hat v'_e\to \hat v'_e$,  then their  natural transformations from Lemma \ref{cor:indstrata-tweaked-nat} are related by 
$$H^{[\delta']}_e=H^{[\gamma]}_e\circ H^{[\delta]}_e \circ H^{[\gamma]^\inv}_e$$ for any path $\gamma:\hat v_e\to \hat v'_e$ in $\hat e$.   As the evaluation of an isolated loop is just the trace, it  is independent of  $\hat v_e$.

\begin{proposition} \label{prop:deffunc}There is a symmetric monoidal functor
$G:\mathrm{Cob}_3^{\mathrm{def}}\to \mathrm{span}^{\mathrm{fib}}(\repGrpd)$ that assigns
\begin{compactitem}
\item to a stratified surface the pair $(\mathcal A^\Sigma\sslash \mathcal G^\Sigma, F_\Sigma)$, where  $F_\Sigma: \mathcal A^\Sigma\sslash \mathcal G^\Sigma\to \Vect_\C$ is the functor from \eqref{eq:fsigmadef},
\item to a cobordism $\Sigma_0\xrightarrow{i_0} M\xleftarrow{i_1} \Sigma_1$ the pair $(\mathcal A^{\Sigma_0}\sslash \mathcal G^{\Sigma_0}\xleftarrow{P_0}\mathcal A^M\sslash \mathcal G^M\xrightarrow{P_1} \mathcal A^{\Sigma_1}\sslash \mathcal G^{\Sigma_1}, \mu^M)$ of the fibrant span assigned by Theorem \ref{th:classcobfunc} and the natural transformation  $\mu^M: F_{\Sigma_0}P_0\Rightarrow F_{\Sigma_1}P_1$ from \eqref{eq:natdef}.
\end{compactitem}
\end{proposition}

\begin{proof}  That the classical defect data defines a functor $G:\mathrm{Cob}_3^{\mathrm{def}}\to \mathrm{span}^{\mathrm{fib}}(\Grpd)$  was already shown in Theorem \ref{th:classcobfunc}. It remains to show that 
\begin{compactenum}
  \item the gluing of cobordisms induces the composition of natural transformations,
  \item  the cylinder cobordism defines the identity natural transformation. 
  \end{compactenum}

 1.~Let
  $\Sigma_0\xrightarrow{i_0} M_{01}\xleftarrow{i_1} \Sigma_1$, and
  $\Sigma_1\xrightarrow{i_1'} M_{12}\xleftarrow{i_2} \Sigma_2$  be  stratified cobordisms and
  $\Sigma_0 \xrightarrow{j_1} M \xleftarrow{j_2} \Sigma_2$ their composite with $M=M_{01}\amalg_{\Sigma_1} M_{12}$.  By  Theorem \ref{th:classcobfunc} this
  yields the following  diagram, where the top two levels commute and the top diamond is a pullback
  $$ \xymatrix{ &&& \A^{M} \sslash \G^{M} \ar[dl]_{P_{01}} \ar[dd]^{P}\ar[dr]^{P_{12}}\\
  && \A^{M_{01}}\sslash \G^{M_{01}} \ar[dl]_{P_0} \ddtwocell \omit {^<-2>\,\,\,  \mu^{M_{01}} }  \ar[dr]_{P_1}&& \A^{M_{12}}\sslash \G^{M_{12}}  \ar[dl]^{P_1'}    \ar[dr]^{P_2} \ddtwocell \omit {^<-2>\,\,\,  \mu^{M_{12}} }\\
  & \A^{\Sigma_0} \sslash \G^{\Sigma_0} \ar[drr]_{F_{\Sigma_0}}& &  \A^{\Sigma_1} \sslash \G^{\Sigma_1} \ar[d]_{F_{\Sigma_1}}& &  \A^{\Sigma_2} \sslash \G^{\Sigma_2} \ar[dll]^{F_{\Sigma_2}}
 \\ &&& \Vect_\C.  &&&& }$$
    We  show that the natural transformations for  the cobordisms from \eqref{eq:natdef} satisfy  $\mu^M=(\mu^{M_{12}}P_{12})\circ (\mu^{M_{01}}P_{01})$.
 
 Let $\Gamma^{M_{01}}_{col}$ and $\Gamma^{M_{12}}_{col}$ 
 be the coloured graphs associated to $M_{01}$ and $M_{12}$ and $\mu^{M_{01}}=\langle \Gamma^{M_{01}}_{col}\rangle$ and $\mu^{M_{12}}=\langle \Gamma^{M_{12}}_{col}\rangle$ their natural transformations from \eqref{eq:natdef}. Pre-composition with   $P_{01}$ and $P_{12}$ equips the underlying graphs  with a colouring by the groupoid $\mathcal A^M\sslash\mathcal G^M$ and defines coloured graphs  $ \Gamma^{M_{01}}_{col*}$ and $ \Gamma^{M_{12}}_{col*}$ with
 $$
 \mu^{M_{01}}P_{01}=\langle \Gamma^{M_{01}}_{col*}\rangle \qquad \textrm{ and } \qquad  \mu^{M_{12}}P_{12}=\langle \Gamma^{M_{12}}_{col*}\rangle. 
 $$
Denoting by $\Gamma^{M_{01}}_{col*}\#\Gamma^{M_{12}}_{col*}$ their concatenation along their edge ends in $\Sigma_1$, we then have
 $$
 (\mu^{M_{12}}P_{12})\circ (\mu^{M_{01}}P_{01})=\langle \Gamma^{M_{01}}_{col*}\rangle\circ \langle \Gamma^{M_{12}}_{col*}\rangle=\langle \Gamma^{M_{01}}_{col*}\#\Gamma^{M_{12}}_{col*}\rangle.
 $$
The coloured graph $\Gamma^M_{col}$ is obtained from $\Gamma^{M_{01}}_{col*}\#\Gamma^{M_{12}}_{col*}$ in two steps. {One first erases}  all bivalent vertices {on}  the edges of $M_{01}$ and $M_{12}$ that intersect $\Sigma_1$. One then  adds a bivalent vertex to all edges of the resulting new graph that now have no bivalent vertex. 
This bivalent vertex  is then labelled by (i) the natural transformation from \mbox{\eqref{eq:nonloopnat}}, if the resulting edge is  incident to at least one vertex or boundary vertex,  or  (ii) by the natural transformation from \eqref{eq:loopnat}, if it is an isolated loop. We then have
\begin{align}\label{eq:glueval}
\mu^M=\langle \Gamma^M_{col}\rangle=\langle \Gamma^{M_{01}}_{col*}\#\Gamma^{M_{12}}_{col*}\rangle.
\end{align}
This follows, because for each point $ x\in S^1$ the homotopy class  of a path based at $x$  that goes once around $S^1$ in the direction of its orientation is the composite of the homotopy classes of paths obtained by subdividing $S^1$ into a finite number of intervals. An analogous statement holds for the homotopy classes $[\gamma]:0\to 1$ of paths in $[0,1]$. Applying the identities in Lemma \ref{cor:indstrata-tweaked-nat} then gives \eqref{eq:glueval}.

2.~Let $M=\Sigma\times[0,1]$ be the cylinder cobordism with a single $(k+1)$-stratum $s\times[0,1]$ for each $k$-stratum $s$ of $\Sigma$. 
Then the coloured graph for $M$ consists of a single line with a bivalent vertex $v_e$ for each vertex $v$ of $\Sigma$.  For each vertex $v$ of $\Sigma$ and associated edge $e=v\times[0,1]$ of $\Sigma\times [0,1]$ the two ends of this line are labelled with functors $F_{e}^s=\rho_vP_{e}^s: \mathcal A^{M}\sslash \mathcal G^{M}\to \Vect_\C$ and $F_e^t=\rho_vP_{e}^t: \mathcal A^{M}\sslash \mathcal G^{M}\to \Vect_\C$. The vertex $v_e$ in the middle is labelled with the natural transformation  $\mu^{e}=\rho_vP_{e}^d: P_{e}^s\Rightarrow P_{e}^t$ from \eqref{eq:nonloopnat}.

We then have for each edge $e=v\times[0,1]$ of $M$ the groupoid  $\mathcal D_e=\mathcal D_v=M_v\sslash G_v$.  As  $e$ is incident at two boundary vertices, we can identify $\hat e\cong [0,1]$ and the inclusions  for the boundary vertices from \eqref{eq:stratboundemb} are given by  $\iota_0,\iota_1:\bullet \to [0,1]$. 

The functors
 $P_e^s, P_e^t:\mathcal A^M\sslash \mathcal G^M\to \mathcal D_v$ from Definition \ref{cor:indstrata-tweaked}  send a gauge configuration $A: T^M\Rightarrow D^M$ to $A_e(0), A_e(1)\in M_v$  and a gauge transformation $\gamma: A\Rrightarrow A'$ to the group elements $\gamma_e(0),\gamma_e(1)\in G_v$ with $\gamma_e(0)\rhd A_e(0)=A'_e(0)$ and  $\gamma_e(1)\rhd A_e(1)=A'_e(1)$, respectively. 
 
The natural transformation $P_e^d: P_e^s\Rightarrow P_e^t$ from Lemma \ref{cor:indstrata-tweaked-nat} has components $(P_e^d)_A=A_e([\delta]): A_e(0)\to A_e(1)$, where  $\delta:0\to 1$ is a path in $[0,1]$. The equivalence $\mathcal A^{M}\sslash \mathcal G^M\cong \GRPd(\Pi_1[0,1], \mathcal A^\Sigma\sslash \mathcal G^\Sigma)$ from the proof of Theorem \ref{th:classcobfunc} identifies 
\begin{compactitem}
\item the 
functors $P_e^s,P_e^t:\mathcal A^M\sslash \mathcal G^M\to \mathcal D_v$ with the functors $P_v\Pi_1(\iota_0)^*, P_v\Pi_1(\iota_1)^*: \mathcal A^\Sigma\sslash \mathcal G^\Sigma\to \mathcal D_v$ induced by the inclusions $\iota_0,\iota_1:\bullet\to [0,1]$,
\item  the 
natural transformation $P^d_e: P_e^s\Rightarrow P_e^t$ with  $P_v [\delta]^*: P_v\Pi_1(\iota_0)^*\Rightarrow P_v\Pi_1(\iota_1)^*$ induced by the homotopy class of any path $\delta:0\to 1$ in $[0,1]$.
\end{compactitem}
This gives the following commuting diagram
\begin{align*}
\xymatrix{
\mathcal A^M\sslash \mathcal G^M \ar[d]_\cong 
\ar@/^7pt/[rr]^{P_e^s} _{\Downarrow P_e^d}
\ar@/_7pt/[rr]_{P_e^t} 
&  &  {\mathcal D}_v
\\
\mathcal (A^\Sigma\sslash \mathcal G^\Sigma)^{\Pi_1[0,1]}  \ar[d]_{P_v^{\Pi_1[0,1]}}
\ar@/^7pt/[rr]^{\Pi_1(\iota_0)^*} _{\Downarrow [\delta]^*}
\ar@/_7pt/[rr]_{\Pi_1(\iota_1)^*} 
& & (\mathcal A^\Sigma\sslash \mathcal G^\Sigma)^\bullet \cong  \mathcal A^\Sigma\sslash\mathcal G^\Sigma \ar[u]_{P_v} \ar[d]^{P_v}\\
\mathcal D_v^{\Pi_1[0,1]} \ar@/^7pt/[rr]^{\Pi_1(\iota_0)^*} _{\Downarrow  [\delta]^*}
\ar@/_7pt/[rr]_{\Pi_1(\iota_1)^*}  & & \mathcal D_v^\bullet \cong  \mathcal D_v. 
}
\end{align*}
Composing with  $\rho_e=\rho_v: \mathcal D_v\to \Vect_\C$ then yields $\mu^e=\rho_vP_v [\delta]^*: \rho_vP_v\Pi_1(\iota_0)^*\Rightarrow \rho_vP_v\Pi_1(\iota_1)^*$.
Taking the tensor product over all edges $e$ of $M$ gives $\mu^M=\langle \Gamma^M_{col}\rangle= F_\Sigma [\delta]^*: F_\Sigma \Pi_1(\iota_0)^*\Rightarrow F_\Sigma\Pi_1(\iota_1)^*$, and combining this result with Lemma \ref{lem:pi1interval} proves the claim. 
\end{proof}

By composing the functor $G:\mathrm{Cob}_3^{\mathrm{def}}\to \mathrm{span}^{\mathrm{fib}}(\repGrpd)$ from Proposition \ref{prop:deffunc} with the  symmetric monoidal functor $L:\mathrm{span}^{\mathrm{fib}}(\repGrpd)\to \Vect_\C$ from Proposition \ref{prop:lfunctor} we obtain a defect TQFT.

\begin{theorem} \label{cor:limtot}The symmetric monoidal functor $\mathcal Z=LG:\mathrm{Cob}^{\mathrm{def}}_3\to \Vect_\C$ is a defect TQFT.
\end{theorem}

We will now show how ordinary Dijkgraaf-Witten TQFT without defects is recovered from this defect TQFT by considering  surfaces  and  cobordisms with only 
codim $0$-strata.

We first consider defect surfaces and defect cobordisms with only codim $\leq 1$ strata. Such defect surfaces and defect cobordisms form a symmetric monoidal subcategory $\mathrm{Cob}_3^{\mathrm{class}}\subset \mathrm{Cob}_3^{\mathrm{def}}$. For any defect surface $\Sigma$ in this subcategory  the associated functor $F_\Sigma: \mathcal A^\Sigma\sslash\G^\Sigma\to \Vect_\C$  from \eqref{eq:fsigmadef} is the constant functor with values in $\C$. For any defect cobordism $M$ in this subcategory the associated natural transformation $\mu^M:F_{\Sigma_0}P_0\Rightarrow F_{\Sigma_1}P_1$ from \eqref{eq:natdef} is the identity natural transformation on the constant functor. 

By Remark \ref{rem:subcatrem} such representations and intertwiners form the monoidal subcategory  $\mathrm{span}^{\mathrm{fib}}(\Grpd)\subset \mathrm{span}^{\mathrm{fib}}(\repGrpd)$. The restriction of the functor $L:\mathrm{span}^{\mathrm{fib}}(\repGrpd)\to \Vect_\C$ to this subcategory is the  symmetric monoidal functor $L_{\mathrm{class}}:\mathrm{span}^{\mathrm{fib}}(\Grpd)\to \Vect_\C$ from Corollary \ref{cor:quinnclass}. We thus obtain a commuting diagram of functors, where the vertical arrows are inclusions
$$
\xymatrix{
\mathrm{Cob}_3^{\mathrm{def}} \ar[r]^{G\qquad } & \mathrm{span}^{\mathrm{fib}}(\repGrpd) \ar[r]^{\qquad L} & \Vect_\C.\\
\mathrm{Cob}_3^{\mathrm{class}} \ar@{^{(}->}[u] \ar[r]_{G_{\mathrm{class}}\qquad} & \ar@{^{(}->}[u]\mathrm{span}^{\mathrm{fib}}(\Grpd) \ar[ru]_{L_{\mathrm{class}}}
}
$$

\begin{corollary}\label{cor:classdef} The restriction of the functor $\mathcal Z:\mathrm{Cob}_3^{\mathrm{def}}\to \Vect_\C$ from Theorem \ref{cor:limtot} to the monoidal subcategory $\mathrm{Cob}_3^{\mathrm{class}}\subset \mathrm{Cob}_3^{\mathrm{def}}$ assigns
\begin{compactitem}
\item to a  surface $\Sigma$ the vector space $\mathcal Z^H(\Sigma)=\langle \pi_0(\mathcal A^\Sigma\sslash\mathcal G^\Sigma)\rangle_\C$,
\item to the equivalence class of a cobordism $\Sigma_0\xrightarrow{i_0} M\xleftarrow{i_1} \Sigma_1$ the linear map $\mathcal Z^H(M):\mathcal Z^H(\Sigma_0)\to \mathcal Z^H(\Sigma_1)$ with matrix elements
$$
\big\langle [A_1] \mid \mathcal Z^H(M)\mid [A_0]\big\rangle=\chi^\pi\big(\PC_{A_1}(\mathcal A^{\Sigma_1}\sslash \G^{\Sigma_1})\big)\cdot \chi^\pi\big(\{ A_0\mid \mathcal A^M\sslash \mathcal G^M\mid A_1\}\big).
$$
\end{compactitem}
\end{corollary}

Here, the notation is as in Corollary \ref{cor:quinnclass}.
This defect TQFT with only codim $\leq 1$ strata was investigated by Fuchs, Schweigert and Valentino  in \cite{FSV} in the more general case of twisted Dijkgraaf-Witten TQFT, where  2-strata and 3-strata are equipped with cocycles.  The TQFT from Corollary \ref{cor:classdef} agrees with the results in \cite{FSV}, if the latter are restricted to trivial cocycles. 
In particular, the graphical calculus for group cocycles  in \cite[Section 3.4]{FSV} is also defined in the case where all group cocycles are trivial, and then reduces to our diagrammatic description from Section \ref{subsec:geomdescript}.

Ordinary Dijkgraaf-Witten TQFT without defects corresponds to defect surfaces and defect cobordisms with only codim $0$-strata. Such defect surfaces and defect cobordisms have as defect data a single finite group $H$ for each codim 0-stratum.   This defines for each finite group $H$ a symmetric 
 monoidal subcategory $\mathrm{Cob}^H_3\subset \mathrm{Cob}_3^{\mathrm{class}}$, in which all 0-strata are labelled with $H$.  The restriction of the defect TQFT from 
 Corollary \ref{cor:classdef} to this subcategory is given by Quinn's finite total homotopy TQFT \cite{Q,FMP} for the classifying space of a finite group. This is precisely Dijkgraaf-Witten TQFT with a trivial cocycle.
The latter is given by the formulas in  Corollary \ref{cor:classdef} and the following corollary. This follows from \cite[\S 8.2]{FMP} and \cite{Mo}.

\begin{corollary} For any finite group $H$, the restriction of $G_{\mathrm{class}}: \mathrm{Cob}_3^{\mathrm{class}}\to \mathrm{span}^{\mathrm{fib}}(\Grpd)$ to the monoidal subcategory $\mathrm{Cob}_3^H\subset \mathrm{Cob}_3^{\mathrm{class}}$ assigns
\begin{compactitem}
\item to a surface $\Sigma$ the essentially finite groupoid $\mathcal A^\Sigma\sslash G^\Sigma=(\bullet\sslash H)^{\Pi_1(\Sigma)}$,
\item to the equivalence class of a cobordism $\Sigma_0\xrightarrow{i_0} M\xleftarrow{i_1} \Sigma_1$  the fibrant span of essentially finite groupoids 
$$(\bullet\sslash H)^{\Pi_1(\Sigma_0)}\xleftarrow{P_0=\Pi_1(i_0)^*} (\bullet\sslash H)^{\Pi_1(M)}\xrightarrow{P_1=\Pi_1(i_1)^*} (\bullet\sslash H)^{\Pi_1(\Sigma_1)}.$$
\end{compactitem}
The resulting  TQFT
$\mathcal Z^H:\mathrm{Cob}^H_3\to\Vect_\C$  coincides with Dijkgraaf-Witten TQFT for the finite group $H$.
\end{corollary}

Note that in this case a coherent basepoint set is obtained by choosing a single basepoint on each connected component of a surface $\Sigma$ and a single basepoint on each connected component of the boundary $\partial M$ of a cobordism $M$. For a connected surface $\Sigma$ this  yields the reduced gauge groupoid $\mathcal A^\Sigma\sslash\G^\Sigma=\Hom(\pi_1(\Sigma), H)\sslash H$, where $H$ acts by conjugation.

\section{Examples}
\label{sec:examples}

In this section, we illustrate our formalism and the resulting defect TQFTs with examples. We first give some  examples of  defect data in Section \ref{sec:defexamples}. 
In Section \ref{subse:qudoublemods} we then compute examples of the two-dimensional part of a defect TQFT, the vector spaces assigned to  surfaces.
Section \ref{subsec:defcobs}  treats examples of  defect cobordisms.

The results in Section \ref{subse:qudoublemods} are  of independent interest in topological quantum computing.  The vector space a Dijkgraaf-Witten TQFT assigns to a surface 
is the ground state of  Kitaev's quantum double model \cite{K} and also describes the closely related Levin-Wen model \cite{LW}. Our defect TQFT gives a simple description of quantum double models with defects that can be viewed as a special case of the models introduced by Kitaev and Kong \cite{KK}. In our formalism, the  ground state of the quantum double model with defects can be computed efficiently and without choosing a triangulation.

\subsection{Defect data}
\label{sec:defexamples}

To treat examples of defect TQFTs we require some preliminaries about defect data.  First, we introduce  a notion of transparent defect  that corresponds to the absence of defects.  

\begin{definition}\label{def:transpdef} Let $X$ be a stratified $n$-manifold  with  defect data.
\begin{compactenum}
\item A \textbf{transparent}  codim 1-stratum is  labelled with a group $G$ as a $G\times G^{op}$-set.
\item A \textbf{transparent} codim 2-stratum on a codim 1-stratum labelled with a $G\times H^{op}$-set $M$ is  labelled with 
\begin{align}\label{eq:transprep}
\rho_M: M\times M\sslash  G\times H\to \Vect_\C\qquad \rho_M(m,m')=\begin{cases} \C & m=m'\\ 0 & m\neq m'\end{cases}\qquad \rho_M(g,h)=\id.
\end{align}
\item A \textbf{transparent}  codim 3-stratum on a codim 2-stratum   with  $\rho$ is  labelled
with  $\id_{\rho}:  \rho\Rightarrow \rho$.
\end{compactenum}
\end{definition}

We will also need non-transparent representations of action groupoids and intertwiners between them. 
For this, we relate them to modules  over  smash products of group algebras and their intertwiners. 
For a finite set $M$ we denote by $\C^M$ the algebra of maps $f: M\to\C$ with the pointwise multiplication and with the basis given by the maps $\delta_m: M\to \C$ with $\delta_m(n)=0$ for $n\neq m$ and $\delta_m(m)=1$. 
If $\rhd: G\times M\to M$ is an action of a finite group $G$, then $\C^M$ becomes a module algebra over the Hopf algebra $\C[G]$ with the induced $\C[G]$-action given by $g\rhd \delta_m=\delta_{g\rhd m}$ on the basis. The associated smash product $\C[G]\ltimes \C^M$ is the associative unital  algebra   $\C[G]\oo \C^M$ with the multiplication
\begin{align}\label{eq:smash}
(g\oo\delta_m)\cdot (h\oo \delta_n)=\delta_m(h\rhd n)\, gh\oo \delta_n.
\end{align}
We denote by  $\C[G]\ltimes\C^M\mathrm{-Mod}$ the category of its finite-dimensional representations  and  intertwiners.

\begin{lemma} \label{lem:smashrep}For any action $\rhd: G\times M\to M$  of a finite group $G$ on a finite set $M$ the categories
$\Vect_\C^{M\sslash G}$ and $\C[G]\ltimes\C^M\mathrm{-Mod}$ are equivalent.
\end{lemma}

\begin{proof} The equivalence is given by   the functor $F:\Vect_\C^{M\sslash G}\to \C[G]\ltimes\C^M\mathrm{-Mod}$ that sends
\begin{compactitem}
\item a functor $\rho: M\sslash G\to \Vect_\C$ to the vector space $V_\rho=\oplus_{m\in M} \rho(m)$ with the $\C[G]\ltimes\C^M$-module structure
\begin{align*}
(g\oo \delta_m)\rhd v= \iota_{g\rhd m}\circ \rho(g)\circ \pi_m(v),
\end{align*}
\item  a natural transformation $\mu: \rho\Rightarrow\rho'$ to the intertwiner
\begin{align*}
\phi_\mu=\Sigma_{m\in M}\, \iota'_{m}\circ \mu_m\circ \pi_m: \oplus_{m\in M} \rho(m)\to \oplus_{m\in M}\rho'(m),
\end{align*}
\end{compactitem}
and the functor $H: \C[G]\ltimes\C^M\mathrm{-Mod}\to \Vect_\C^{M\sslash G}$ that sends
\begin{compactitem}
\item  a $\C[G]\ltimes\C^M$-module $(V,\rhd)$ to the functor $\rho: M\sslash G\to\Vect_\C$ given on the objects by 
$\rho(m)=(1\oo \delta_m)\rhd V$ and on a morphism $g: m\to g\rhd m$  by $\rho(g)=(g\oo \delta_m)\rhd-:\rho(m)\to \rho(g\rhd m)$,
\item  an intertwiner $\phi: V\to V'$ to the natural transformation $\mu:\rho\Rightarrow\rho'$ whose components are the (co)restrictions  of $\phi$ to the subspaces $V_m=(1\oo\delta_m)\rhd V$ and $V'_m=(1\oo\delta_m)\rhd' V'$.
\end{compactitem}
\end{proof}

To construct representations of an action groupoid $M\sslash G$ it is then sufficient to construct representations of the smash product $\C[G]\ltimes\C^M$. The construction of irreducible representations of such smash products   is   analogous to the construction of irreducible representations of the Drinfeld double $D(G)$  by Gould \cite{G}. 

Consider a $G$-orbit $\mathcal O$ on $M$ and a representative $m_0\in \mathcal O$ with stabiliser $N=\mathrm{Stab}(m_0)$. Fix for each element $m\in \mathcal O$ an element $h_m\in G$ with $h_m\rhd m_0=m$ such that $h_{m_0}=1$. For each representation $\sigma: N\to \mathrm{Aut}_\C(V)$ consider the induced representation of $G$ on $\C[G]\oo_{\C[N]} V$. For any basis $B$ of $V$ the elements  $h_m\oo v$ for  $m\in \mathcal O$ and $v\in B$ form  a  basis  of $\C[G]\oo_{\C[N]}V$, and its 
 $\C[G]\ltimes \C^M$-module structure  is  given by
\begin{align}\label{eq:modsmash}
(g\oo \delta_n)\rhd (h_m\oo v)=\delta_n(m)\;h_{g\rhd m}\oo \sigma(h_{g\rhd m}^\inv g h_m)v. 
\end{align}
It is easy to see that this module is irreducible if and only if $\sigma$ is irreducible.  Thus, $G$-orbits on $M$ and irreducible representations of their stabilisers define  irreducible representations of $\C[G]\ltimes\C^M$. 
By combining this result with  Lemma \ref{lem:smashrep} one obtains functors $\rho: M\sslash G\to \Vect_\C$.

\begin{corollary}\label{cor:actact} Let $\rhd: G\times M\to M$ be an action of a finite group $G$ on a finite set $M$. Let $\mathcal O$ a $G$-orbit with a representative $m_0\in \mathcal O$ and $h_m\in G$    with $h_m\rhd m_0=m$ for each $m\in \mathcal O$ such that $h_{m_0}=1$. 
Then  any  representation $\sigma:\mathrm{Stab}(m_0)\to \mathrm{Aut}_\C(V)$ defines a functor $\rho_\sigma: {M\sslash G}\to \Vect_\C$ with
\begin{compactitem}
\item $\rho_\sigma(m)=V$ for each $m\in \mathcal O$ and $\rho_\sigma(m)=0$ for $m\notin\mathcal O$,
\item $\rho_\sigma(g)(v)=\sigma(h^\inv_{g\rhd m}gh_m)v$ for a morphism $g: m\to g\rhd m$, $m\in \mathcal O$, and $\rho(g)=\id_0$ if $m\notin \mathcal O$.
\end{compactitem}
\end{corollary}

\begin{example} \label{ex:actorbb}$\quad$ 
\begin{compactenum}
\item  Let $G\sslash G$ the action groupoid for the conjugation action of a finite group $G$ on itself. Then the smash product \eqref{eq:smash} is the (untwisted) Drinfeld double $D(G)$ and the irreducible representations in \eqref{eq:modsmash} coincide with the irreducible representations of $D(G)$ constructed in \cite{G}, {see  \cite{W}.}

\item  Consider a $G$-set $M$ and the $G$-set $M\times M$ with the diagonal action. The $G$-orbits $\mathcal O_m=G\rhd(m,m)$ for $m\in M$ with the trivial stabiliser representations   yield the transparent representation from \eqref{eq:transprep}.
\end{compactenum}
\end{example}

\subsection{Defect surfaces and generalised quantum double models}
\label{subse:qudoublemods}

In this section, we show that the vector space a defect TQFT assigns to a closed surface $\Sigma$ can be viewed as a generalisation of the ground state of Kitaev's quantum double model for a finite group $G$.  This model was introduced in \cite{K} as a realistic model for a topological quantum computer.

 It takes  as algebraic data a 
finite group $G$.  Its topological data is a directed  graph $\Gamma \subset \Sigma$ embedded into a closed oriented  surface  $\Sigma$ such that $\Sigma\setminus \Gamma$ is a disjoint union of discs. For simplicity, we take $\Sigma$ to be connected.
 In our setting, $\Gamma$ is the dual graph of a fine stratification of  $\Sigma$. We denote by 
$V$, $E$ and $P$ the sets of vertices, edges and plaquettes of $\Gamma$.
We say that a plaquette is $n$-valent if its dual vertex has $n$ incident edge ends.

The quantum double model associates to this data the complex vector space $\mathcal H=\otimes_{e\in E} \C[G]=\C[G^{\times E}]$, together with two families of linear endomorphisms of 
$\mathcal H$. The first family is associated with vertices $v$ of $\Sigma$ and indexed by group elements $g\in G$. It describes the action of graph gauge transformations at $v$
$$
A^g_v:\mathcal H\to\mathcal H,\quad h_e\mapsto g_{t(e)}\cdot h_e\cdot g_{s(e)}^\inv, $$
where $h_e$ denotes the group element assigned to an edge $e\in E$,  $g_v=g$ and $g_w=1$ for $w\neq v$.
The second family of endomorphisms is associated with  plaquettes $p\in P$ together with a choice of a {vertex $v_p$}  in $p$  and also indexed by  elements $g\in G$. It  assigns to an element of $G^{\times E}$ the value 1 if the  ordered oriented product $h_p$ of group elements at the edges of $p$ starting at $v_p$ satisfies $h_p=g$ and $0$ otherwise
$$
B^g_p:\mathcal H\to\mathcal H, \quad x\mapsto \delta_g(h_p)\, x.
$$
For any site {$s=(v_p,p)$, a pair of a plaquette $p$ and a chosen vertex $v_p$ at $p$,}  the maps $A_v^g$ and $B^h_p$ for $g,h\in G$ form a representation of the Drinfeld double $D(G)$ on $\mathcal H$. In terms of the parametrisation \eqref{eq:smash} it is given by
$$
(g\oo \delta_h)\rhd x=A_v^g B_p^h x.
$$
The ground state of the model corresponds to the trivial representation of $D(G)$ at each site of $\Gamma$. The  maps
$$
A_v=|G|^\inv \Sigma_{g\in G} A_v^g:\mathcal H\to\mathcal H\qquad B_p=B_p^1:\mathcal H\to\mathcal H
$$
for $v\in V$ and $p\in P$ are commuting projectors, and the  ground state of the quantum double model on $\Sigma$ is 
\begin{align}\label{eq:grdstate}
\mathcal H_{inv}=\left(\cap_{v\in V} \mathrm{im}(A_v)\right)\cap \left(\cap_{p\in P} \mathrm{im}(B_p)\right)\cong \langle \mathrm{Hom}(\pi_1(\Sigma),G)/G\rangle_\C.
\end{align}
It was shown by Balsam and Kirillov \cite{BK} in the more general setting of a finite-dimensional semisimple Hopf algebra $H$  that the ground state coincides with the vector space the  Turaev-Viro-Barrett-Westbury TQFT for $H\mathrm{-Mod}$ assigns to $\Sigma$. In our setting this is  the free vector space generated by conjugacy classes of group homomorphisms $\rho:\pi_1(\Sigma)\to G$, assigned to $\Sigma$ by the untwisted  Dijkgraaf-Witten TQFT for $G$.  

It is also shown in \cite{BK} that excitations in this model are given by irreducible representations of the Drinfeld double $D(H)$. In our situation this is the Drinfeld double $D(G)$ given by 
 \eqref{eq:smash} for the conjugation action.
Excitations are associated with sites  $s=(v,p)$ and form  irreducible representations of  $D(G)$  as in Corollary \ref{cor:actact} and Example \ref{ex:actorbb}, 1. They are  given by a conjugacy class $\mathcal C\subset G$ and a representation  of a stabiliser. 

 We will now show how the groupoid $\mathcal A\sslash \mathcal G$ of gauge configurations and gauge transformations on a stratified surface $\Sigma$, the associated functor $F_\Sigma: \mathcal A\sslash \mathcal G\to \Vect_\C$ from \eqref{eq:fsigmadef} and the  vector space 
$\mathcal Z(\Sigma)=\lim F_\Sigma$ generalise the quantum double model. More specifically, they define  a  quantum double model with excitations,  domain walls and  domain wall junctions that can be viewed as a concrete and accessible manifestation of the more general models considered by  Kitaev and Kong \cite{KK}.  
 
 We consider a fine stratification of $\Sigma$ whose dual is the directed graph $\Gamma$. Then the framing of the stratification assigns to each plaquette $p$ of $\Gamma$ a vertex $v_p$. 
By Definition \ref{def:defectdata} a  labelling with defect data assigns
\begin{compactitem}
\item a finite group $G_v$ to each vertex $v$ of $\Gamma$,
\item a finite $G_{t(e)}\times G_{s(e)}^{op}$-set $M_e$ to each edge $e$ of $\Gamma$,
\item a functor $\rho_p: \Pi_{e\in p} M_e\sslash  \Pi_{v\in p} G_v\to \Vect_\C$ to each plaquette $p$ of $\Gamma$. 
\end{compactitem}
Here,  the products are taken over all  edges and vertices in $p$, taking into account their multiplicities. A set $M_e$ or a group $G_v$ occurs as often as the path along the boundary of $p$ traverses $e$ and $v$.

By Corollary \ref{cor:redggrpd} the  groupoid of gauge configurations and gauge transformations  is equivalent to  the groupoid of  graph gauge configurations and graph gauge transformations on $\Gamma$
$$
\mathcal A\sslash \mathcal G=\left(\Pi_{e\in E} M_e\right)\sslash  \left(\Pi_{v\in V} G_v\right).
$$
\begin{compactitem}
\item A  graph gauge configuration  assigns  an element $m_e\in M_e$ to each edge $e$ of $\Gamma$.
\item A graph gauge transformation assigns an element $g_v\in G_v$ to each vertex $v$ of $\Gamma$ and acts  by
$$
m_e\mapsto g_{t(e)}\rhd m_e\lhd g_{s(e)}^\inv.
$$
\end{compactitem}
The  functor $F_\Sigma: \mathcal A\sslash \mathcal G\to \Vect_\C$ from \eqref{eq:fsigmadef} sends  a gauge configuration to the tensor product of all vector spaces  associated to the labelled plaquettes of $\Gamma$ and  a gauge transformation to a  linear map between them. The defect TQFT assigns to $\Sigma$ the vector space 
$$\mathcal Z(\Sigma)=\lim F_\Sigma.$$ 
To make contact with quantum double models, we consider the following specific cases of plaquettes, their  action groupoids $\mathcal A\sslash \mathcal G$ and the associated representations $\rho: \mathcal A\sslash \mathcal G\to \Vect_\C$:
\begin{compactitem}

\item  \textbf{flat bulk plaquette:} a plaquette with edges  labelled by a group $G$, where $\rho$ assigns the vector space $\C$ if the oriented ordered product of its group elements  is trivial  and else the null vector space. For instance, 
the flat plaquette in  \eqref{eq:flatpic} has the action groupoid $G^{\times 5}\sslash G^{\times 5}$ with
$$(v_1,\ldots, v_5)\rhd (g_1,\ldots g_5)=(v_1g_1v_2^\inv, v_2g_2v_3^\inv, v_4 g_3 v_3^\inv, v_4 g_4 v_5^\inv, v_1 g_5 v_5^\inv)$$
and  $\rho: G^{\times 5}\sslash G^{\times 5}\to \Vect_\C$ with $\rho(g_1,\ldots, g_5)=\C$ for $g_1g_2g_3^\inv g_4 g_5^\inv=1$ and $\rho(g_1,\ldots, g_5)=0$ else. 
\begin{align}\label{eq:flatpic}
\begin{tikzpicture}[scale=.6, baseline=(current  bounding  box.center)]
\draw[color=black,line width=1.5pt] (1,0)--(-1,0) node[sloped, pos=0.5, allow upside down]{\arrowOut}  node[pos=.5, anchor=north]{$g_5$};
\draw[color=black,line width=1.5pt] (1,0)--(2,1) node[sloped, pos=0.5, allow upside down]{\arrowOut}  node[pos=.5, anchor=north west]{$g_4$};
\draw[color=black,line width=1.5pt] (0,2)--(2,1) node[sloped, pos=0.5, allow upside down]{\arrowOut}  node[pos=.5, anchor=south west]{$g_3$};
\draw[color=black,line width=1.5pt] (0,2)--(-2,1) node[sloped, pos=0.5, allow upside down]{\arrowOut}  node[pos=.5, anchor=south east]{$g_2$};
\draw[color=black,line width=1.5pt] (-2,1)--(-1,0) node[sloped, pos=0.5, allow upside down]{\arrowOut}  node[pos=.5, anchor=north east]{$g_1$};
\draw[fill=black] (-1,0) circle (.1) node[anchor=north east]{$v_1$};
\draw[fill=black] (1,0) circle (.1) node[anchor=north west]{$v_5$};
\draw[fill=black] (2,1) circle (.1) node[anchor=west]{$v_4$};
\draw[fill=black] (-2,1) circle (.1) node[anchor=east]{$v_2$};
\draw[fill=black] (0,2) circle (.1) node[anchor=south]{$v_3$};
\end{tikzpicture}
\end{align}
\item  \textbf{plaquette with excitation:} a plaquette with edges  labelled by a group $G$ and  a representation $\rho$ as in Corollary \ref{cor:actact}, given by a  conjugacy class $\mathcal C$ in $G$  and a representation $\sigma$ of its stabiliser. The  representation is applied to the oriented ordered product $g_p$ of the group elements along $p$.
For instance, the plaquette in \eqref{eq:expic} has the action groupoid $G\sslash G$ for the conjugation action and $\rho: G\sslash G\to \Vect_\C$ given by  $\mathcal C\subset G$ and  $\sigma: \mathrm{Stab}(g_0)\to \mathrm{Aut}_\C(\C)$ for some $g_0\in \mathcal C$. 
\begin{align}\label{eq:expic}
\quad\nonumber\\[-15ex]
\begin{tikzpicture}[scale=.6]
\draw[line width=1.5pt] (0,0).. controls (5,-4) and (5,4) ..(0,0) node[sloped, pos=0.5, allow upside down]{\arrowOut}  node[pos=.5, anchor=west]{$\;g$};
\draw[fill=black] (0,0) circle (.1);
\draw[fill=red, color=red] (1.7,0) circle (.1) node[anchor=west]{$(\mathcal C,\sigma)$};
\end{tikzpicture}\nonumber\\[-15ex]
\end{align}

\item  \textbf{flat domain wall plaquette:} a bivalent plaquette with edges labelled by a $G\times H^{op}$-set $M$ 
and  with the transparent representation from Definition \ref{def:transpdef}. 
 The plaquette in \eqref{eq:domwall} has the groupoid $M\times M\sslash G\times H$ with $(g,h)\rhd (m,m')=(g\rhd m\lhd h^\inv, g\rhd m'\lhd h^\inv)$ and   $\rho: M\times M\sslash G\times H\to \Vect_\C$ with $\rho(m,m')=\C$ for $m=m'$ and $\rho(m,m')=0$ else. 
\begin{align}\label{eq:domwall}
\begin{tikzpicture}[scale=.6]
\draw[line width=1.5pt, color=blue] (0,0).. controls (1,1) and (2,1) ..(3,0) node[sloped, pos=0.7, allow upside down]{\arrowOut}  node[pos=.7, anchor=south west]{$m$};
\draw[line width=1.5pt, color=blue] (0,0).. controls (1,-1) and (2,-1) ..(3,0) node[sloped, pos=0.7, allow upside down]{\arrowOut}  node[pos=.7, anchor=north west]{$m'$};
\draw[fill=gray, color=gray] (0,0) circle (.1);
\draw[fill=black] (3,0) circle (.1);
\node at (-.1,0) [color=gray, anchor=east]{$h$};
\node at (3.1,0) [color=black, anchor=west]{$g$};
\draw[line width=1pt, dashed, color=blue] (1.5,-1.5)--(1.5,1.5) node[anchor=south]{$M$};
\node at (3,1.5)[anchor=south]{$G$};
\node at (0,1.5)[anchor=south, color=gray]{$H$};
\end{tikzpicture}
\end{align}

\item  \textbf{domain wall plaquette with domain wall excitation:} bivalent plaquette with edges  labelled by  a $G\times H^{op}$-set $M$ and  with a functor $\rho: M\times M\sslash G\times H\to \Vect_\C$ as in Corollary \ref{cor:actact}. The plaquette in \eqref{eq:domwallexc}  has  the action groupoid $M\times M\sslash G\times H$ for  $(g,h)\rhd (m,m')=(g\rhd m\lhd h^\inv, g\rhd m'\lhd h^\inv)$ and the functor $\rho: M\times M\sslash G\times H\to \Vect_\C$ given by a $G\times H$-orbit $\mathcal O\subset M\times M$ and a representation $\sigma: \mathrm{Stab}(m_0,m'_0)\to \mathrm{Aut}_\C(V)$ for an element $(m_0,m'_0)\in \mathcal O$. 
\begin{align}\label{eq:domwallexc}
\begin{tikzpicture}[scale=.6]
\draw[line width=1.5pt, color=blue] (0,0).. controls (1,1.5) and (3,1.5) ..(4,0) node[sloped, pos=0.7, allow upside down]{\arrowOut}  node[pos=.7, anchor=south west]{$m$};
\draw[line width=1.5pt, color=blue] (0,0).. controls (1,-1.5) and (3,-1.5) ..(4,0) node[sloped, pos=0.7, allow upside down]{\arrowOut}  node[pos=.7, anchor=north west]{$m'$};
\draw[fill=gray, color=gray] (0,0) circle (.1);
\draw[fill=black] (4,0) circle (.1);
\node at (-.1,0) [color=gray, anchor=east]{$h$};
\node at (4.1,0) [color=black, anchor=west]{$g$};
\draw[line width=1pt, dashed, color=blue] (2,-2)--(2,2) node[anchor=south]{$M$};
\node at (4,2)[anchor=south]{$G$};
\node at (0,2)[anchor=south, color=gray]{$H$};
\draw[fill=blue, color=blue] (2,0) circle (.12) node[anchor=west]{$(\mathcal O,\sigma)$};
\end{tikzpicture}
\end{align}

\item  \textbf{domain wall plaquette with bulk excitation:} they describe the fusing of bulk excitations into domain walls, which were investigated in \cite{KK} and \cite{FSV}. 
The plaquette in \eqref{eq:bulkexc} has  the action 
groupoid $G\times M\times M\sslash  G\times H$ with the group 
action $(g,h)\rhd (c,m,m')=(gcg^\inv, g\rhd m\lhd h^\inv, g\rhd m'\lhd h^\inv)$ and the functor $\rho: G\times M\times M\sslash  G\times H\to \Vect_\C$ given by a conjugacy class $\mathcal C\subset G$ and a representation $\sigma: \mathrm{Stab}(g_0)\to \mathrm{Aut}_\C(V)$ for $g_0\in \mathcal C$: 
\begin{align*}
\rho(c,m,m')&=\begin{cases} 
V & m'=c\rhd m, c\in \mathcal C\\
0 & \text{else}
\end{cases}\\
\rho(g,h)&=\begin{cases} \sigma(h_{gcg^\inv}^\inv g h_c) &  m'=c\rhd m, c\in \mathcal C\\
\id_0 & \text{else}
 \end{cases}\quad\text{for}\;(g,h): (c,m,m')\to (g,h)\rhd (c,m,m').
\end{align*}
Here,  $h_g\in G$ are fixed group elements with $h_g g_0 h_g^\inv=g$ for $g\in \mathcal C$ and $h_{g_0}=1$, as in Corollary \ref{cor:actact}.
\begin{align}\label{eq:bulkexc}
\begin{tikzpicture}[scale=.7]
\draw[line width=1.5pt, color=blue] (0,0).. controls (1,1.8) and (4,1.8) ..(5,0) node[sloped, pos=0.4, allow upside down]{\arrowOut}  node[pos=.4, anchor=south]{$m$};
\draw[line width=1.5pt, color=blue] (0,0).. controls (1,-1.8) and (4,-1.8) ..(5,0) node[sloped, pos=0.4, allow upside down]{\arrowOut}  node[pos=.4, anchor=north]{$m'$};
\draw[line width=1.5pt] (5,0).. controls (2,2) and (2,-2)..(5,0) node[sloped, pos=0.3, allow upside down]{\arrowOut}  node[pos=.3, anchor=south]{$c$};
\draw[fill=gray, color=gray] (0,0) circle (.1);
\draw[fill=black] (5,0) circle (.1);
\draw[fill=red, color=red] (3,0) circle (.1) node[anchor=west]{$(\mathcal C,\sigma)$};
\node at (-.1,0) [color=gray, anchor=east]{$h$};
\node at (5.1,0) [color=black, anchor=west]{$g$};
\draw[line width=1pt, dashed, color=blue] (2.5,-2)--(2.5,2) node[anchor=south]{$M$};
\node at (4,2)[anchor=south]{$G$};
\node at (1,2)[anchor=south, color=gray]{$H$};
\end{tikzpicture}
\end{align}

\item  \textbf{junction of domain walls:} a plaquette whose edges are labelled with elements of distinct sets with group actions and equipped with a representation of the associated action groupoid.  The action groupoid for the plaquette \eqref{eq:junc} is $(M\times N\times P\times Q)\sslash  (G\times H\times K\times L)$ with the group action
$$
(g,h,k,l)\rhd (m,n,p,q)=(g\rhd m\lhd l^\inv, h\rhd n\lhd g^\inv, h\rhd p\lhd k^\inv, l\rhd q\lhd k^\inv)
$$
and $\rho$ is a representation $\rho:(M\times N\times P\times Q)\sslash  (G\times H\times K\times L)\to \Vect_\C$. 
\begin{align}\label{eq:junc}
\begin{tikzpicture}[scale=.7]
\draw[color=blue, line width=1.5pt] (0,0)--(3,0) node[sloped, pos=0.3, allow upside down]{\arrowOut}  node[pos=.3, anchor=north]{$m$};
\draw[color=violet, line width=1.5pt] (3,0)--(3,3) node[sloped, pos=0.3, allow upside down]{\arrowOut}  node[pos=.3, anchor=west]{$n$};
\draw[color=cyan, line width=1.5pt] (0,3)--(0,0) node[sloped, pos=0.7, allow upside down]{\arrowOut}  node[pos=.7, anchor=east]{$q$};
\draw[color=red, line width=1.5pt] (0,3)--(3,3) node[sloped, pos=0.3, allow upside down]{\arrowOut}  node[pos=.3, anchor=south]{$p$};
\draw[color=blue, line width=1pt, dashed] (1.5,1.5)--(1.5,-.5) node[anchor=north]{$M$};
\draw[color=red, line width=1pt, dashed] (1.5,1.5)--(1.5,3.5) node[anchor=south]{$P$};
\draw[color=violet, line width=1pt, dashed] (1.5,1.5)--(3.5,1.5) node[anchor=west]{$N$};
\draw[color=cyan, line width=1pt, dashed] (1.5,1.5)--(-.5,1.5) node[anchor=east]{$Q$};
\draw[color=black, fill=black](1.5,1.5) circle (.1) node[anchor=south east]{$\rho$};
\draw[color=black, fill=black](3,3) circle (.1) node[anchor=south west]{$h$};
\node at (4.5,3.5) [anchor=south west]{$H$};
\draw[color=black, fill=black](3,0) circle (.1) node[anchor=north west]{$g$};
\node at (4.5,-.5) [anchor=north west]{$G$};
\draw[color=black, fill=black](0,3) circle (.1) node[anchor=south east]{$k$};
\node at (-1.5,3.5) [anchor=south east]{$K$};
\draw[color=black, fill=black](0,0) circle (.1) node[anchor=north east]{$l$};
\node at (-1.5,-.5) [anchor=north east]{$L$};
\end{tikzpicture}
\end{align}

\end{compactitem}

We now compute the vector space $\mathcal Z(\Sigma)$ of the defect TQFT for different stratifications of surfaces. The first example shows that for a surface with transparent defects the result is the ground state of the quantum double model or, equivalently, the vector space assigned by a Dijkgraaf-Witten TQFT.

\bigskip
\begin{example}\label{ex:groundstate}  \textbf{(Ground state of Kitaev's quantum double model)}

 \textbf{Stratification and defect data:}\\
Let $\Sigma$ be a surface with a fine stratification and dual graph $\Gamma$, such that 
\begin{compactitem}
\item every vertex $v$ of $\Gamma$ is labelled with the same finite group $G$,
\item every edge $e$ of $\Gamma$ is labelled with $G$ as an $G\times G^{op}$-set,
\item every $n$-valent plaquette $p$ of $\Gamma$ is labelled with  $\rho_p: G^{\times n}\sslash  G^{\times n}\to \Vect_\C$
that assigns the vector space $\C$ if the ordered oriented product of the group elements at edges of $p$ is trivial and else $0$.
\end{compactitem} 
 \textbf{Gauge configurations and gauge transformations:}\\
The groupoid of gauge configurations and of gauge transformations is
$$
\mathcal A\sslash \mathcal G=G^{\times E}\sslash G^{\times V}.
$$
\begin{compactitem}
\item A gauge configuration is an assignment of an element $g_e\in G$ to each edge $e$ of $\Gamma$.
\item A  gauge transformation is an assignment of an element $g_v\in G$ to each vertex $v$ of $\Gamma$ and acts by
$$
g_e\mapsto g_{t(e)}\cdot g_e\cdot g_{s(e)}^\inv.
$$ 
\end{compactitem}
 \textbf{Vector space for the defect surface:}\\
The functor $F_\Sigma:\mathcal A\sslash \mathcal G\to\Vect_\C$ assigns to a gauge configuration the  null vector space  unless
the ordered oriented product of the group elements vanishes 
at all plaquettes and in that case  the vector space $\C$. 
It follows that  $\mathcal Z(\Sigma)=\lim F_\Sigma$ is given by
 $$
\mathcal Z(\Sigma)= \langle \pi_0(\mathcal A\sslash \mathcal G)\rangle_\C= \langle \mathrm{Hom}(\pi_1(\Sigma), G)/G\rangle_\C.
$$
\end{example}

We now consider an example where  all edges of $\Gamma$ are labelled again by a single group $G$, but where some plaquettes are equipped with non-trivial representations of the associated action groupoids. This corresponds to quantum double models with excitations that are treated in \cite{BK} and \cite{KK}.

\bigskip
\begin{example}  \textbf{(Excitations in  quantum double models)}

 \textbf{Stratification:}\\
Let $\Sigma$ be an oriented surface of genus $g$ with  marked points $x_1,\ldots, x_n$ and $\Sigma^*=\Sigma\setminus\{x_1,\ldots, x_n\}$. Choose a directed graph $\Gamma$ consisting of loops $m_i$ 
that go counterclockwise around $x_i$  
and   $a$- and $b$-cycles  $a_1,b_1,\ldots, a_g,b_g$, all based at  $x\in \Sigma^*$, such that  
$$
\pi_1(\Sigma^*,x)=\langle a_1,b_1,\ldots, a_g,b_g, m_1,\ldots, m_n\mid m_1\cdots m_n[a_1,b_1]\cdots[a_g,b_g]=1\rangle.
$$
Then $\Gamma$ has $n$  plaquettes $p_i$ bordered by the loops $m_i$ and a plaquette $p$ bordered by $m=m_1\cdots m_n[a_1,b_1]\cdots[a_g,b_g]$.

 \textbf{Defect data:}
We assign 
\begin{compactitem}
\item to each edge $m_1,\ldots, m_n, a_1,\ldots, b_g$ a copy of the same finite group $G$,
\item to the plaquette $p_i$ a representation $\rho_i: G\sslash G\to \mathrm{Vect}_\C$  as in Corollary \ref{cor:actact},  given by 
a conjugacy class $\mathcal C_i\subset G$ and a representation $\sigma_i: N_i\to \mathrm{Aut}_\C(V_i)$  of  $N_i=\mathrm{Stab}(g_i)$ for some $g_i\in \mathcal C_i$,
\item to the plaquette $p$ the representation $\rho: G\sslash G\to \mathrm{Vect}_\C$ given by the conjugacy class $\{1\}\subset G$ and the trivial representation $\sigma: G\to \mathrm{Aut}_\C(\C)$.
\end{compactitem}
 \textbf{Gauge configurations and gauge transformations:}\\
The groupoid of reduced gauge configurations and gauge transformations is given as follows:
\begin{compactitem}
\item  A  gauge configuration is an assignment  of an element of $G$ to each loop $m_1,\ldots, m_n,a_1,\ldots, b_g$.
\item The group of  gauge transformations is $G$ and acts  by simultaneous conjugation. 
\end{compactitem}
 \textbf{Vector space of the defect surface:}\\
The  functor $F_\Sigma: \mathcal A^\Sigma\sslash \mathcal G^\Sigma\to \Vect_\C$ assigns to a gauge configuration $A$ the null vector space, unless the group elements for  $m_i$ and $m$ are all in  the assigned conjugacy classes. In this case, one has  $F_\Sigma(A)=V_1\oo \ldots\oo V_n$.

Thus, gauge configurations that contribute to $\lim F_\Sigma$ are in bijection with elements of the set
 $\mathrm{Hom}_*(\pi_1(\Sigma^*,x), G)$  of group homomorphisms $f:\pi_1(\Sigma^*,x)\to G$ with $f(m_i)\in \mathcal C_i$.
With formula  \eqref{eq:limexp} we obtain
\begin{align*}
\mathcal Z(\Sigma)=\lim F_\Sigma&=\!\!\!\!\!\!\!\!\!\!\!\!\!\bigoplus_{A\in \mathcal R(\mathrm{Hom}_*(\pi_1(\Sigma^*,x), G))}\!\!\! \!\!\!\!\!\!\!\!\!\!(V_1\oo\ldots\oo V_n)^{\mathrm{Aut}(A)}
\cong(V_1\oo \ldots\oo V_n)\oo_{\C[G]} \langle  \mathrm{Hom}_*(\pi_1(\Sigma^*,x), G)\rangle_\C,
\end{align*}
where the direct sum runs over a set of representatives of the conjugacy classes in $\mathrm{Hom}_*(\pi_1(\Sigma,x), G)$ and the actions of $\C[G]$ are induced by the conjugation action of $G$ and  by Corollary \ref{cor:actact}. This coincides with the result of \cite{BK} and \cite{KK} if one restricts their results to representation categories of finite groups.
\end{example}

In the next example we compute the vector space $\mathcal Z(\Sigma)$ in the presence of a domain wall and a domain wall excitation. Here, it is more efficient to  work with a   coherent basepoint set.

\bigskip
\begin{example}  \textbf{(Domain walls and excitations on domain walls)}

 \textbf{Stratification:}\\
Let $\Sigma$ be a surface of genus $g$, $e$ an oriented separating curve on $\Sigma$  such that $\Sigma\setminus e=\Omega_1\amalg \Omega_2$ and $v$ a vertex on $e$, as in Figure \ref{fig:separating}. 
Consider the  stratification  of $\Sigma$ with
\begin{compactitem}
\item 2-strata $t_i=\Omega_i$ with    $\hat t_i=\overline \Omega_i$,
\item 1-stratum $s=e\setminus v$ with   $\hat s= [0,1]$,
\item 0-stratum $v$ with $\hat v=\{v\}$, 
\end{compactitem}
that corresponds to the graded graph
\begin{align*}
\begin{tikzpicture}[scale=.6]
\draw[fill=black](0,0) circle (.1);
\node at (0,-.1) [anchor=north]{$v$};
\draw[fill=black](0,2) circle (.1);
\node at (0,2.1) [anchor=south]{$s$};
\draw[fill=black](1,3) circle (.1);
\node at (1.1, 3)[anchor=west]{$t_2$};
\draw[fill=black](-1,3) circle (.1);
\node at (-1.1, 3)[anchor=east]{$t_1$};
\draw[line width=1pt] (0,0).. controls (1,.5) and (1,1.5).. (0,2) node[sloped, pos=0.5, allow upside down]{\arrowOut} node[pos=.5, anchor=west]{$\iota^1_v$};
\draw[line width=1pt] (0,0).. controls (-1,.5) and (-1,1.5).. (0,2) node[sloped, pos=0.5, allow upside down]{\arrowOut} node[pos=.5, anchor=east]{$\iota^0_v$};
\draw[line width=1pt] (0,2)--(1,3) node[sloped, pos=0.5, allow upside down]{\arrowOut} node[pos=.5, anchor=north west]{$\iota_2$};
\draw[line width=1pt] (0,2)--(-1,3) node[sloped, pos=0.5, allow upside down]{\arrowOut} node[pos=.5, anchor=north east]{$\iota_1$};
\end{tikzpicture}
\end{align*}
The set $\{v\}$ defines a coherent basepoint set.  The maps $\iota^i_{v}:\hat v\to \hat s$, $v\mapsto i$   send  $v$ to the endpoints  of $[0,1]$.
The maps $\iota_i: \hat s\to \hat t_i$  from  \eqref{eq:commtop}  send both endpoints of $[0,1]$ to $v\in \overline \Omega_i$. 

 \textbf{Defect data:}\\
Orient the normal of $e$ such that it points towards $\Omega_1$ and  assign:
\begin{compactitem}
\item to the 2-strata $t_1$, $t_2$  finite groups $G_1$, $G_2$,
\item to the 1-stratum $s$  a finite $G_1\times G_2^{op}$-set $M$,
\item to the vertex $v$  a representation $\rho: M\times M\sslash G_1\times G_2\to \Vect_\C$ given by a $G_1\times G_2$-orbit $\mathcal O\subset M\times M$ and a representation $\sigma: N\to \mathrm{Aut}_\C(V)$ of a stabiliser.
\end{compactitem} 
 \textbf{Gauge configurations and gauge transformations:}\\
 The groupoid of reduced gauge configurations and gauge transformations on $\Sigma$ is given as follows:
 \begin{compactitem}
 \item a gauge configuration  by group homomorphisms $\rho_i: \pi_1(\overline \Omega_i, v)\to G_i$ and  $m_s, m_t\in M$ such that
 \begin{align}\label{eq:mtcond}
 m_t=\rho_1(e)\rhd m_s\lhd \rho_2(e)^\inv,
 \end{align}
 \item a  gauge transformation  by a pair $(g_1,g_2)\in G_1\times G_2$  acting on a gauge configuration as
 \begin{align}\label{gtrex}
 \rho_i\mapsto g_i\cdot \rho_i\cdot g_i^\inv,\qquad m_s\mapsto g_1\rhd m_s\lhd g_2^\inv, \qquad m_t\mapsto g_1\rhd m_t\lhd g_2^\inv.
 \end{align}
 \end{compactitem}
 As $m_t$ is determined by $m_s$ and $\rho_1,\rho_2$ via \eqref{eq:mtcond}, we can  omit $m_t$ and set $m=m_s$.
  
   \textbf{Vector space of the defect surface:}\\
  The functor $F_\Sigma:\mathcal A^\Sigma\sslash \mathcal G^\Sigma\to \Vect_\C$ assigns to a gauge configuration $A=(\rho_1,\rho_2, m)$ the vector space
$V$ if $m\in M_{\rho_1,\rho_2}$ and else the null vector space, where
$$M_{\rho_1,\rho_2}=\{m\in M\mid (m,\rho_1(e)\rhd m\lhd \rho_2(e)^\inv)\in \mathcal O\}.$$  
We denote by 
$\mathcal R$  a set of representatives of the orbits of the $G_1\times G_2$-action \eqref{gtrex} on the set of triples $(\rho_1,\rho_2, m)$ with $m\in M_{(\rho_1,\rho_2)}$. 
Then 
Formula
 \eqref{eq:limexp}  yields the  vector space
 \begin{align*}
 \mathcal Z(\Sigma)=\lim F_\Sigma=\!\!\!\!\bigoplus_{(\rho_1,\rho_2,m)\in \mathcal R} \!\!\!\!V^{\mathrm{Aut}(\rho_1,\rho_2, m)}.
 \end{align*}
 
 \textbf{Transparent point defect:}\\
If $\rho: M\times M\sslash G\times H\to\Vect_\C$ is the transparent representation with $\rho(m,m)=\C$ and $\rho(m,m')=0$ for $m\neq m'$, then $M_{\rho_1,\rho_2}=M^{\pi_1(e,v)}$ is the fixed point set of the $\Z=\pi_1(e, v)$-action    
 $$\rhd: \Z\times M\to M, \qquad n\rhd m=\rho_1(e)^n\rhd m\lhd \rho_2(e)^{-n}.$$ 
Denoting by $\mathcal A^*$ the the set of triples $(\rho_1,\rho_2, m)$ of  group homomorphisms $\rho_i:\pi_1(\overline \Omega_i,v)\to G_i$  and a fixed point $m\in M^{\pi_1(e,v)}$ and by 
$\sim$ the equivalence relation  given by \eqref{gtrex}, we obtain  
  \begin{align*}
 \mathcal Z(\Sigma)=\lim F_\Sigma=\langle \mathcal A^*/\sim\rangle_\C.
 \end{align*}
   \textbf{Transparent point and line defect:}\\
 For  $G_1=G_2=M=G$  the condition $\rho_1(e)\rhd m\lhd \rho_2(e)^\inv=m$ states that  $\rho_1(e), \rho_2(e)\in G$ are conjugate. By applying the gauge transformations \eqref{gtrex} one can achieve $\rho_1(e)=\rho_2(e)$ and  combine  $\rho_1, \rho_2$ into a  group homomorphism $\rho:\pi_1(\Sigma,v)\to G$. 
 Then  $\mathcal Z(\Sigma)$ is the  ground state from Example \ref{ex:groundstate}.
 
  \textbf{Impermeable line defect:}\\ If $M=\bullet$  and  $v$ carries the transparent representation, then 
  gauge configurations are given by group homomorphisms $\rho_i: \pi_1(\overline \Omega_i,v)\to G_i$ for $i=1,2$ and gauge transformations act  as in \eqref{gtrex}. This yields
 $$
 \mathcal Z(\Sigma)=\lim F_\Sigma=\langle \Hom(\pi_1(\overline \Omega_1, v), G_1)/G_1\rangle_\C\otimes \langle \Hom(\pi_1(\overline \Omega_2,v), G_2)/G_2\rangle_\C.
 $$
\end{example}

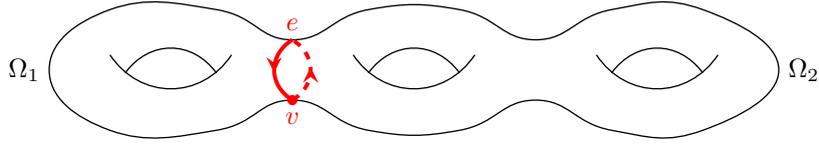
\begin{figure}
\begin{align*}
\begin{tikzpicture}[scale=.4]
\draw[line width=.5pt, color=black] plot [smooth cycle, tension=.8] coordinates 
      {(0,0) (2,2) (5.5,2) (8,1) (10.5,2) (13.5,2) (16,1)  (18.5,2) (21.5,2) (24,0)
(21.5,-2) (18.5,-2) (16,-1)(13.5,-2) (10.5,-2) (8,-1)    (5.5,-2) (2,-2)};
      \draw[line width=.5pt, ] (2,.5)..controls (3,-1) and (5,-1)  ..(6,.5);
      \draw[line width=.5pt, ] (2.5,-.1)..controls (3.3,1) and (4.7,1)  ..(5.5,-.1);
            \draw[line width=.5pt] (10,.5)..controls (11,-1) and (13,-1)  ..(14,.5);
      \draw[line width=.5pt] (10.5,-.1)..controls (11.3,1) and (12.7,1)  ..(13.5,-.1);
        \draw[line width=.5pt] (18,.5)..controls (19,-1) and (21,-1)  ..(22,.5);
      \draw[line width=.5pt] (18.5,-.1)..controls (19.3,1) and (20.7,1)  ..(21.5,-.1);
      \node at (24,0) [anchor=west]{$\Omega_2$};
       \node at (0,0) [anchor=east]{$\Omega_1$};
      \draw[ line width=1.5pt, color=red] (8,1).. controls (7.2,.5) and (7.2,-.5) ..(8,-1) node[sloped, pos=0.5, allow upside down]{\arrowOut};
      \draw[style=dashed, line width=1.5pt, color=red] (8,-1).. controls (8.8,-.5) and (8.8,.5) ..(8,1) node[sloped, pos=0.5, allow upside down]{\arrowOut};
      \node at (8,1)[color=red, anchor=south]{$e$};  
      \draw[color=red, fill=red] (8,-1) circle (.15) ;
      \node at (8,-1.1)[anchor=north, color=red]{$v$};
\end{tikzpicture}
\end{align*}
\caption{Separating curve $e$ on a surface $\Sigma$ with $\Sigma\setminus e=\Omega_1\amalg \Omega_2$ and a vertex $v$ on $e$.}
\label{fig:separating}
\end{figure}

\subsection{Defect cobordisms}
\label{subsec:defcobs}

In this section we compute examples of  the linear maps associated with defect cobordisms $\Sigma_0\xrightarrow{\iota_0} M\xleftarrow{\iota_1}\Sigma_1$. The simplest case is {a closed 3-manifold $M$} with $\Sigma_0=\Sigma_1=\emptyset$. In this case one has $\mathcal Z(\Sigma_0)=\mathcal Z(\Sigma_1)=\C$ and  the linear map $\mathcal Z(M):\C\to \C$ is a complex number that depends on the manifold $M$ as well as the defect data of the defect strata inside it.

\begin{example}  \textbf{(Closed 3-manifolds with defects)}

Let $M$ be closed stratified 3-manifold, viewed as a defect cobordism. 
Then  $\mathcal A^{\partial M}\sslash \mathcal G^{\partial M}=\mathcal A^\emptyset\sslash \mathcal G^\emptyset=\bullet$ is the terminal groupoid, and the
 functor from \eqref{eq:fsigmadef}
for the boundary   is  $F_{\partial M}=F_\emptyset=\C:\bullet \to \C$.  
  
  The  natural transformation $\mu^M:\C\Rightarrow \C$ 
  from \eqref{eq:natdef} and Proposition \ref{prop:deffunc} is  a map $\mu^M:\mathcal A^M\to \C$, $A\mapsto \mu^M_A$ that assigns to every gauge configuration a number in $\C$ and is constant on the gauge equivalence classes.
The functor $\mathcal Z:\mathrm{Cob}^{\mathrm{def}}_3\to \Vect_\C$ from Theorem \ref{cor:limtot}   averages $\mu^M$ over the gauge equivalence classes of gauge configurations on $M$.
Proposition \ref{prop:lfunctor} and the orbit stabiliser theorem yield 
\begin{align}
\mathcal Z(M)=\!\!\!\!\!\!\!\!\sum_{A\in \mathcal R(\mathcal A^M\sslash \mathcal G^M)} \! \frac{\mu^M_A}{|\mathrm{Aut}(A)|}=|\mathcal G_M|^\inv\!\!\sum_{A\in \mathcal A^M} \mu^M_A.
\end{align}
If $M$ has only 2- and 3-strata, then  $\mu^M:\A^M\to \C$ is constant, and $\mathcal Z(M)$ counts the number of gauge equivalence classes of gauge configurations on $M$, weighted by the cardinality of their automorphism groups.
\end{example}

A specific example of a closed 3-manifold with only 2- and 3-strata is one obtained by gluing two handlebodies along a genus $g$-defect surface. This yields a  simple  expression for the associated manifold invariant. 

\bigskip
\begin{example}  \textbf{(Closed 3-manifold by gluing handlebodies)}

 \textbf{Stratification:}\\
Let $M$ be a 3-manifold obtained by gluing two handlebodies $H_1,H_2$ of genus $g$ along their boundaries $\partial H_1=\partial H_2=\Sigma$. 
This defines a stratification of $M$ with
\begin{compactitem}
\item  two 3-strata $t_i=\mathring H_i$  with associated spaces $\hat t_i=H_i$,
\item a single 2-stratum $s=\Sigma$ with $\hat s=s=\Sigma$,
\item  the maps $\iota_i: \hat s\to \hat t_i$ from \eqref{eq:commtop}  given by the inclusions $\iota_i:\Sigma\to H_i$.
\end{compactitem}
 Any point $x\in \Sigma$ defines a coherent basepoint set and group homomorphisms $\pi_1(\iota_i): \pi_1(\Sigma,x)\to \pi_1(H_i,x)$.

 \textbf{Defect data:}\\
Let the normal of $\Sigma$ point towards $H_1$ and label
\begin{compactitem} 
\item  the 3-strata $\mathring H_i$  with finite groups $G_i$,
\item  the  2-stratum $\Sigma$ with a finite $G_1\times G_2^{op}$-set $N$. 
\end{compactitem}
 \textbf{Gauge configurations and gauge transformations:}\\
The groupoid of reduced gauge configurations and transformations with respect to $\{x\}$ is given as follows:
\begin{compactitem}
\item A  gauge configuration  is a triple $(\rho_1,\rho_2, n)$ of
group homomorphisms $\rho_i:\pi_1(H_i,x)\to G_i$ and  a fixed point $n\in N$ under the group action
\begin{align}\label{eq:act}
\rhd: \pi_1(\Sigma,x)\times N\to N, \qquad (\lambda, n)\mapsto \rho_1\circ \pi_1(\iota_1)(\lambda)\rhd n\lhd \rho_2\circ \pi_1(\iota_2) (\lambda)^\inv. 
\end{align}
\item A  gauge transformation is an element $g=(g_1,g_2)\in G_1\times G_2$ and acts by
\begin{align}\label{eq:gactt}
\rho_i\mapsto g_i\cdot \rho_i\cdot g_i^\inv,\qquad n\mapsto g_1\rhd n\lhd g_2^\inv.
\end{align}
\end{compactitem}
 \textbf{Linear map for the defect cobordism:}\\
As  $\pi_1(H_i,x)\cong F_{g}$, conjugacy classes of group homomorphisms $\rho_i:\pi_1(H_i,x)\to G_i$ are simply conjugacy classes  $C_i\subset G_i^{\times g}$ under simultaneous conjugation. Denoting by $N^{(\rho_1,\rho_2)}$ the fixed point set of  \eqref{eq:act}, we obtain with
  Proposition \ref{prop:lfunctor} and the orbit stabiliser theorem
\begin{align}\label{eq:statsumsurf}
\mathcal Z(M)&=\sum_{ [(\rho_1,\rho_2,n)]}   | \mathrm{Aut}(\rho_1,\rho_2,n)|^\inv
=\frac 1 {|G_1|\cdot |G_2|}\sum_{\rho_1\in G_1^{\times g}}\sum_{\rho_2\in G_2^{\times g}} |N^{(\rho_1,\rho_2)}|.
\end{align}
 \textbf{Transparent defect surface:}\\
If $G_1=G_2=N=G$ with $g_1\rhd m\lhd g_2^\inv=g_1\cdot m\cdot g_2^\inv$, then  \eqref{eq:act} states that  $\rho_1$ and $\rho_2$  induce conjugate group homomorphisms $\rho_i\circ \pi_1(\iota_i):\pi_1(\Sigma,x)\to G$ and hence define a group homomorphism $\rho: \pi_1(M,x)\to G$. Formula \eqref{eq:statsumsurf} then reduces to the formula for untwisted Dijkgraaf-Witten theory without defects
$$
\mathcal Z(M)=|G|^\inv \cdot \left|\mathrm{Hom}(\pi_1(M),G)\right|.
$$
 \textbf{Impermeable defect surface:}\\
For the terminal set $N=\bullet$ formula \eqref{eq:statsumsurf} reduces to 
$
\mathcal Z(M)= |G_1|^{g-1}\cdot|G_2|^{g-1}=\mathcal Z(H_1)\cdot\mathcal Z(H_2)
$.
\end{example}

We now consider an example of a 3-manifold $X$ with a non-trivial boundary $\partial X$, viewed as a defect cobordism $\partial X\xrightarrow{\iota} X\leftarrow \emptyset$.
 In particular, it shows how defect lines on transparent defect planes are associated with representation of the Drinfeld double $D(G)$.

\bigskip
\begin{example}\label{ex:moebius}  \textbf{(Solid torus with two  defect planes and an isolated loop)}

We consider the defect cobordism $X=D^2\times S^1$ with $\partial X=S^1\times S^1$  
obtained by gluing the top and bottom face of the following cylinder along the identity map
\begin{align*}
\begin{tikzpicture}[scale=.6, baseline=(current  bounding  box.center)]
\draw[color=black, line width=1pt] (0,0) ellipse (2cm and 1cm);
\draw[color=red, line width=1pt] (-2,4)--(2,4);
\draw[color=red, line width=1pt] (-2,0)--(2,0);
\draw[color=red, draw=none,  fill=blue, fill opacity=.2] (0,4)--(2,4)--(2,0)--(0,0)--(0,4);
\draw[color=red, draw=none,  fill=green, fill opacity=.2] (2,4)--(-2,4)--(-2,0)--(2,0)--(2,4);
\node at (-1.2,4)[anchor=south, color=red]{$s_2$};
\node at (1.2,4)[anchor=south, color=red]{$s_1$};
\draw[color=red,->,>=stealth, line width=1pt] (1,4)--(.4, 3.6);
\draw[color=red,->,>=stealth, line width=1pt] (-1,4)--(-.4, 4.4);
\draw[color=blue, line width=2pt] (0,4)--(0,0)  node[sloped, pos=0.5, allow upside down]{\arrowOut}  node[pos=.5, anchor=west]{$r$};  
\draw[color=red, line width=1pt] (-2,4)--(-2,0);
\draw[color=red, line width=1pt] (2,4)--(2,0);
\node at (0, -1.5){$t_1$};
\node at (0, 5.5){$t_2$};
\draw[color=black, line width=1pt] (0,4) ellipse (2cm and 1cm);
\draw[color=black, fill=black] (0,4) circle (.15) node[anchor=south]{$q_3$};
\draw[color=black, fill=black] (0,0) circle (.15) node[anchor=north]{$q_3$};
\draw[color=red, fill=red] (2,4) circle (.15) node[anchor=south west]{$q_1$};
\draw[color=red, fill=red] (2,0) circle (.15) node[anchor=north west]{$q_1$};
\draw[color=red, fill=red] (-2,4) circle (.15) node[anchor=south east]{$q_2$};
\draw[color=red, fill=red] (-2,0) circle (.15) node[anchor=north east]{$q_2$};
\end{tikzpicture}
\end{align*}
 \textbf{Stratification:} The stratification has a single  1-stratum $r$, two 2-strata $s_1,s_2$ and two 3-strata $t_1,t_2$. Its boundary 
1-strata are $u_1=s_1\cap \partial X$, $u_2=s_2\cap \partial X$ and its boundary 2-strata $v_1=t_1\cap \partial X$, $v_2=t_2\cap \partial X$. 
The spaces associated to the strata and boundary strata of $X$ are given by
\begin{compactitem}
\item $r=\hat r=S^1$ is a circle, 
\item $\hat s_1=\hat s_2=[0,1]\times S^1$ are annuli,
\item $\hat t_1=\hat t_2=D^2_+\times S^1$, where $D_+^2$ is a closed half-disc,
\item $u_1=\hat u_1=u_2=\hat u_2=S^1$ are circles,
\item $\hat v_1=\hat v_2=[0,1]\times S^1$are annuli.
\end{compactitem}
The point $q_3\in r$ and  $q_1,q_2\in \partial X$ define a coherent basepoint set for $X$.

 \textbf{Classical defect data:} The strata of $X$ are labelled as follows
\begin{compactitem}
\item the 3-strata $t_1,t_2$ with finite groups $G_1$, $G_2$,
\item  the 2-strata $s_1$, $s_2$ with finite $G_1\times G_2^{op}$-sets $M_1$, $M_2$,
\item  the 1-stratum $r$ with a representation $\phi: M_1\times M_2\sslash G_1\times G_2\to \Vect_\C$.
\end{compactitem}

 \textbf{Gauge configurations and gauge transformations:}\\
A gauge configuration for the coherent basepoint set consists of
\begin{compactitem}
\item functors $\rho_j:\Pi_1(\hat t_j,\{q_1,q_2,q_3\})\to \bullet\sslash G_j$ for $j=1,2$,
\item elements $m_1\in M_1$, $m_2\in M_2$ for $q_1,q_2$ and $m'_1\in M_1$, $m'_2\in M_2$ for $q_3$ that satisfy 
\begin{align*}
&\rho_1(\sigma_1)\rhd m_1\lhd \rho_2(\sigma_1)^\inv=m_1, & &\rho_2(\sigma_2)\rhd m_2\lhd \rho_1(\sigma_2)^\inv=m_2,\\ 
&\rho_1(\tau_1)\rhd m_1\lhd \rho_2(\tau_1)^\inv=m'_1, & &\rho_2(\tau_2)\rhd m_2\lhd\rho_1(\tau_2)^\inv=m'_2,
\end{align*}
for all paths  $\sigma_j: q_j\to q_j$ and $\tau_j: q_j\to q_3$  in $\hat s_j$ for $j=1,2$.
\end{compactitem}
Gauge transformations are elements $(h_1^1, h_1^2, h^3_1, h_2^1, h_2^2,h^3_2)\in G^{\times 3}_1 \times G^{\times 3}_2$, where $h^j_1\in G_1$ and $h^j_2\in G_2$ are associated to $q_j$ for $j=1,2,3$. They  act by $\rho_i(\nu)\mapsto h_i^k\cdot \rho_i(\nu)\cdot(h_i^l)^\inv$ for all paths $\nu: q_l\to q_k$ in $\hat t_i$ and by
\begin{align*}
&m_1\mapsto h_1^1\rhd m_1\lhd (h_2^1)^\inv,\quad m_2\mapsto h_2^2\rhd m_2\lhd (h_1^2)^\inv,\quad m'_1\mapsto h_1^3\rhd m'_1\lhd (h_2^3)^\inv,\quad m'_2\mapsto h_2^3\rhd m'_2\lhd (h_1^3)^\inv.
\end{align*}
 \textbf{Reduced gauge groupoid:}\\
We fix paths $\tau_j:q_j\to q_3$  in $\hat s_j$ for $j=1,2$.  By applying gauge transformations $(h_1^1,1, 1, h_2^1,1, 1)$ at $q_1$ and $(1, h_1^2, 1,1, h_2^2, 1)$ at $q_2$ we can achieve $\rho_1(\tau_j)=\rho_2(\tau_j)=1$ for $j=1,2$, which implies $m'_1=m_1$ and $m'_2=m_2$. 

This yields a reduced gauge groupoid  $\A^X\sslash \G^X=\A^{\partial X}\sslash G^{\partial X}$, where the projection functor on the boundary gauge groupoid is the identity  $P_\partial=\id:\A^X\sslash \G^X\to \A^{\partial X}\sslash \G^{\partial X}$, and whose 
 \begin{compactitem}
 \item gauge configurations are quadruples $(g_1,g_2, m_1,m_2)$ with $g_j=\rho_j(\sigma_1)\in G_j$, $m_j\in M_j$ such that 
 \begin{align}\label{eq:redgcfg}
 g_1\rhd m_1\lhd g_2^\inv=m_1\qquad g_2\rhd m_2\lhd g_1^\inv=m_2,
 \end{align}
 \item gauge transformations are elements $(h_1,h_2)\in G_1\times G_2$ and act by
 \begin{align}\label{eq:redgtrafo}
 m_1\mapsto h_1\rhd m_1\lhd h_2^\inv,\qquad m_2\mapsto h_2\rhd m_2\lhd h_1^\inv,\qquad g_1\mapsto h_1 g_1 h_1^\inv,\qquad g_2\mapsto h_2g_2h_2^\inv.
 \end{align}
 \end{compactitem}

 \textbf{Boundary vector space:} \\
The functor $F_{\partial X}:\A^{\partial X}\sslash G^{\partial X}\to \Vect_\C$ from \eqref{eq:fsigmadef} is the constant functor $F_{\partial X}=\C$ that assigns to every boundary gauge configuration the vector space $\C$ and to every boundary gauge transformation  $\id_\C$.
The defect TQFT assigns to $\partial X$ the free complex vector space generated by  the path-components of  $\A^{\partial X}\sslash G^{\partial X}$ 
\begin{align*}
&\mathcal Z(\partial X)=\lim F_{\partial X}\\
&=\langle\, \{(g_1,g_2,m_1,m_2)\in G_1\times G_2\times M_1\times M_2\mid g_1\rhd m_1\lhd g^\inv_2=m_1, g_2\rhd m_2\lhd g^\inv_1=m_2 \}\,/\,G_1\times G_2\,\rangle_\C.
\end{align*}

 \textbf{Quantum defect data:}\\
We interpret $X$ as a cobordism $\partial X\xrightarrow{\iota} X\leftarrow \emptyset$. The 
natural transformation $\mu^X: F_{\partial X}=\C\Rightarrow \C$ from \eqref{eq:natdef}   is given by the representation $\phi: M_1\times M_2\sslash G_1\times G_2\to \Vect_\C$ for the isolated loop $r$ 
and has components 
 $$\mu^X_{(g_1,g_2,m_1,m_2)}=\tr( 
  \phi_{m_1,m_2}(g_1,g_2))\in \C,$$
 where 
 $\phi_{m_1,m_2}(g_1,g_2):\phi(m_1,m_2)\to \phi(m_1,m_2)$ 
 is the linear endomorphism assigned to the automorphism  $(g_1,g_2): (m_1,m_2)\to (m_1,m_2)$   in $M_1\times M_2\sslash G_1\times G_2$.
 The linear map $\mathcal Z(X): \mathcal Z(\partial X)\to \C$ for the cobordism $\partial X\xrightarrow{\iota} X\leftarrow \emptyset$ is the image of $\mu^X: \C\Rightarrow \C$ under the functor $L$ from Proposition \ref{prop:lfunctor} and given by
$$
\mathcal Z(X): \mathcal Z(\partial X)\to \C, \quad   [(g_1,g_2,m_1,m_2)]\mapsto   \mu^X_{(g_1,g_2,m_1,m_2)}.$$
 \textbf{Transparent defect line:}\\
If we choose  $M_1=M_2=M$ and the transparent representation $\phi: M\times M\sslash G_1\times G_2\to \Vect_\C$ with $\phi(m_1,m_2)=\delta_{m_1}(m_2) \C$, then we have $\mu^X_{(g_1,g_2, m_1, m_2)}=\delta_{m_1}(m_2)$ and
$$
\mathcal Z(X): \mathcal Z(\partial X)\to \C, \quad   [(g_1,g_2,m_1,m_2)]\mapsto  \delta_{m_1}(m_2).$$
 \textbf{Transparent defect surfaces:}\\
If we choose two transparent defect surfaces,  we have
 $G_1=G_2=M_1=M_2=G$.
 Condition \eqref{eq:redgcfg} states that 
 $$
 g_2=m_2\cdot g_1\cdot m_2^\inv\qquad g_1=(m_1m_2)\cdot g_1\cdot (m_1m_2)^\inv.
 $$
 By applying a gauge transformation $h_2\in G$, we can achieve $m_2=1$ and $g_2=g_1$. 
 By setting $a:=m_1\in G$ and $b:=g_1\in G$ we transform the second condition into $[a,b]=1$ and obtain the reduced gauge groupoid
$$\A^X\sslash G^X=\A^{\partial X}\sslash G^{\partial X}=\Hom(\Z\times \Z, G)/G.$$
 The  representation  $\phi: G\times G\sslash G\times G\to \Vect_\C$ is then given by a representation $\phi': G\sslash G\to \Vect_\C$  of the action groupoid for the conjugation action of $G$ on itself.  By Example \ref{ex:actorbb} this corresponds to a representation of the Drinfeld double $D(G)$. If we choose an irreducible representation as in Corollary \ref{cor:actact} and Example \ref{ex:actorbb}  given by a conjugacy class $\mathcal C\subset G$ and a representation $\sigma: N\to \Aut_\C(V)$ we obtain
$$
\mathcal Z(X): \mathcal Z(\partial X)\to \C, \quad   [(a,b)]\mapsto  \delta_{\mathcal C}(a)\, \tr \rho_\sigma(b).$$
In particular, a transparent defect line $r$ corresponds to $\mathcal C=\{1\}$ and the trivial representation $\sigma$. In this case, the linear map associated to $X$ coincides  with the one of  Dijkgraaf-Witten theory for $G$ without defects
$$
\mathcal Z(X)=\mathcal Z_{DW}: \mathcal Z(\partial X)\to \C, \quad   [(a,b)]\mapsto  \delta_{e}(a).$$
 \textbf{Impermeable defect surfaces:}\\
If we choose sets $M_1,M_2$ with the trivial $G_1\times G_2$-actions,  a representation   $\phi: M_1\times M_2\sslash G_1\times G_2\to \Vect_\C$ is simply an assignment of a complex vector space $\phi(m_1,m_2)$ to each element $(m_1,m_2)\in M_1\times M_2$. Condition \eqref{eq:redgcfg} is trivially satisfied and  the boundary vector space is given by the 
sets  $C_{G_i}$ of conjugacy classes in $G_i$
$$
\mathcal Z(\partial X)=\langle C_{G_1}\times C_{G_2}\times M_1\times M_2\rangle_\C.
$$
The linear map for the cobordism becomes
$$
\mathcal Z(X): \mathcal Z(\partial X)\to \C, \quad   (C_{g_1},C_{g_2},m_1,m_2)\mapsto \dim_\C \phi(m_1,m_2).
$$
\end{example}

\appendix
\section{Appendix}
\label{sec:appendix}

\subsection{Kan extensions along 
 discrete opfibrations }
 
For a graph $Q$ we denote by $\maq=L(Q)$ the  free category generated by $Q$ and for a graph map $f:Q\to Q'$ by $F=L(f):\maq\to\maq'$ the induced functor between free categories, as in Section \ref{subsec:graphs}.

Recall that for
an insertion $f:Q\to Q'$  the associated functor $F:\maq\to\maq'$ is a discrete opfibration by Lemma \ref{lem:I-faithfulandInsertive}. 
The left Kan extension $\Lan_F B:\maq'\to\mac$ of a functor $B:\maq\to \mac$ into a  cocomplete category $\mac$  is then given by  Perrone and Tholen \cite[Corollary 5.8]{PT}. We specialise \cite[Corollary 5.8]{PT},  which holds more generally for Kan extensions along 
opfibrations between small categories, to the case at hand and give a direct proof for the convenience of the reader.

\begin{lemma}\label{lem:kanweakinsertion} Let $Q,Q'$ be graded
graphs, $f:Q\to Q'$ an insertion and $\mac$ a cocomplete category. 

The left Kan extension of a functor $B:\maq\to \mac$ along $F:\maq\to\maq'$ is given on  $v' \in \Ob \maq'$ by 
\begin{align*}
\Lan_F B(v')=\textstyle \coprod_{v\in f^\inv(v')} B(v),
\end{align*}
and the  components of the natural transformation $\eta: B \Rightarrow (\Lan_F B) F$ are 
the inclusions for the coproduct. 
On morphisms $q':v'\to w'$ of $\maq'$  
the functor $\Lan_F B$ is given by the commuting diagram
\begin{align*}
\xymatrix@C=60pt{\coprod_{v\in f^\inv(v')} B(v) \ar[r]^{\Lan_FB(q')} & \coprod_{w\in f^\inv(w')} B(w)\\
B(v) \ar[u]^{\eta_{v}}\ar[r]_{B(q_{v})}& B(w), \ar[u]_{\eta_{w}}
}
\end{align*}
where    $q_{v}: v\to w$ is the unique morphism starting in $v\in f^\inv(v')$ with $F(q_{v})=q'$.
\end{lemma}

\begin{proof} 
To show that $(\Lan_F B, \eta)$ has the universal property of the left Kan extension, let $G:\maq'\to \mac$ be a functor and $\gamma: B\Rightarrow GF$ a natural transformation. Any natural transformation $\delta: \Lan_F B\Rightarrow G$ with $(\delta F)\circ \eta=\gamma$ is determined uniquely by that condition, as $\eta$ is given by the inclusions for the coproduct. To show existence, it is sufficient to verify the naturality condition, which holds, because the following diagram commutes for any morphism $q':v'\to w'$
$$
\xymatrix{
\coprod_{v\in f^\inv(v')}  B(v) \ar[rrr]^{\Lan_F B(q')} \ar[dd]_{\delta_{v'}}
& & &  \coprod_{w\in f^\inv(w')}  B(w)  \ar[dd]^{\delta_{w'}}\\
& B(v) \ar[lu]^{\eta_v}  \ar[r]^{B(q_v)}  \ar[ld]_{\gamma_v}   & B(w) \ar[ru]_{\eta_w} \ar[rd]^{\gamma_w}\\
G(v')=GF(v) \ar[rrr]_{G(q')} & & & GF(w)=G(w').
}
$$
The upper quadrilateral commutes by definition of $\Lan_F B$ on the morphisms,  the lower one by naturality of $\gamma$ and the two triangles due to the condition $(\delta F)\circ \eta=\gamma$.
\end{proof}

\subsection{ Results used in the proof of Theorem \ref{th:classcobfunc}}

In this section, we collect some results on ends and limits of functors into a bicomplete closed symmetric monoidal category $\V$ and a result on pushouts of such functors and natural transformations for the proof of Theorem \ref{th:classcobfunc}. The main case of interest is $\V=\Grpd$.
Throughout this section, let  $\maq$ be a small category and $\V$ a bicomplete closed symmetric  monoidal category. Note that this implies that $\CAT(\maq,\V)$ is bicomplete.

\begin{lemma}\label{lem:compat} For all functors $D: \maq \to \V$ the functor
$\HOM_\V( - ,D): \CAT(\maq ,\V)^{op} \to \CAT(\maq^\op \times \maq ,\V)$
preserves limits. 
\end{lemma}

\begin{proof}
This is a direct consequence of the fact that   $\HOM_V(-,v): \V^{op}\to \V$  preserves limits for all $v\in\Ob\V$.
As $\V$ is bicomplete and $\maq$ small, the categories $\CAT(\maq,\V)$ and $\CAT(\maq^{op}\times\maq,\V)\cong \CAT(\maq^{op},\CAT(\maq, \V))$ are bicomplete as well, and their limits are pointwise.
\end{proof}

\begin{lemma}\label{lem:enriched Kan prop}
Let $F: \maq' \to \maq$,   $T': \maq' \to\V$ and $D\colon \maq \to\V$ be functors. There is an isomorphism
$$\int_{x' \in \maq'}\HOM_\V\big (T'(x'), DF(x') \big)\cong\int_{x \in \maq}\HOM_\V\big (\Lan_F T'(x), D(x) \big),$$
which is natural in $T'$ and in $D$.
\end{lemma}

\begin{proof} For any object $v$ in $\V$, we  have  the following bijections, which are natural in $v$:
\begin{align*}
&\hom_\V\left(v, \int_{x' \in \maq'}\HOM_\V\big (T'(x'), DF(x') \big)\right)\cong 
 \int_{x' \in \maq'}\hom_\V\left(v,\HOM_\V\big (T'(x'),DF(x') \big)\right) \\
 &\cong \int_{x' \in \maq'}\hom_\V\big(v\otimes_\V T'(x'), DF(x') \big)
 \cong \int_{x' \in \maq'}\hom_\V\left( T'(x'),\HOM_\V\big(v, DF(x')\big) \right)
  \\
 &= \int_{x' \in \maq'}\hom_\V\left( T'(x'), F^* \HOM_\V\big(v,  D(-)\big) (x') \right)
 \cong \int_{x \in \maq}\hom_\V\left( \Lan_F T'(x),\HOM_\V\big(v,  D(x)\big) \right)
 \\
 &\cong \int_{x \in \maq}\hom_\V\left( v,\HOM_\V\big( \Lan_F T'(x),  D(x)\big) \right)
 \cong\hom_\V\left ( v ,\int_{x \in \maq}\HOM_\V\big( \Lan_F T'(x),  D(x)\big) \right).
\end{align*}
Here, the first identity follows from the fact that $\hom_\V(v,-)\colon \V \to \Set $ preserves limits, the second and third from the fact that $\V$ is symmetric monoidal closed. The equality in the fourth step is the definition of $F^*$. The fifth step follows from the universal property of left Kan extension, together with the fact that for all functors $M,L\colon \maq \to \mac$  the set of natural transformations $M\Rightarrow L$ is in bijection with $\int_{x \in \maq} \hom_\mac(M(x),L(x))$. Finally we reverse the initial three steps. The claim then follows by the Yoneda lemma.
\end{proof}

\begin{remark}
 If $\V=\Grpd$, then $\HOM_\V(A,B)$ is the groupoid of functors $A \to B$, and natural transformations between them. In this case  we  have 
$$\int_{x' \in \maq'}\GRPd\big (T'(x'),F^* D(x') \big)\cong \mathrm{Nat}(T',F^* D) \cong \mathrm{Nat}(\Lan_F T', D) \cong \int_{x \in \maq} \GRPd\big (\Lan_F T'(x), D(x) \big),$$
and these bijections preserve the composition of natural transformations.
\end{remark}

We  now consider  the following diagrams of categories, functors and natural transformations,  where $\maq_j$ are cocomplete small categories and $\V$ a bicomplete closed symmetric monoidal category
\begin{equation}\label{diags:TandD}
\begin{gathered}\hskip-1cm
\xymatrix{&\mathcal{Q}_0\ar@/_2pc/[ddrr]_{T_0}
\xtwocell[ddrr]{}<>{  ^<3> \tau_0 }
\ar[dr]^{F_0}&& \mathcal{Q}_1\ar[dl]_{F_1}
\xtwocell[dd]{}<>{  <3> \tau_1 }
\xtwocell[dd]{}<>{  ^<-3> {\tau_1'} }
\ar[dr]^{F_1'} \ar[dd]|{T_1} &&\mathcal{Q}_2\ar[dl]_{F_2}\ar@/^2pc/[ddll]^{T_2}
\xtwocell[ddll]{}<>{  <-3> \tau_2 } \\
&&\ar[dr]_{T_{01}} \mathcal{Q}_{01} && \mathcal{Q}_{12}\ar[dl]^{{T_{12}}}\\
&&&\V
} \xymatrix{&\mathcal{Q}_0\ar[dr]^{F_0}  \ar@/_2pc/[ddrr]_{D_0}
&& \mathcal{Q}_1\ar[dl]_{F_1}
\ar[dr]^{F_1'} \ar[dd]|{D_1} &&\mathcal{Q}_2\ar[dl]_{F_2}   \ar@/^2pc/[ddll]^{D_2}\\
&&\ar[dr]_{D_{01}} \mathcal{Q}_{01} && \mathcal{Q}_{12}\ar[dl]^{{D_{12}}}\\
&&&\V.
}\end{gathered}
\end{equation}
The diagram on the right commutes. The arrows in the one on the left are related by natural transformations, as indicated in the diagram.
Let $\mathcal Q_{012}=\mathcal Q_{01} \amalg_{\mathcal Q_1} \mathcal Q_{12}$ be the pushout of  $F_1: \mathcal Q_1\to \mathcal Q_{01}$ and $F'_1:\mathcal Q_1\to\mathcal Q_{12}$ and consider the associated  commuting diagram 
\begin{equation}\label{eq:pushq}
\begin{gathered}
\xymatrix{ &\mathcal Q_0\ar[dr]^{F_0} \ar@/_2pc/[ddrr]_{J_0} & & \mathcal Q_1\ar[dl]_{F_1} \ar[dr]^{F'_1} \ar[dd]^{J_1} && \mathcal Q_2 \ar[dl]_{F_2}  \ar@/^2pc/[ddll]^{J_2} \\ && \mathcal Q_{01}\ar[dr]_{J_{01}} && \mathcal Q_{12} \ar[dl]^{J_{12}} \\
&&& \mathcal Q_{012},}
\end{gathered}
\end{equation}
in which all new arrows result either from the pushout  or are defined by the commutativity  of the diagram. As the diagram on the right in \eqref{diags:TandD} commutes, the universal property of the pushout induces a functor 
\begin{align}\label{eq:d012def}
D_{012}:\mathcal Q_{012}\to \V\qquad\text{with}\qquad D_{012} J_{01}=D_{01}, \quad D_{012}J_{12}=D_{12}.  
\end{align}

\begin{theorem}\label{thm:main-end} We have the following commuting diagram in $\V$, where the  middle diamond is a pullback
\begin{equation}\label{eq:thappdiagram}\begin{gathered}
\xymatrix{&\End_{\maq_0}[T_0,D_0]  & \End_{\maq_1}[T_1,D_1] &  \End_{\maq_2}[T_2,D_2]\\
& \End_{\maq_{01}}[T_{01}, D_{01}] \ar@{->}[ru]_{\hat\T_1}  \ar@{->}[u]^{ \hat \T_0} & & \End_{\maq_{12}}[T_{12}, D_{12}] \ar@{->}[u]_{\hat \T_2}  \ar@{->}[lu]^{\hat \T'_1}\\
& &\End_{\maq_{012}}[T_{012}, D_{012}] \ar@{->}[ru]_{\hat I_{12}}  \ar@{->}[lu]^{ \hat I_{01}}
}\end{gathered}
\end{equation}
$$
T_{012}=\Lan_{J_{01}} T_{01}\!\!\!\!\coprod_{\Lan_{J_1} T_1} \!\!\!\!\Lan_{J_{12}} T_{12}: \mathcal Q_{012}\to \V,
$$
and we abbreviate  $[-,-]=\HOM_\V(-,-)$ and  $\End_\maq [T,D]:=\int_{v\in \maq} \HOM_\V(T(v), D(v))$. 

\end{theorem}
\begin{proof}
1.~Via the universal property of  $\Lan_{F_0} T_0: \mathcal Q_{01}\to \V$ the 
 natural transformation $\tau_0: T_0\Rightarrow T_{01}F_0$ induces a natural transformation $\hat \tau_0: \Lan_{F_0} T_0\Rightarrow T_{01}$.  Taking its image under  $\Lan_{J_{01}}: \CAT(\mathcal Q_{01},\V)\to \CAT(\mathcal Q_{012},\V)$ and pre-composing with the natural isomorphism $\Lan_{J_{01} F_0} T_0 \cong \Lan_{J_{01}} (\Lan_{F_0} T_0)$ induced by the universal property of the Kan extension
 yields a natural transformation
 \begin{align}\label{eq:t0def} 
\T_0: \Lan_{J_0} T_0=\Lan_{J_{01} F_0} T_0 \cong \Lan_{J_{01}} (\Lan_{F_0} T_0) \xRightarrow{\Lan_{J_{01}}\hat \tau_0 }   \Lan_{J_{01}} T_{01}.
\end{align} 
We analogously define natural transformations
\begin{align}\label{eq:totherdef} 
&\T_1: \Lan_{J_1}T_1 \Rightarrow \Lan_{J_{01}}T_{01}, & 
&\T_1': \Lan_{J_1}T_1 \Rightarrow \Lan_{J_{12}}T_{12}, &
&\T_2: \Lan_{J_2}T_2 \Rightarrow \Lan_{J_{12}}T_{12}.
\end{align}
The pushout  of $\T_1: \Lan_{J_1} T_1\Rightarrow \Lan_{J_{01}} T_{01}$ and $\T'_1: \Lan_{J_1} T_1\Rightarrow \Lan_{J_{12}} T_{12}$ in $\CAT(\mathcal Q_{012},\V)$ yields a functor
\begin{align}\label{eq:deft012}
T_{012}=\Lan_{J_{01}} T_{01}\!\!\!\!\coprod_{\Lan_{J_1} T_1} \!\!\!\!\Lan_{J_{12}} T_{12}: \mathcal Q_{012}\to \V
\end{align}
and the following commuting diagram  in $\CAT(\mathcal Q_{012},\V)$, where $I_{01}$ and $I_{12}$ arise from the  pushout \eqref{eq:deft012}
\begin{equation}\label{eq:tpushout}
\begin{gathered}
\xymatrix{ \Lan_{J_0} T_0 \ar@{=>}[d]_{\T_0}  & \Lan_{J_1} T_1 \ar@{=>}[ld]_{\T_1} \ar@{=>}[rd]^{\T'_1}  & \Lan_{J_2} T_2 \ar@{=>}[d]^{\T_2}\\
 \Lan_{J_{01}}T_{01} \ar@{=>}[rd]_{I_{01}}& & \Lan_{J_{12}}T_{12} \ar@{=>}[ld]^{I_{12}}\\
 & T_{012}.
}
\end{gathered}
\end{equation}
2.~By Lemma \ref{lem:compat} the functor 
$\HOM_\V (-, D_{012}): \CAT(\maq_{012},\V)^{op} \to \CAT(\maq_{012}^{op}\times \maq_{012},\V)$ 
for the functor $D_{012}:\mathcal Q_{012}\to \V$ from \eqref{eq:d012def} sends pushouts in $\CAT(\maq_{012},\V)$ to pullbacks in $\CAT(\maq^{op}_{012}\times \maq_{012},\V)$. 
Applying it to the commuting diagram \eqref{eq:tpushout}  yields the  commuting diagram 
\begin{equation}\label{eq:tdpushout}
\begin{gathered}
\xymatrix{ [\Lan_{J_0} T_0, D_{012}]  &  [\Lan_{J_1} T_1, D_{012}]  &  [ \Lan_{J_2} T_2, D_{012}] \\
 [ \Lan_{J_{01}}T_{01}, D_{012}]  \ar@{=>}[u]^{ [\T_0, D_{012}]}  \ar@{=>}[ru]_{\qquad [\T_1, D_{012}]} & & [\Lan_{J_{12}}T_{12}, D_{012}]  \ar@{=>}[lu]^{ [\T'_{1}, D_{012}]\qquad}  \ar@{=>}[u]_{ [\T_2, D_{012}]} \\ 
 & [T_{012}, D_{012}], \ar@{=>}[ru]_{\qquad [I_{12}, D_{012}]}  \ar@{=>}[lu]^{ [I_{01}, D_{012}]\qquad}
 }
 \end{gathered}
\end{equation}
in $\CAT(\mathcal Q_{012}^{op}\times\mathcal Q_{012}, \V)$, where the middle diamond is a pullback and  $[F,G]:=\HOM_\V  (F, G): \mathcal Q^\op \times \mathcal Q \to \V$ for functors
$F,G: \mathcal Q\to \V$. 
Because the end functor $\End_{\maq_{012}}\colon \CAT( \mathcal Q_{012}^\op \times \mathcal Q_{012},\V) \to \V$ preserves  limits, applying it to \eqref{eq:tdpushout} yields the 
following commuting diagram, where the middle diamond is a pullback
\begin{equation}\label{diag:almost}
\begin{gathered}
\xymatrix{ \End_{\maq_{012}} [\Lan_{J_0} T_0, D_{012}]  &  \End_{\maq_{012}} [\Lan_{J_1} T_1, D_{012}]  &  \End_{\maq_{012}} [ \Lan_{J_2} T_2, D_{012}] \\
 \End_{\maq_{012}}[ \Lan_{J_{01}}T_{01}, D_{012}]  \ar@{->}[u]^{ \End_{\maq_{012}}[\T_0, D_{012}]}  \ar@{->}[ru]_{\qquad \End_{\maq_{012}} [\T_1, D_{012}]} & & \End_{\maq_{012}} [\Lan_{J_{12}}T_{12}, D_{012}]  \ar@{->}[lu]^{ \End_{\maq_{012}} [\T'_{1}, D_{012}]\qquad}  \ar@{->}[u]_{ \End_{\maq_{012}} [\T_2, D_{012}]} \\ 
 & \End_{\maq_{012}}[T_{012}, D_{012}]. \ar@{->}[ru]_{\qquad \End_{\maq_{012}} [I_{12}, D_{012}]}  \ar@{->}[lu]^{\End_{\maq_{012}} [I_{01}, D_{012}]\qquad}
 }
 \end{gathered}
\end{equation}
3.~We apply the natural isomorphism from Lemma \ref{lem:enriched Kan prop} to all entries in diagram \eqref{diag:almost}  except the bottom one. 
For the top left entry of \eqref{diag:almost}   this yields  the following isomorphism in $\V$
\begin{align*} \End_{\maq_{012}}[\Lan_{J_0} T_0, D_{012}] \xrightarrow{\cong} & \End_{\maq_{0}}[T_0,  D_{012} {J_0}] =\End_{\maq_{0}}[T_0,D_{0} ]
\end{align*}
where we used in the last step that $D_{012}J_0=D_0$ by \eqref{eq:d012def}. For the other entries, this is analogous. 
Composing the morphisms on the arrows with these isomorphisms then
 gives  the diagram in the  theorem.
\end{proof}

\section*{Acknowledgments} JFM is funded by EPSRC, via the Programme
Grant EP/W007509/1: Combinatorial Representation Theory: Discovering the Interfaces of Algebra with Geometry and Topology. JFM would like to express his gratitude to
 Tim Porter for discussions on monoidal model categories, fibred categories, categories of fibrant spans and homotopical aspects of defect-TQFTs,  to Fiona Torzewska for discussions on cobordism-like categories arising from cofibrant cospans of spaces and the related TQFTs, to Nicola Gambino and Marcelo Fiore for correspondence on enriched ends and coends, to Jon Woolf for discussions on stratified spaces, to Jeffrey Morton for a series of talks on formulations of extended TQFTs using  spans and cospans, and finally to Ronnie Brown for introducing him to the concept of groupoid fibrations. 

CM is grateful to  the Algebra, Geometry and Integrable Systems Group at The University of Leeds  for hosting her and funding her stay, which  was essential for this project,  with the Programme
Grant EP/W007509/1. She is also grateful to the School of Mathematics of The University of Edinburgh, in particular to Bernd Schroers,
 for hospitality and  funding her  stay, during which further work on the project was undertaken. She  thanks Claudia Scheimbauer, Fiona Torzewska and and Tashi Walde for discussions on (co)spans, \"Od\"ul Tetik for discussions on stratifications, Vincentas Mulevi\v{c}ius for discussions on defects in Dijkgraaf-Witten theory.

Both authors are grateful to Maximilian Ludwig and Max-Niklas Steffen for pointing out small mistakes. We are also very  grateful to the referees of the article for their helpful comments and suggestions for improvement.

\textbf{Open access statement:} For the purpose of open access, the authors
 have applied a Creative
 Commons Attribution (CC BY) licence to any Author Accepted Manuscript
 version arising from this
 submission.\\
\textbf{Data Access Statement:} No data was created while producing this publication.\\
\textbf{Authorship:}  All ideas were developed
by both authors in collaboration and discussion.

\end{document}